\documentclass[a4paper,leqno,final]{amsart}
\pdfoutput=1

\usepackage{ifdraft}
\ifdraft{\usepackage[draft]{showkeys}}{\usepackage[final]{showkeys}}

\usepackage[usenames,dvipsnames]{xcolor} 
\usepackage{todonotes}

\usepackage[T1]{fontenc}
\usepackage[utf8]{inputenc}

\usepackage[final]{microtype}

\usepackage{mathtools}
\mathtoolsset{mathic}

\usepackage{amsmath,amsfonts,mathscinet,eucal,amssymb,amsthm}
\usepackage{bm,extarrows,upgreek,tensor,mathrsfs,textcomp,comment,bbm,braket}

\usepackage[shortlabels]{enumitem}

\usepackage[bbgreekl]{mathbbol}

\usepackage{tikz}
\usetikzlibrary{patterns,matrix,through,arrows,decorations.pathreplacing,decorations.markings,shadows,shapes.geometric,positioning,calc,backgrounds,fit}

\tikzstyle directed=[postaction={decorate,decoration={markings,
    mark=at position #1 with {\arrow{>}}}}]
\tikzstyle rdirected=[postaction={decorate,decoration={markings,
    mark=at position #1 with {\arrow{<}}}}]

\tikzset{anchorbase/.style={baseline={([yshift=-0.5ex]current bounding box.center)}},
arrowinthemiddle/.style={postaction=decorate,decoration={markings,mark=at position 0.5 with {\arrow{>}}}},
arrowinthemiddlerev/.style={postaction=decorate,decoration={markings,mark=at position 0.5 with {\arrow{<}}}},
cross line/.style={preaction={draw=white,line width=4pt,-}},
int/.style={very thick},
zero/.style={thin,dotted},
uno/.style={thin},
aboverotated/.style={above,rotate=60,anchor=west},
belowrotated/.style={below,rotate=60,anchor=east}}
\newcommand{\midarrow}{node[midway,sloped,allow upside down] {\hspace{0.05cm}\tikz[baseline=0] \draw[->] (0,0) -- +(.001,0);}}
\newcommand{\midarrowrev}{node[midway,sloped,allow upside down] {\hspace{0.05cm}\tikz[baseline=0] \draw[-<] (0,0) -- +(.001,0);}}
\newcommand{\nendarrow}{node[near end,sloped,allow upside down] {\hspace{0.05cm}\tikz[baseline=0] \draw[->] (0,0) -- +(.001,0);}}

\newcommand{\nendarrowrev}{node[near end,sloped,allow upside down] {\hspace{0.05cm}\tikz[baseline=0] \draw[-<] (0,0) -- +(.001,0);}}
\newcommand{\startarrow}{node[at start,sloped,allow upside down] {\hspace{0.05cm}\tikz[baseline=0] \draw[->] (0,0) -- +(.001,0);}}
\newcommand{\startarrowrev}{node[at start,sloped,allow upside down] {\hspace{0.05cm}\tikz[baseline=0] \draw[-<] (0,0) -- +(.001,0);}}
\newcommand{\nstartarrow}{node[near start,sloped,allow upside down] {\hspace{0.05cm}\tikz[baseline=0] \draw[->] (0,0) -- +(.001,0);}}
\newcommand{\nstartarrowrev}{node[near start,sloped,allow upside down] {\hspace{0.05cm}\tikz[baseline=0] \draw[-<] (0,0) -- +(.001,0);}}
\newcommand{\drawmiddlee}[2][]{
     \draw[arrowinthemiddlerev] (#2) ++(0,+0.1) .. controls ++(0.5,-0.2) and ++(-0.5,-0.2)  .. node[above] {\(#1\)} ++(1,0);
} 
\newcommand{\drawmiddlef}[2][]{
     \draw[arrowinthemiddle] (#2) ++(0,-0.1) .. controls ++(0.5,0.2) and ++(-0.5,0.2)  .. node[above] {\(#1\)} ++(1,0);
}

\usepackage{graphicx}
\usepackage{wrapfig}
 \usepackage[font=sf, labelfont={sf,bf}, margin=1cm]{caption}
\usepackage{subcaption}

 \usepackage[strict]{changepage}

\usepackage{graphicx}

\usepackage[final=true]{hyperref,bookmark}

\numberwithin{equation}{section}

\usepackage{etoolbox}
\makeatletter
\let\ams@starttoc\@starttoc
\makeatother
\usepackage{parskip}
\makeatletter
\let\@starttoc\ams@starttoc
\patchcmd{\@starttoc}{\makeatletter}{\makeatletter\parskip\z@}{}{}
\makeatother
\newtheoremstyle{myplain} {6pt plus 6pt minus 2pt}
{6pt plus 6pt minus 2pt}
{\itshape}
{}
{\bfseries}
{.}
{.5em}
{}

\theoremstyle{myplain}
\newtheorem{theorem}{Theorem}[section]
\newtheorem*{theorem*}{Theorem}
\newtheorem{lemma}[theorem]{Lemma}
\newtheorem{prop}[theorem]{Proposition}
\newtheorem{corollary}[theorem]{Corollary}

\newtheorem{maintheorem}{Main Theorem}

\newtheoremstyle{mydefinition} {6pt plus 6pt minus 2pt}
{6pt plus 6pt minus 2pt}
{\itshape}
{}
{\bfseries}
{.}
{.5em}
{}

\theoremstyle{mydefinition}
\newtheorem{definition}[theorem]{Definition}

\newtheoremstyle{myexample} {6pt plus 6pt minus 2pt}
{6pt plus 6pt minus 2pt}
{}
{}
{\scshape}
{.}
{.5em}
{}

\theoremstyle{myexample}
\newtheorem{example}[theorem]{Example}

\newtheoremstyle{myremark} {6pt plus 6pt minus 2pt}
{6pt plus 6pt minus 2pt}
{}
{}
{\scshape}
{.}
{.5em}
{}

\theoremstyle{myremark}
\newtheorem{remark}[theorem]{Remark}

\begingroup
    \makeatletter
    \@for\theoremstyle:=mydefinition,myremark,myplain,myexample\do{%
        \expandafter\g@addto@macro\csname th@\theoremstyle\endcsname{%
            \addtolength\thm@preskip\parskip
            }%
        }
\endgroup

\DeclareSymbolFontAlphabet{\mathbb}{AMSb}
\DeclareSymbolFontAlphabet{\mathbbol}{bbold}
\DeclareMathAlphabet{\mathpzc}{OT1}{pzc}{m}{it}

\DeclareSymbolFont{usualmathcal}{OMS}{cmsy}{m}{n}
\DeclareSymbolFontAlphabet{\mathucal}{usualmathcal}

\newcommand{\N}{\mathbb{N}}
\newcommand{\Z}{\mathbb{Z}}

\newcommand{\C}{\mathbb{C}}

\newcommand{\gl}{\mathfrak{gl}}
\newcommand{\catO}{\mathcal{O}}

\newcommand{\suchthat}{\mid} 
\newcommand{\mapto}{\rightarrow}
\newcommand{\longmapto}{\longrightarrow}
\newcommand{\surto}{\twoheadrightarrow}
\newcommand{\into}{\hookrightarrow}

\DeclareMathOperator{\Hom}{Hom}

\DeclareMathOperator{\End}{End}

\newcommand{\id}{\mathrm{id}}

\newcommand{\abs}[1]{\left|#1\right|}

\renewcommand{\epsilon}{\varepsilon}

\renewcommand{\phi}{\varphi}

\newcommand{\down}{{\mathord\vee}}

\newcommand{\calB}{{\mathcal{B}}}
\newcommand{\calG}{{\mathcal{G}}}
\newcommand{\calH}{{\mathcal{H}}}
\newcommand{\calI}{{\mathcal{I}}}
\newcommand{\calM}{{\mathcal{M}}}
\newcommand{\calR}{{\mathcal{R}}}

\newcommand{\calQ}{{\mathcal{Q}}}

\newcommand{\sfE}{{\mathsf{E}}}
\newcommand{\sfF}{{\mathsf{F}}}
\newcommand{\sfP}{{\mathsf{P}}}
\newcommand{\sfT}{{\mathsf{T}}}

\newcommand{\sfi}{{\mathsf{i}}}
\newcommand{\sfp}{{\mathsf{p}}}
\newcommand{\sfs}{{\mathsf{s}}}
\newcommand{\frakh}{{\mathfrak{h}}}

\newcommand{\bolda}{{\boldsymbol{a}}}
\newcommand{\boldb}{{\boldsymbol{b}}}
\newcommand{\boldc}{{\boldsymbol{c}}}
\newcommand{\boldd}{{\boldsymbol{d}}}

\newcommand{\boldeta}{{\boldsymbol{\eta}}}

\newcommand{\acts}{\mathop{\,\;\raisebox{1.7ex}{\rotatebox{-90}{$\circlearrowright$}}\;\,}}
\newcommand{\actsb}{\mathop{\,\;\raisebox{0ex}{\rotatebox{90}{$\circlearrowleft$}}\;\,}}

\newcommand{\checkR}{\check{R}}

\newcommand{\Spider}{\mathsf{Sp}}
\newcommand{\sSpider}{\mathsf{Sp}_{\mathrm{sym}}}

\newcommand{\idem}{\boldsymbol{1}}
\newcommand{\onel}{\boldsymbol{1}_{\lambda}}
\newcommand{\onell}[1]{\boldsymbol{1}_{#1}}
\newcommand{\bigwedgeq}{\textstyle \bigwedge_q}
\DeclareMathOperator{\bigsymm}{S}

\newcommand{\bigsymmq}{\bigsymm_q}
\DeclareMathOperator{\bigT}{T}
\newcommand{\wedgeq}{\mathbin{\wedge_q}}

\newcommand{\fieldk}{{\mathbbm{k}}}

\newcommand{\ucalH}{{\mathucal H}}

\newcommand{\RT}{{\mathcal{RT}}}

\newcommand{\Br}{\mathrm{Br}}
\newcommand{\uBr}{\mathsf{Br}}
\newcommand{\catH}{\mathsf{H}}

\newcommand{\quantumq}{{\boldsymbol{\mathrm{q}}}}
\newcommand{\qbin}[2]{\genfrac{[}{]}{0pt}{}{#1}{#2}}

\newcommand{\ev}{\mathrm{ev}}
\newcommand{\coev}{\mathrm{coev}}
\newcommand{\counit}{\boldsymbol{\mathrm{u}}}

\newcommand{\U}{\dot{{\bf U}}}

\newcommand{\slm}{\mathfrak{sl}_m}

\newcommand{\Links}{{\mathsf{Links}}}
\newcommand{\catTangles}{{\mathsf{Tangles}}}
\newcommand{\lab}{{\ell\!\textit{a\hspace{-1pt}b}}}
\newcommand{\catRep}{{\mathsf{Rep}}}

\newcommand{\lattice}{\sfP}

\hypersetup{
  colorlinks = true,
  urlcolor = blue,
  pdfauthor = {Hoel Queffelec and Antonio Sartori},
  pdfkeywords = {Howe duality, Alexander polynomial, HOMFLY-PT polynomial, quantum link invariants, Spider category},
  pdftitle = {Mixed quantum skew Howe duality and link invariants of type A},
  pdfsubject = {},
  pdfpagemode = UseNone,
  bookmarksopen = true,
  bookmarksopenlevel = 3,
  pdfdisplaydoctitle = true }

\title[Mixed skew Howe duality and link invariants]{Mixed quantum skew Howe duality and link invariants of type $A$}

\author{Hoel Queffelec}
\address{Mathematical Sciences Institute, Australian National University, J.D. 27 Union Lane, Acton ACT 2601, Australia}
\email{hoel.queffelec@anu.edu.au}

\author{Antonio Sartori}
\address{Mathematisches Institut, Albert-Ludwigs-Universität Freiburg,
Eckerstraße 1, 79104 Freiburg im Breisgau, Germany}
\email{antonio.sartori@math.uni-freiburg.de}

\tikzset{smallnodes/.style={every node/.style={font=\footnotesize}}}

\begin{document}

\begin{abstract}
We define a ribbon category \(\Spider(\beta)\), depending on a parameter \(\beta\), which encompasses Cautis, Kamnitzer and Morrison's spider category, and describes for \(\beta=m-n\) the monoidal category of representations of \(U_q(\gl_{m|n})\) generated by exterior powers of the vector representation and their duals. We identify this category \(\Spider(\beta)\) with a direct limit of quotients of a dual idempotented quantum group \(\U_q(\gl_{r+s})\), proving a mixed version of skew Howe duality in which exterior powers and their duals appear at the same time. We show that the category \(\Spider(\beta)\) gives a unified natural setting for defining the colored \(\gl_{m|n}\) link invariant (for \(\beta=m-n\)) and the colored HOMFLY-PT polynomial (for \(\beta\) generic).
\end{abstract}

\maketitle

\setcounter{tocdepth}{1}
\tableofcontents

\section{Introduction}
\label{sec:introduction}

In 1985 Jones \cite{Jones} defined a remarkable new polynomial invariant of links.
A few years later, Reshetikhin and Turaev \cite{MR1036112} reinterpreted the Jones polynomial and generalized it by defining a whole class of quantum link invariants. Their work, based on the use of quantum groups and built upon the introduction of  ribbon categories, opened the way to a world of  connections between knot theory and representation theory. In particular, they defined quantum link invariants attached to finite dimensional representations of semisimple Lie algebras, recovering the Jones polynomial for \(\mathfrak{sl}_2\).
Although the Alexander polynomial \cite{MR1501429} was not originally covered by Reshetikhin and Turaev's process, it has been well-known to experts that it is possible to recover it using the super quantum group \(U_q(\gl_{1|1})\) \cite{MR2255851} (see also \cite{SarAlexander} for a more accessible explanation and for additional references). (An alternative and earlier construction uses the quantum group \(U_q(\mathfrak{sl}_2)\) for \(q\) a root of unity \cite{MurakamiAlexander}, but we will not pursue this approach.)
In particular, both the Jones and the Alexander polynomials fit together inside the family of quantum link invariants associated to the Lie superalgebras \(\gl_{m|n}\). We shall mention another link invariant, the HOMFLY-PT polynomial \cite{Homfly,PT}, which in a certain sense could be interpreted as a limit version for \(m \rightarrow \infty\) of the \(\gl_m\) link invariants, but does not  fit \emph{a priori}  in the setting of quantum link invariants.

A major new step in knot theory was the idea of \emph{categorification}. Khovanov's seminal work \cite{Kh1,Kh2,Kh6} lifting the Jones polynomial to a homology theory, initiated an effort to categorify the polynomial knot invariants, which produced an \(\mathfrak{sl}_3\)  \cite{Kh5} and then an \(\mathfrak{sl}_m\) link homology  \cite{KhR} via matrix factorizations, as well as a triply graded link homology theory \cite{MR2339573}, based on Soergel bimodules, categorifying the HOMFLY-PT polynomial. In a parallel and connected way, a categorification program for quantum groups and their representations started and prospered in the last fifteen years, using tools from representation theory and geometry. Milestones were the categorifications using the BGG category \(\catO\), initiated in \cite{MR1714141,MR2305608} and the definition of the Khovanov-Lauda-Rouquier algebra \cite{KL,MR2763732,MR2628852,Rou2}. We point out here that, although a categorification of the Alexander polynomial does exist \cite{MR2372850}, it comes from a completely different theory (symplectic geometry tools and in particular Heegaard-Floer homology), and a representation theoretical interpretation of it is still missing. Conjecturally, there should be a whole family of \(\gl_{m|n}\) link homologies which should encompass all previously mentioned categorifications (see also \cite{2013arXiv1304.3481G}).

In a remarkable paper, Cautis, Kamnitzer and Morrison \cite{CKM} showed how the classical skew Howe duality \cite{MR986027,MR1321638} for the pair \((\gl_m,\gl_k)\) can be used to interpret the braiding of \(U_q(\gl_m)\)--representations as the action of a dual quantum group \(U_q(\gl_k)\). In particular, they describe the combinatorics of intertwiners of tensor products of exterior powers of the natural representation \(\C_q^m\) of \(U_q(\gl_m)\) using a spider category, which is equivalent to the direct limit  for \(k \rightarrow \infty\) of some quotients of the idempotented versions of a dual quantum group \(U_q(\gl_k)\). This approach seems to give the right combinatorial setting for categorification \cite{LQR1,MY,TubbenhauerSln,QR}, and has been recently  used to give uniqueness results for link homology theories \cite{2015arXiv150206011M}.

The work of Cautis, Kamnitzer and Morrison actually really applies to braids. Indeed, in the process of assigning an intertwiner of representations to an oriented tangle (and, in particular, a knot or link), one needs to have both the standard representation \(\C_q^{m}\) and its dual at his disposal. The fact that Cautis, Kamnitzer and Morrison's result yields \(\mathfrak{gl}_m\) link invariants (and can be used for their categorification) is based on the identification of the exterior powers 
\begin{equation}
\textstyle \bigwedge^a \C^m \cong \big(\bigwedge^{m-a} \C^m\big)^* \quad\text{as \(\mathfrak{sl}_m\)--representations}.\label{eq:197}
\end{equation}
With the aim of studying \(\gl_{m|n}\) link invariants and possibly their categorification, this has two main limitations.

The first one concerns its extension to the super case, which is needed to define \(\smash{\gl_{m|n}}\) link invariants (and, in particular, the Alexander polynomial).
Here we denote by \(\gl_{m|n}\) the general Lie superalgebra and by \(\smash{\C_q^{m|n}}\) its standard representation. A super version of skew Howe duality does hold for the pair \((\gl_{m|n},\gl_k)\) and can be used, in a similar way as in \cite{CKM}, to describe intertwiners of tensor products of exterior powers of \(\C_q^{m|n}\). This approach can even be used to categorify  these representations, see \cite{2013arXiv1305.6162S}.  However, since in the case \(n \neq 0\) no exterior power of the standard representation of \(\gl_{m|n}\) is isomorphic to its dual, the dual of the standard representation does not appear in skew Howe duality. Indeed, there is a qualitative distinction: the representation appearing in skew Howe duality is always semisimple, while in general mixed tensor products of the vector representation together with its dual are not. 

The second limitation appears when one wants to inspect the HOMFLY-PT polynomial. In this case, roughly speaking, one would like to take a limit for \(m \rightarrow \infty\). The identification \eqref{eq:197} above becomes then meaningless and again the dual of the standard representation disappears from the picture.

\emph{The goal of this paper} is to overcome these two limitations. \emph{Our main result} is an extension of skew Howe duality to a more general picture where both exterior powers of the vector representation \emph{and of its dual} appear at the same time, and where one can make sense of the limit \(m \rightarrow \infty\) by taking a generic value of the parameter \(\beta\).

An overview of the categories and functors that we will consider is presented in figure~\ref{fig:overview}. The top row of figure~\ref{fig:overview} is the generalization of \cite{CKM} to the super case of \(\smash{\gl_{m|n}}\), and describes intertwiners of exterior powers of the natural representation \(\smash{\C_q^{m|n}}\). The second row is our extension to exterior powers and their duals. The third row will serve as auxiliary tool. We will describe some of the categories and functors involved in the following.

\begin{figure}
  \centering
  \begin{tikzpicture}[baseline=(current bounding box.center)]
  \matrix (m) [matrix of math nodes, row sep=3em, column
  sep=10.5em, text height=1.5ex, text depth=0.25ex] {
    \catH & \Spider^+ & \catRep^+_{m|n} \\
    \uBr(\beta) & \Spider(\beta) & \catRep_{m|n} \\
    \catTangles_q & \Spider & \\};
  \node(A) at ($(m-1-2) + (1.5,-0.8) $)  {$\Spider^+_{\leq m}$};
  \node(B) at ($(m-2-2) + (1.5,-0.8) $)  {$\Spider(\beta)_{\leq m}$};
  \node(C) at ($(m-3-1) + (2.5,0) $)  {$\catTangles^\lab_q$};

  \path[right hook->] (A) edge (B);

  \path[->] (m-1-1) edge node[above] {$ \calI^+$} (m-1-2);
  \path[->>] (m-1-2) edge node[above] {$\calG^+_{m|n}$} (m-1-3);
  \path[->] (m-2-1) edge node[above] {$\calI_\beta $} (m-2-2);
  \draw[cross line,->>] (m-2-2) -- node[xshift=0.5cm,above] {$\calG_{m|n}$} (m-2-3);
  \draw[->>] (m-3-2) .. controls ++(0:2) and ++ (-120:2) .. (m-2-3)  node[midway,below] {$\calG_{m|n}$} ;

  \path[right hook->] (m-1-1) edge  (m-2-1);
  \path[right hook->] (m-1-2) edge node[left] {\(\sfi\)} (m-2-2);
  \path[right hook->] (m-1-3) edge  (m-2-3);
  \path[->>] (m-3-1) edge (m-2-1);
  \path[->>] (m-3-2) edge node[left] {\(\sfp\)} (m-2-2);
  \draw[->>] (C) .. controls ++(-30:3) and ++(-90:3) ..  (m-2-3) node[midway,below] {$\RT_{m|n}$};
  \draw[right hook->] (m-3-1) edge (C);
  \draw[->] (C) edge node[yshift=0.8ex,left=0.5ex] {\(\calQ_\beta\)} (m-2-2);

  \path[->>] (m-1-2) edge (A);
  \path[->>] (m-2-2) edge (B);
  \path[->>] (A) edge (m-1-3);
  \path[->>] (B) edge (m-2-3);
\end{tikzpicture}
   \caption{Overview of the categories and functors involved in our construction. (The third column makes sense only for \(\beta=m-n\).)}
  \label{fig:overview}
\end{figure}
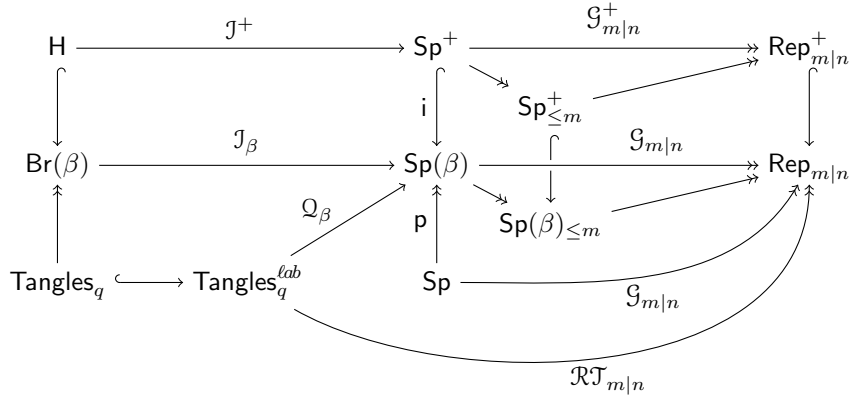

\subsection*{Graphical calculus for duals}

Let \(m,n\) be two non-negative integers. We will always set  \(d=m-n\). We denote by \(\catRep_{m|n}^+\) (respectively, \(\catRep_{m|n}\)) the monoidal category of \(\smash{U_q(\mathfrak{gl}_{m|n})}\)--representations generated by exterior powers of the natural representation \(\smash{\C_q^{m|n}}\) (respectively, exterior powers of  \(\smash{\C_q^{m|n}}\) and  their duals). The first goal of our paper is to obtain a unified graphical calculus for \(\catRep_{m|n}\). This extends in two directions the work of Cautis, Kamnitzer and Morrison, who described the category \(\catRep_{m|0}^+\) (and with the trick \eqref{eq:197} actually described \(\catRep_{m|0}\), if restricting to \(\mathfrak{sl}_m\)).

The graphical calculus of \cite{CKM} is given by a spider category (introduced in rank 2 already in \cite{MR1403861}) describing \(\catRep_{m|0}^+\). The extension to the super case (replacing \(\gl_m\) by \(\gl_{m|n}\)), which we will develop in all details (see section~\ref{sec:super-skew-howe}), is relatively easy. Similarly to \cite{CKM} we can define a  monoidal category \(\Spider^+\) together with a full functor \(\smash{\calG^+_{m|n}\colon \Spider^+ \mapto \catRep_{m|n}^+}\) (theorem~\ref{thm:2}). The category \(\Spider^+\) is some sort of universal tool, built in order to describe all categories \(\catRep_{m|n}^+\) at the same time. It is an interesting task to find  all relations in \(\smash{\catRep_{m|n}^+}\), i.e.\ to describe a monoidal ideal \(\calI_{m|n}\) of \(\Spider^+\) such that \(\calG^+_{m|n}\) descends to a fully faithful functor  \(\Spider^+ / \calI_{m|n} \mapto \catRep^+_{m|n}\). This is very easy for \(n=0\) (one just kills objects of \(\Spider^+\) labeled by integers greater than \(m\), whence the notation \(\Spider^+_{\leq m} = \Spider^+ / \calI_{m|0})\)) and it is possible also in general, although the relations are not so nice (see \cite{Grant} for the case of \(\smash{\gl_{1|1}}\)).
An interesting fact is that we can  lift the braiding of all the categories \(\smash{\catRep_{m|n}^+}\) directly to \(\smash{\Spider^+}\) in a uniform way, turning it into a braided category (proposition~\ref{prop:5}).

Now, with the goal of studying link invariants and their categorification, we are interested in understanding the full category \(\catRep_{m|n}\), and not only its ``positive part'' \(\smash{\catRep_{m|n}^+}\). But, as we already explained, in the case \(n \neq 0\) the trick \eqref{eq:197} does not work anymore. Hence we define a rigid category \(\Spider\)  by  adding duals to \(\smash{\Spider^+}\), and we show that  the functor \(\smash{\calG_{m|n}^+}\)  extends to a full functor \(\smash{\calG_{m|n} \colon\Spider\mapto \catRep_{m|n}}\) (proposition~\ref{prop:7}).
The category \(\Spider\), however, is too big: its endomorphism spaces are infinite-dimensional and the  functor \(\calG_{m|n}\), hence, cannot be faithful. Moreover, there is no map from  the category of oriented framed tangles to \(\Spider\) which could give link invariants.  Hence we introduce a quotient \(\Spider(\beta)\) of \(\Spider\), depending on a parameter \(\beta\) which can be either a generic variable (corresponding to the second variable of the HOMFLY-PT polynomial) or an integer. Our first main result is then a complete graphical description of the category \(\catRep_{m|n}\):
\begin{maintheorem}[See theorem~\ref{thm:6} and proposition~\ref{lem:16}]
  For \(\beta=d=m-n\)
   we get an induced functor
  \(\calG_{m|n}\colon \Spider(d) \mapto \catRep_{m|n}\)
   which is still full. Moreover, this descends to a fully faithful functor \(\Spider(d)/(\calI_{m|n}) \mapto \catRep_{m|n}\). (Here \((\calI_{m|n})\) denotes the monoidal ideal generated by the image of \(\calI_{m|n}\) in \(\Spider(d)\)).
\end{maintheorem}
Similarly as before, the ribbon structure of the representation categories \(\catRep_{m|n}\) lifts in a unified way to \(\Spider(\beta)\), which is hence a ribbon category (theorem~\ref{thm:9}).

\subsection*{Mixed skew-Howe duality and the doubled Schur algebra}

In \cite{CKM} the main tool yielding the description of the category \(\catRep^+_{m}\) is skew Howe duality, which gives two commuting actions 
  \begin{equation}
    \label{eq:199}
    U_q(\gl_m) \acts  \bigwedgeq^{N} (\C_q^{m} \otimes \C_q^r) \actsb U_q(\gl_{r})
  \end{equation}
producing a description of \(U_q(\gl_m)\) intertwiners by elements of \(U_q(\gl_r)\), and \textit{vice versa}. In particular, this provides an identification of the category \(\Spider^+\) with a the direct limit for \(r \rightarrow \infty\) of a quotient of the idempotented version of  \( U_q(\gl_r)\). Furthermore, the action of the braid group on the \(U_q(\gl_m)\)--modules is identified as the quantum Weyl group action generated by Lusztig's symmetries \cite{Lus4,KamnTin,CKL,CKM}. 

Surprisingly, it is possible to generalize skew Howe duality and get the dual \(\big(\C_q^{m|n}\big)^*\), together with its exterior powers, into the picture. In our setting, we proceed the other way around with respect to \cite{CKM}, and we deduce such generalization from our graphical calculus:
\begin{maintheorem}[See theorem~\ref{thm:7} and corollary~\ref{cor:4}]\label{thm:11}
On the representation
\begin{equation}
\label{eq:200}
 \bigoplus_{\mathclap{\qquad \substack{a_1+\dotsb+a_r-a_{r+1}\\-\dotsb-a_{r+s}-(m-n)s = N}}} \quad \bigwedgeq^{a_1} \C_q^{m|n} \otimes \dotsb \otimes \bigwedgeq^{a_r}\C_q^{m|n} \otimes \bigwedgeq^{a_{r+1}} \big(\C_q^{m|n}\big)^* \otimes \dotsb \otimes \bigwedgeq^{a_{r+s}} \big(\C_q^{m|n}\big)^*
\end{equation}
there is a naturally defined \(\smash{U_q(\gl_{m|n})}\)--equivariant action of the quantum group \(U_q(\gl_{r+s})\) and of its idempotented version \(\U_q(\gl_{r+s})\). The latter action generates the full centralizer of the \(\smash{U_q(\gl_{m|n})}\)--action. This induces an equivalence of categories between \(\Spider(d)\) and the  direct limit of a quotient of \(\U_q(\gl_{r+s})\) for \(r,s \rightarrow \infty\).
\end{maintheorem}
Such a quotient was already introduced by the two authors in \cite{QS} with the name of \emph{doubled Schur algebra}. For an illustrative picture see figure~\ref{fig:exampleMixedHowe} and example~\ref{ex:1} on page~\pageref{ex:1}.

To the best of our knowledge, the dual actions on \eqref{eq:200} are new, also in the non-quantized setting. A related duality in the non-quantized setting was analyzed in \cite[section~3]{MR2037723}, where the roles of \(\gl_{m|n}\) and \(\gl_{r+s}\) are swapped. However, the duality from \cite{MR2037723} is somehow easier, since the representation involved is semisimple, while \eqref{eq:200} is in general not semisimple.
Moreover, it is not clear how to quantize the results from \cite[section~3]{MR2037723}; indeed, it is not even evident how to define the dual action in the quantized setting.
We point out that it should be possible to generalize our methods in order to obtain dual actions of \(U_q(\gl_{m|n})\) and \(U_q(\gl_{r+s|r'+s'})\) on a bigger space than \eqref{eq:200} (which would contain also symmetric powers). For \(q=1\), this would include the duality of \cite{MR2037723} as a special case.

\subsection*{Link invariants, categorification and the parameter \(\beta\)}

Let us denote by \(\catTangles^\lab\) the category of oriented framed tangles whose strands are labeled by positive numbers. For all \(\beta\) we have a  functor \(\calQ_{\beta}\colon\catTangles^\lab \mapto \Spider(\beta)\), which 
for \(\beta=d\),
composed with the action of \(\Spider(d)\) on \(\catRep_{m|n}\), gives the Reshetikhin-Turaev invariant (proposition~\ref{prop:9}). This allows us to define inside \(\Spider(d)\) the \(\gl_{m|n}\) link invariant labeled by exterior powers of the vector representation.
The study of the link invariants is only briefly outlined in this paper, but the results we obtain here are central in the second part of our first paper \cite{QS}. The reader interested in the consequences of our approach in the study of link invariants should refer to \cite{QS}.

We stress that this could also be of great interest for \emph{categorification}. Indeed, after \cite{CKM}, the categorification of \(\gl_m\) link invariants can be interpreted as a categorification of the functor \(\catTangles^\lab \mapto \Spider^+_{\leq m}\), i.e.\ as a 2-functor from the 2-category of tangles to a 2-categorical lift of \(\Spider^+_{\leq m}\) (cf.\ \cite{LQR1,QR}). What such a lift should be becomes clear after \(\Spider^+_{\leq m}\) is identified by \cite{CKM} with a sum of quotients of idempotented version of quantum groups: the natural target of such 2-functor is then some version of the KLR 2-category \cite{MR2628852,Rou2}.

Similarly, in our case, the categorification of \(\smash{\gl_{m|n}}\) link invariants should be a categorification of the functor \(\catTangles_q^\lab \mapto \Spider(d)\). We believe that the identification of \(\Spider(d)\) with the direct limit of doubled Schur algebras can point in the right direction for lifting \(\Spider(d)\) to a 2-category and hence constructing a categorification of  \(\gl_{m|n}\) link invariants via categorification of such doubled Schur algebras.

Moreover, we point out that the definition of the category \(\Spider(\beta)\) depends on a parameter \(\beta\), and the action of \(\Spider(\beta)\) on \(U_q(\gl_{m|n})\) representation comes when \(\beta\) is specialized to \(d=m-n\). Nevertheless, our construction of \(\Spider(\beta)\) makes sense also for a generic value of \(\beta\). In this case, the natural functor \(\calQ_\beta \colon \catTangles^\lab \mapto \Spider(\beta)\) can be used to define the HOMFLY-PT polynomial \cite{Homfly,PT}. A categorification of this functor, via a similar strategy as in the case \(\beta=d\) explained above, could be useful for better understanding the triply graded link homology \cite{MR2339573} and its relation with \(\gl_m\) and more generally \(\gl_{m|n}\) link homology (see also \cite{2013arXiv1304.3481G} for some conjectural properties).

\subsection*{Structure of the paper}
\label{sec:structure-paper}

In section~\ref{sec:overview} we will fix some notations and recall the definition of the quantized oriented Brauer algebra. In section~\ref{sec:general-linear-lie} we will collect some basic facts on the quantum enveloping superalgebra \(U_q(\gl_{m|n})\) and its representations. Section~\ref{sec:super-skew-howe} will be devoted to the proof of the quantized super version of classical skew Howe duality. In section~\ref{sec:graphical-version} we will recall the definition of the braided category \(\Spider^+\) from \cite{CKM} and we will construct its action on \(\smash{\catRep^+_{m|n}}\). In the following section~\ref{sec:graph-calc-duals} we will define our main object, the  category \(\Spider(\beta)\), by adding duals to \(\Spider^+\). We will prove that this category acts fully on \(\catRep_{m|n}\) for \(\beta=m-n\) and we will construct its ribbon structure. In section~\ref{sec:dbleSchur} we will give the quantum group interpretation of \(\Spider(\beta)\), proving our main result. Finally, in section~\ref{sec:link-invariants} we describe some applications to link invariants.

\subsubsection*{Acknowledgments}

We would like to thank Christian Blanchet, Aaron Lauda, Tony Licata, Scott Morrison, David Rose, Catharina Stroppel, Daniel Tubbenhauer, Emmanuel Wagner and Weiqiang Wang for useful discussions and comments. 

\section{The category of tangles}
\label{sec:overview}

\subsection{The parameter \(\beta\)}
\label{sec:conventions}

All our diagrammatic categories will be defined over a field \(\fieldk\) containing the complex numbers \(\C\) and two elements \(q\) and \(q^\beta\) which are not roots of the unity. The first main case we are interested in is \(\fieldk=\C(q,q^\beta)\), i.e.\ a transcendental extension of \(\C(q)\) by a formal variable \(q^\beta\). In this case we say that \(\beta\) is \emph{generic}. The second main example is when \(\beta\) is some integer \(d\) and \(\fieldk=\C(q)\). We will sometimes write \(\C_q\) instead of \(\C(q)\) for shortness.

For \(x \in \Z \beta + \Z\) and \(k \in \N\) we define
\begin{align}
  \label{eq:195}
  [x] & = \frac{q^x - q^{-x}}{q - q^{-1}},\\
  \qbin{x}{k} & = \frac{[x][x-1]\dotsb [x-k+1]}{[k][k-1]\dotsb [1]}.
\end{align}
The following identities hold for all \(x,y \in \Z \beta + \Z\) and \(k \in \N\), and will be used often through the paper:
\begin{gather}
  [x][y+1] - [x+1][y] = [x-y],\label{eq:3}\\
  [x][y] - [x-1][y-1] = [x+y-1],\label{eq:4} \\
  [x+y]=q^y [x] + q^{-x}[y]= q^{-y}[x] + q^{x}[y], \label{eq:5}\\
  \qbin{x+1}{k}=q^{x+1-k}\qbin{x}{k-1}+q^{-k}\qbin{x}{k}=q^{k}\qbin{x}{k}+q^{-x-1+k}\qbin{x}{k-1}, \label{eq:6}\\
\sum_{\ell=0}^k (-q)^\ell\qbin{x-\ell}{k}\qbin{k}{\ell}=q^{-k x + k^2}.
\end{gather}

\subsection{Tangles}
\label{sec:tangles}

 Let \(\catTangles\) be the monoidal category 
of oriented framed tangles. Its objects are sequences of \(\{1,1^*\}\) and its morphisms are oriented framed tangles modulo isotopy. The category \(\catTangles\) can be described in terms of generators and relations, see for example \cite{MR1881401}. Let also \(\catTangles_q\) be the additive closure of its \(\fieldk\)--linear version, with morphisms being \(\fieldk\)--vector spaces
\begin{equation}
  \label{eq:2}
  \catTangles_{q}(\boldeta,\boldeta') = \operatorname{span}_{\fieldk} \catTangles(\boldeta,\boldeta')
\end{equation}
and where as objects we have formal (finite) direct sums of objects of \(\catTangles\).

We let \(\catTangles^\lab\) be the category consisting of oriented framed tangles whose connected components are labeled by positive integers. The objects of \(\catTangles^\lab\) are sequences of \(\{a,a^* \suchthat a \in \Z_{>0}\}\). Similarly as before, we denote by \(\catTangles^\lab_q\) its additive \(\fieldk\)--linear version. There are  obvious inclusions \(\catTangles \into \catTangles^\lab\) and \(\catTangles_q \into \catTangles_q^\lab\) given by labeling all strands by \(1\).

We will regard all these tangle categories  as monoidal categories, with monoidal unity being the empty sequence.

\subsection{The quantized oriented Brauer category}
\label{sec:quant-wall-brau}

We recall the definition of the quantized oriented Brauer algebra/category, following \cite{PhdWeimer} and \cite{MR3181742} (see also \cite{Blanchet_Hecke,GP_Hecke} for a related approach).

\begin{definition}\label{def:2}
  The \emph{quantized oriented Brauer category} \(\uBr(\beta)\) is the quotient of \(\catTangles_q\) modulo the following relations
\begin{subequations}
  \begin{align}
    \begin{tikzpicture}[smallnodes,anchorbase,xscale=0.7,yscale=0.5]
   \draw[uno] (0,0)  -- ++(0,0.3) \midarrow .. controls ++(0,0.7) and ++(0,-0.7) .. ++(1,1.4) -- ++(0,0.3)  \midarrow ;
   \draw[uno,cross line] (1,0)  -- ++(0,0.3) \midarrow .. controls ++(0,0.7) and ++(0,-0.7) .. ++(-1,1.4) -- ++(0,0.3)  \midarrow ;
    \end{tikzpicture} \; - \;
     \begin{tikzpicture}[smallnodes,anchorbase,xscale=0.7,yscale=0.5]
   \draw[uno] (1,0)  -- ++(0,0.3) \midarrow .. controls ++(0,0.7) and ++(0,-0.7) .. ++(-1,1.4) -- ++(0,0.3)  \midarrow ;
   \draw[uno,cross line] (0,0)  -- ++(0,0.3) \midarrow .. controls ++(0,0.7) and ++(0,-0.7) .. ++(1,1.4) -- ++(0,0.3)  \midarrow ;
    \end{tikzpicture} \; & = (q-q^{-1}) \;
    \begin{tikzpicture}[smallnodes,anchorbase,yscale=0.5,xscale=0.7]
   \draw[uno] (1,0) .. controls ++(0.25,0.5) and ++(0.25,-0.5)  .. ++(0,2) \midarrow ;
   \draw[uno] (2,0) .. controls ++(-0.25,0.5) and ++(-0.25,-0.5)  .. ++(0,2) \midarrow  ;
    \end{tikzpicture}\;,  \qquad&
        \begin{tikzpicture}[smallnodes,anchorbase,scale=0.7]
      \draw[uno] (0,0) arc (0:360:0.5cm) \midarrowrev;
    \end{tikzpicture} \; =\;
        \begin{tikzpicture}[smallnodes,anchorbase,scale=0.7]
      \draw[uno] (0,0) arc (0:360:0.5cm) \midarrow;
    \end{tikzpicture} \; &= [\beta], \label{eq:52}\\
    \begin{tikzpicture}[smallnodes,anchorbase,xscale=0.7,yscale=0.5]
   \draw[uno] (0.8,1)  .. controls ++(0,-0.3) and ++(0.3,0) .. ++(-0.3,-0.6)  .. controls ++(-0.5,0) and ++(0,-0.7) .. ++(-0.5,1.6) \nendarrow  ;
   \draw[uno, cross line] (0.8,1)  .. controls ++(0,0.3) and ++(0.3,0) .. ++(-0.3,0.6) \startarrowrev .. controls ++(-0.5,0) and ++(0,0.7) .. ++(-0.5,-1.6) \nendarrowrev;
    \end{tikzpicture} \; = \;
    \begin{tikzpicture}[smallnodes,anchorbase,xscale=0.7,yscale=0.5]
   \draw[uno, cross line] (-0.8,1)  .. controls ++(0,0.3) and ++(-0.3,0) .. ++(0.3,0.6) \startarrowrev .. controls ++(0.5,0) and ++(0,0.7) .. ++(0.5,-1.6) \nendarrowrev ;
   \draw[uno, cross line] (-0.8,1)  .. controls ++(0,-0.3) and ++(-0.3,0) .. ++(0.3,-0.6)  .. controls ++(0.5,0) and ++(0,-0.7) .. ++(0.5,1.6) \nendarrow  ;
    \end{tikzpicture} \;& = q^{-\beta} \;
    \begin{tikzpicture}[smallnodes,anchorbase,xscale=0.7,yscale=0.5]
      \draw[uno] (0,0)  -- ++(0,2) \midarrow ;
    \end{tikzpicture}\;, &
    \begin{tikzpicture}[smallnodes,anchorbase,xscale=0.7,yscale=0.5]
   \draw[uno, cross line] (0.8,1)  .. controls ++(0,0.3) and ++(0.3,0) .. ++(-0.3,0.6) \startarrowrev .. controls ++(-0.5,0) and ++(0,0.7) .. ++(-0.5,-1.6) \nendarrowrev ;
   \draw[uno, cross line] (0.8,1)  .. controls ++(0,-0.3) and ++(0.3,0) .. ++(-0.3,-0.6)  .. controls ++(-0.5,0) and ++(0,-0.7) .. ++(-0.5,1.6) \nendarrow  ;
    \end{tikzpicture} \; = \;
    \begin{tikzpicture}[smallnodes,anchorbase,xscale=0.7,yscale=0.5]
   \draw[uno] (-0.8,1)  .. controls ++(0,-0.3) and ++(-0.3,0) .. ++(0.3,-0.6)  .. controls ++(0.5,0) and ++(0,-0.7) .. ++(0.5,1.6) \nendarrow  ;
   \draw[uno, cross line] (-0.8,1)  .. controls ++(0,0.3) and ++(-0.3,0) .. ++(0.3,0.6) \startarrowrev .. controls ++(0.5,0) and ++(0,0.7) .. ++(0.5,-1.6) \nendarrowrev;
    \end{tikzpicture} \; & = q^{+\beta} \;
    \begin{tikzpicture}[smallnodes,anchorbase,xscale=0.7,yscale=0.5]
      \draw[uno] (0,0)  -- ++(0,2) \midarrow ;
    \end{tikzpicture}\;.\label{eq:54}
  \end{align}
\end{subequations}
\end{definition}

\begin{definition}
  \label{def:3}
  If \(\boldeta \in \{1,1^*\}^N\) is a sequence of orientations, the \emph{quantized oriented Brauer algebra} \(\Br_\boldeta(\beta)\) is the endomorphism algebra \(\End_{\uBr(\beta)}(\boldeta)\).
\end{definition}

For \(r,s\geq 0\) let
\begin{equation}
  \label{eq:166}
  \boldeta_{r,s} = (\underbrace{1,\dotsc,1}_r,\underbrace{1^*,\dotsc,1^*}_s).
\end{equation}
Then \(\Br_{\boldeta_{r,s}}(\beta)\) is the \emph{walled Brauer algebra} with parameter \(\beta\).

It is shown in \cite[Lemma~1.4]{MR3181742} and in \cite{PhdWeimer} that \(\Br_\boldeta(\beta)\) is free of rank \(k!\), where \(k\) is the length of \(\boldeta\).

We recall that for \(r\geq 0\) the \emph{Hecke algebra} \(\ucalH_r\) is the \(\C(q)\) algebra on generators \(H_i\) for \(i=1,\dotsc,r-1\) with relations 
  \begin{subequations}
    \label{eq:53}
    \begin{align}
      H_i H_j &= H_j H_i && \text{if } \vert i-j\vert>2,\label{catO:eq:7} \\
      H_i H_{i+1} H_i  &= H_{i+1} H_i H_{i+1}, &&
      H_i^2=(q^{-1}-q)H_i+1,\label{catO:eq:9}
    \end{align}
  \end{subequations}
see \cite{MR560412}, \cite{MR1445511}. It follows that if \(\boldeta = 1^N\) then \(\Br_\boldeta(\beta)\), which does not depend on \(\beta\), is isomorphic to \(\ucalH_N\). We denote by \(\catH\) the full subcategory of \(\uBr(\beta)\) monoidally generated by \(1\), which  does not depend on \(\beta\).

\section{The general linear Lie superalgebra}
\label{sec:general-linear-lie}

We will denote by \(m\) and \(n\) two non-negative integers, and we will always suppose that either \(m\) or \(n\) is non-zero. We set also \(d=m-n\). In the following, as usual, \emph{super} means \(\Z/2\Z\)--graded.

\subsection{The Lie superalgebra \texorpdfstring{$\gl_{m|n}$}{gl(m|n)}}
\label{sec:lie-super-gl_mn}

Let \(I=I_{m|n}\) denote the set \(\{1, \dotsc , m+n\}\). We define a  parity function
\begin{equation}
  \label{eq:7}
  \abs{i} =
  \begin{cases}
    0 & \text{if } i \leq m,\\
    1 & \text{if } i > m.
  \end{cases}
\end{equation}
We let also \(I^0=I \setminus \{m+n\}\).

Let \(\C^{m|n}\) be the vector superspace on homogeneous basis \(x_i\) for \(i \in I_{m|n}\), with \(\Z/2\Z\)--degree \(\abs{x_i}=\abs{i}\). The \emph{general linear Lie superalgebra} \(\smash{\gl_{m|n}}\) is the endomorphism space of \(\C^{m|n}\) equipped with the supercommutator \([y,z]=yz-(-1)^{\abs{y}\abs{z}}zy\). Let $\frakh \subset \gl_{m|n}$ be the Cartan subalgebra consisting of
all diagonal matrices, and let  \(\{h_i \suchthat i \in I\}\) be its standard basis.  Let $\{\epsilon_i \suchthat i \in I\}$ be
the corresponding dual basis of $\frakh^*$.
We define a non-degenerate
symmetric bilinear form on $\frakh^*$ by setting on the basis
\begin{equation}
  \label{alex:eq:47}
  (\epsilon_i,\epsilon_j) = (-1)^{\abs{i}}\delta_{ij},
\end{equation}
where as usual \(\delta_{ij}\) denotes the Kronecker delta.
We denote by \(\sfP\) the \emph{weight lattice} and by \(\sfP^*\) its dual: they are the free \(\Z\)--modules spanned by the \(\epsilon_i\)'s and the \(h_i\)'s, respectively.

The roots of \(\gl_{m|n}\) are \(\{\epsilon_i - \epsilon_j \suchthat i,j \in I\}\). A root is said to be \emph{even} (respectively, \emph{odd}) if the corresponding root vector is. We fix the  standard triangular decomposition, with  positive roots  \(\{\epsilon_i - \epsilon_j \suchthat i < j\}\) and  simple roots  \(\{ \alpha_i = \epsilon_i - \epsilon_{i+1} \suchthat i \in I^0\}\). 
We define  the simple coroots
\begin{equation}
\alpha_i^\down = (-1)^{\abs{i}} h_i - (-1)^{\abs{i+1}} h_{i+1} \qquad \text{for all }i \in I^0.\label{eq:8}
\end{equation}

We let \(\rho_0\) denote half the sum of the even positive roots and \(\rho_1\) denote half the sum of the odd positive roots, and we set \(\rho=\rho_0 - \rho_1\). Explicitly we have
\begin{multline}
  \label{eq:9}
  2\rho = (d-1)\epsilon_1 + (d-3) \epsilon_2 + \dotsb + (d-2m+1) \epsilon_m \\+ (d+2n-1)\epsilon_{m+1} + (d+2n-3) \epsilon_{m+2}+  \dotsb   + (d+1) \epsilon_{m+n}.
\end{multline}

\subsection{The quantum enveloping superalgebra}
\label{alex:sec:quant-envel-super-1}

The \emph{quantum enveloping superalgebra}
$U_q=U_q(\gl_{m|n})$ is defined to be the unital superalgebra
over $\C (q)$ with generators $E_i$, $F_i$, \(\quantumq^h\)  for $i \in I, h
\in \sfP^*$
subject to the following relations: 
\vspace{\abovedisplayskip}
\begin{subequations}
\begingroup
\setlength{\belowdisplayskip}{\jot}
\setlength{\belowdisplayshortskip}{\jot} 
\setlength{\abovedisplayskip}{0pt}
\setlength{\abovedisplayshortskip}{0pt}
  \begin{align}
    \quantumq^0&=1, & \quantumq^h \quantumq^{h'} & = \quantumq^{h+h'},\label{eq:12}\\
    \quantumq^h E_i &=q^{\langle h, \alpha_i \rangle} E_i \quantumq^h, & \quantumq^h F_i&=q^{-\langle h,\alpha_i \rangle}F_i\quantumq^h,
\label{eq:13}
\end{align}
\begin{align}(-1)^{\abs{i}} E_iF_i - (-1)^{\abs{i+1}} F_iE_i =  \frac{K_i-K_i^{-1}}{q-q^{-1}},\label{eq:14}
\end{align}
\begin{align}
E_m^2 & =F_m^2=0,\label{eq:15}\\
E_iE_j=E_jE_i &\text{ and } F_iF_j = F_jF_i &&\text{if } \abs{i-j}\geq 2,\label{eq:16}\\
E_iF_j&=F_jE_i &&\text{if } i \neq j,\label{eq:17}
\end{align}
\begin{align}
E_i^2 E_{i+1} - [2] E_i E_{i+1} E_i + E_{i+1} E_i^2 &= 0 &\text{ for } i,i+1 \neq m, \label{eq:18}\\
E_{i+1}^2 E_{i} - [2] E_{i+1} E_{i} E_{i+1} + E_{i} E_{i+1}^2& = 0 &\text{ for } i,i+1 \neq m,\label{eq:19}\\
F_i^2 F_{i+1} - [2] F_i F_{i+1} F_i + F_{i+1} F_i^2 &= 0 &\text{ for } i,i+1 \neq m,\label{eq:20}\\
F_{i+1}^2 F_{i} - [2] F_{i+1} F_{i} F_{i+1} + F_{i} F_{i+1}^2 &= 0 &\text{ for } i,i+1 \neq m,\label{eq:21}
\end{align}
\begin{gather}
\begin{multlined}
  E_m E_{m-1} E_m E_{m+1} + E_{m-1} E_m E_{m+1} E_m + E_m
  E_{m+1} E_m E_{m-1} \\+ E_{m+1} E_m E_{m-1} E_m - [2] E_m
  E_{m-1} E_{m+1} E_m = 0,
\end{multlined}
\label{eq:22}\\
\begin{multlined}
  F_m F_{m-1} F_m F_{m+1} + F_{m-1} F_m F_{m+1} F_m + F_m
  F_{m+1} F_m F_{m-1} \\+ F_{m+1} F_m F_{m-1} F_m - [2] F_m
  F_{m-1} F_{m+1} F_m = 0,
\end{multlined}
\label{eq:23}
\end{gather}
\endgroup
\end{subequations}
\vspace{\belowdisplayskip}
where \(K_i = q^{\alpha_i^{\down}}\).

We define a \emph{comultiplication} $\Delta\colon U_q \mapto U_q \otimes U_q$, a \emph{counit} $\counit\colon U_q \mapto \C(q)$ and an \emph{antipode} $S\colon U_q \mapto U_q$ by setting on the generators
\begin{equation}
\begin{aligned}
  \Delta(E_i)&= E_i \otimes K_i^{-1}+1 \otimes E_i, & \Delta(F_i)&=F_i \otimes 1 + K_i \otimes F_i,\\
  S(E_i)&=-E_iK_i, & S(F_i)&=- K_i^{-1}F_i,\\
  \Delta(\quantumq^h)&=\quantumq^h \otimes \quantumq^h, &   S(\quantumq^h)&=\quantumq^{-h},\\
  \counit(E_i)& =\counit(F_i)=0, & \counit(\quantumq^h)&=1,
\end{aligned}\label{eq:24}
\end{equation}
and extending $\Delta$ and $\counit$ to superalgebra homomorphisms and $S$ to a superalgebra anti-homomorphism.
As well-known (and easy to check),  the maps \(\Delta\), \(\counit\) and \(S\) turn \(U_q(\gl_{m|n})\) into a Hopf superalgebra.

\subsection{Representations}
\label{sec:representations}

We let \(\C_q=\C(q)\) be the trivial representation of \(\smash{U_q(\gl_{m|n})}\). We denote by \(V=\C^{m|n}_q\) its natural (vector) representation. It is a \(\C(q)\)--vector superspace of homogeneous basis 
\(x_i\) for \(i \in I\), with grading  given by \(\abs{x_i}=\abs{i}\). (Note that by a slight abuse of notation we also denoted by \(x_i\) the  basis vectors of \(\C^{m|n}\).) The action of \(U_q(\gl_{m|n})\) is determined by
\begin{equation}
\label{eq:26}
  E_i x_j = \delta_{i+1,j} x_i, \qquad  F_i x_j=\delta_{i,j} x_{i+1}, \qquad \quantumq^h x_i =q^{\langle h,\epsilon_i \rangle} x_i.
\end{equation}

A weight \(\lambda \in \sfP\) is said to be \emph{dominant} if it is a dominant weight for \(\gl_m \oplus \gl_n \subset \gl_{m|n}\). Given a dominant weight \(\lambda \in \sfP\), one can define the quantum Kac module \(K_q(\lambda)\) with highest weight \(\lambda\) by parabolic induction.
 If \(\lambda\) is typical (that is, if \((\lambda + \rho, \alpha)\neq 0\) for all odd roots \(\alpha\)), then \(K_q(\lambda)\) is simple, and we denote it by  \(L_{q,{m|n}}(\lambda)\) or also simply by \(L_q(\lambda)\) (see \cite[Section~2]{MR3159079}). If \(\lambda\) is not typical, one can define \(L_q(\lambda)\) as the unique simple quotient of \(K_q(\lambda)\). Anyway, we will only need the typical case.

We recall that a partition of \(N\) is a  sequence of positive integers \(\uplambda=(\uplambda_1,\uplambda_2,\dotsc,\uplambda_\ell)\)  for some \(\ell \geq 0\) with \(\uplambda_1 \geq \uplambda_2 \geq \dotsb \geq \uplambda_\ell\) such that \( \abs{\uplambda} = \uplambda_1 + \dotsb + \uplambda_\ell = N\). To a partition we associate a Young diagram, which we also denote by \(\uplambda\), as in the following picture:
\begin{equation}
  \begin{tikzpicture}[scale=0.4,every node/.style={font=\footnotesize},anchorbase]
    \draw (0,0) rectangle (5,-1);
    \draw (0,-1) rectangle (4,-2);
    \draw (0,-3.5) rectangle (2,-4.5);
    \node at (1,-2.5) {$\vdots$};
    \node at (2.5,-0.5) {$\uplambda_1$};
    \node at (2,-1.5) {$\uplambda_2$};
    \node at (1,-4) {$\uplambda_\ell$};
  \end{tikzpicture}
\end{equation}
The transposed \(\uplambda^T\) of a partition \(\uplambda\) is obtained by taking the symmetric Young diagram around the diagonal. In formulas, \(\uplambda^T_i = \#\{h \suchthat \uplambda_h\geq i \}\). 

A partition is said to be a \((m,n)\)--hook if \(\uplambda_i \leq n\) for all \(i > m\), or equivalently if its Young diagram fits into the union of two orthogonal strips of width \(m\) and \(n\), as in the following picture:
\begin{equation}
  \label{eq:28}
  \begin{tikzpicture}[scale=0.4,every node/.style={font=\footnotesize},anchorbase]
    \draw[dashed] (0,0) -- ++(7,0);
    \draw[dashed] (3,-3) -- ++(4,0);
    \draw[<->] (6.5,0) -- node[fill=white] {$m$} ++(0,-3);
    \draw[dashed] (3,-3) -- ++(0,-2);
    \draw[dashed] (0,0) -- ++(0,-5);
    \draw[<->] (0,-4.5) -- node[fill=white] {$n$} ++(3,0); 
    \draw (0,0) rectangle (5,-1);
    \draw (0,-1) rectangle (4,-2);
    \draw (0,-2) rectangle (2,-3);
    \draw (0,-3) rectangle (1,-4);
  \end{tikzpicture}
\end{equation}
Let \(H_{m|n}\) denote the set of all \((m,n)\)--hook partitions. Let \(\uplambda=(\uplambda_1,\dotsc,\uplambda_\ell) \in H_{m|n}\) and write \(\uplambda\) as concatenation of two partitions \(\uplambda'=(\uplambda_1,\dotsc,\uplambda_m)\) and \(\uplambda^{\prime\prime T}\). Notice that \(\uplambda^{\prime \prime}\) is a partition with at most \(n\) parts. Then the weight
\begin{equation}
  \label{eq:29}
  \lambda= \uplambda'_1 \epsilon_1 + \dotsb+ \uplambda'_m \epsilon_m + \uplambda''_1 \epsilon_{m+1} + \dotsb + \uplambda''_{n} \epsilon_{m+n}
\end{equation}
is typical and we let \(L_q(\uplambda) = L_q(\lambda)\). As a convention, if \(\uplambda\) is a partition which is not in \(H_{m|n}\), then we just set \(L_{q,m|n}(\uplambda)=0\).

\subsection{Ribbon structure}
\label{sec:ribbon-structure-1}

As well-known, the \(\hbar\)--version \(U_\hbar(\gl_{m|n})\) of the quantum enveloping algebra of \(\gl_{m|n}\) can be endowed with a ribbon structure, and this allows to turn the category of finite-dimensional \(U_q(\gl_{m|n})\)--representations into a ribbon category. Actually, in the literature it is more common to find the ribbon structure of \(U_\hbar (\mathfrak{sl}_{m|n})\), see 
for example \cite{MR2171805}, \cite{MR2640994}. For choosing a universal \(R\)--matrix for \(\gl_{m|n}\) one has an additional degree of freedom, due to the presence of a nontrivial center. This allows to obtain slightly nicer formulas. We explain our conventions in this subsection. As a reference for ribbon algebras and \(\hbar\)--versions of the quantum enveloping algebras, see for example \cite{MR1300632} or \cite{MR1321145}.

Let \(R'\) be the universal \(R\)--matrix for \(\mathfrak{sl}_{m|n}\) from \cite{MR2640994}. As \(R\)--matrix for \(\gl_{m|n}\) we take \(\smash{R=q^{\frac{1}{2} I \otimes I} (R')^{-1}}\), where \(I = h_1+\dotsb+h_{m+n} \in \sfP^*\) is the identity matrix. Using the ribbon structure of \(U_\hbar(\mathfrak{sl}_{m|n})\) it is straightforward to check that \(R\) defines a braided structure on \(U_\hbar(\gl_{m|n})\). Let also \(v'\) be the ribbon element from \cite{MR2640994}. Then by setting \(\smash{v = q^{\frac{1}{2} I^2} v'} \) we get a ribbon element for our braided structure of \(U_\hbar(\gl_{m|n})\), and hence a ribbon structure of the latter.

\subsection{Braiding and twist}
\label{sec:r-matrix}

On the simple finite-dimensional representation \(L_q(\lambda)\) the twist operator acts as
\begin{equation}
  \label{eq:35}
  \theta_{L_q(\lambda)}(w) = q^{-(\lambda,\lambda+2\rho)} w,
\end{equation}
see also \cite[Lemma~3.1]{MR2640994}.\footnote{Notice that in \cite{MR2640994} \((\cdot,\cdot)\) denotes the bilinear form of the (dual of) the Cartan of \(\mathfrak{sl}_{m|n}\), while in \eqref{eq:35} we denoted by \((\cdot,\cdot)\)  the bilinear form of the (dual of) the Cartan of \(\smash{\gl_{m|n}}\). The fact that we get apparently the same formula is due to our rescaling of the universal \(R\)--matrix, see above.} As for  the braiding, we will only need the explicit action on tensor powers of the vector representation. According to our conventions, this is given by
\begin{equation}
  \label{eq:30}
  \check R_{V,V} (x_a \otimes x_b) =
  \begin{cases}
    q^{-1} x_a \otimes x_a & \text{if } a = b \leq m,\\
    (-1)^{\abs{a}\abs{b}} x_b \otimes x_a & \text{if } a < b,\\
    (-1)^{\abs{a}\abs{b}} x_b \otimes x_a + (q^{-1}- q) x_a \otimes x_b & \text{if } a > b,\\
    -q x_a \otimes x_a & \text{if } a = b > m,
  \end{cases}
\end{equation}
see also \cite{MR2491760}. 

 The map \(\check R_{V,V}\) can be used in a straightforward way to define an action of the Hecke algebra \(\ucalH_r\) (see section~\ref{sec:overview}) on \(V^{\otimes r}\). We recall the following result, also known as super Schur-Weyl duality:

 \begin{prop}[\cite{MR2251378}]
   \label{prop:1}
   The actions of \(\ucalH_r\) and \(U_q(\gl_{m|n})\) on \(V^{\otimes r}\) commute with each other and generate each other's centralizer. Moreover, we have the following decomposition as a module over \(U_q(\gl_{m|n}) \otimes \ucalH_r\):
   \begin{equation}
     \label{eq:32}
     V^{\otimes r} \cong \bigoplus_{\substack{\uplambda \in H_{m|n}\\\abs{\uplambda}=r}} L_q(\uplambda) \otimes S(\uplambda)
   \end{equation}
where \(S(\uplambda)\) is the simple \(\ucalH_r\)--module corresponding to the partition \(\uplambda\).
 \end{prop}

We stress that it follows in particular that \(V^{\otimes r}\) is semisimple.

\subsection{Braided symmetric and exterior powers}
\label{sec:braid-symm-exter}

Let \(W\) be any finite dimensional representation of \(U_q(\gl_{m|n})\), and let \(\bigT W= \bigoplus_{\ell \geq 0} \bigotimes^\ell W\) denote its tensor algebra.    We recall from \cite{MR2386232} that  one can define the braided exterior power \(\bigwedgeq^\bullet W\) of \(W\) as the quotient of \(\bigT W\) modulo the ideal generated by the \(\checkR_{W,W}\)--eigenvectors of \(\bigotimes^2 W\) with positive eigenvalues. Analogously,  one can define the braided symmetric power \(\bigsymmq^\bullet W\) of \(W\) as the quotient of \(\bigT W\) modulo the ideal generated by the \(\checkR_{W,W}\)--eigenvectors of \(\bigotimes^2 W\) with negative eigenvalues. Here, the positive and negative eigenvalues of \(\checkR_{W,W}\) are \(\{q^i \suchthat i \in \Z\}\) and \(\{-q^i \suchthat i \in \Z\}\), respectively.

The braided exterior and symmetric algebras are naturally \(\Z\)--graded (note that this is an additional grading to the \(\Z/2\Z\)--grading, which in this case we will call \(\Z/2\Z\)--parity to avoid confusion). Moreover, their graded dimensions are smaller or equal to the graded dimensions of the usual exterior and symmetric algebras. If the equality holds, we say that \(W\) is \emph{flat}.

\begin{remark}\label{rem:9}
It is easy to see that if we forget the \(\Z/2\Z\)--parity, then we have an isomorphism between \(\bigwedgeq^\bullet (W \langle 1 \rangle)\) and \(\bigsymmq^\bullet W\), where \(\langle 1 \rangle\) denotes the shift by one in the \(\Z/2\Z\)--parity.
For \(p \in \Z/2\Z\) and \(j \in \Z_{\geq 0}\) let \((\bigsymmq^j W)_{p}\) denote the component of \(\bigsymmq^\bullet W\) with \(\Z/2\Z\)--parity \(p\) and \(\Z\)--degree \(j\), so that  \(\bigsymmq^\bullet (W) = \bigoplus_{p \in \Z_2, j \in \Z_{\geq 0}} (\bigsymmq^j W)_{p}\).
If we define \(\mathbb{P} \bigsymmq^\bullet W\) to have graded components \((\mathbb{P} \bigsymmq^\bullet W)_{p,j} = (\bigsymmq^j W)_{p+j}\) then we have an honest graded isomorphism \(\bigwedgeq^\bullet (W \langle 1 \rangle) \cong \mathbb{P} \bigsymmq^\bullet W\) which also preserves the \(\Z/2\Z\)--parity.
\end{remark}

In the next subsection, we will analyze further braided exterior powers of \(V\).

\subsection{Exterior powers}
\label{sec:exterior-powers}

We define \(I\) to be the vector subspace of \(\bigotimes^2 V\) spanned by
\begin{equation}
  \label{eq:36}
  \begin{aligned}
   & x_a \otimes x_a &&  \text{for } a \leq m,\\
   & x_a \otimes x_b + (-1)^{\abs{a}\abs{b}} q^{-1} x_b \otimes x_a && \text{for } a < b.
  \end{aligned}
\end{equation}
Notice that the braiding \(\check R_{V,V}\) has minimal polynomial \((t-q^{-1})(t+q)\) and by construction \(I\) is the eigenspace with eigenvalue \(q^{-1}\). In particular, \(I\) is a \(U_q(\gl_{m|n})\)--subrepresentation of \(\smash{\bigotimes^2 V}\).

Let \(\bigT V= \bigoplus_{\ell \geq 0} \bigotimes^\ell V\) be the tensor algebra of \(V\). We define  \(\bigwedgeq^\bullet V = \bigT V / \langle I \rangle\): it is the free algebra on generators \(x_a\) with multiplication \(\wedge_q\) modulo the relations
\begin{equation}
  \label{eq:38}
  \begin{aligned}
    x_a \wedgeq x_a &= 0 &&\text{for } a \leq m,\\
    x_a \wedgeq x_b &= -(-1)^{\abs{a}\abs{b}} q^{-1} x_b \wedgeq x_a &&\text{for } a<b.
  \end{aligned}
\end{equation}
Since \(\bigT V\) is a \(U_q(\gl_{m|n})\)--representation and \(\langle  I\rangle\) is an invariant subspace, it follows that \(\bigwedgeq^\bullet V\) is naturally a \(U_q(\gl_{m|n})\)--representation, and it decomposes as \(\bigoplus_{i=0}^\infty \bigwedgeq^i V\). It follows from proposition~\ref{prop:1}, or it is easy to check explicitly, that each \(\bigwedgeq^i V\) is irreducible, and
\begin{equation}
  \label{eq:39}
  \bigwedgeq^i V \cong
  \begin{cases}
    L_q(\epsilon_1+\dotsb+\epsilon_i) & \text{if } i \leq m,\\
    L_q(\epsilon_1+\dotsb+\epsilon_m+(i-m)\epsilon_{m+1}) & \text{if } i > m \text{ and } n \geq 1,\\
    0 & \text{if } i > m \text{ and } n=0.
  \end{cases}
\end{equation}
We will often write \(\bigwedgeq^i \) instead of \(\bigwedgeq^i V\) for shortness.

\begin{definition}\label{def:1}
   We define \(\smash{\catRep^+_{m|n}}\)
  to be the full additive subcategory of representations of
  \(U_q(\gl_{m|n})\)
  monoidally generated by the \(\bigwedgeq^i V\)'s
  for \(i \geq 0\).
  We define moreover \(\catRep_{m|n}\)
  to be the full additive subcategory of representations of
  \(\smash[b]{U_q(\gl_{m|n})}\)
  monoidally generated by the \(\bigwedgeq^i V\)'s
  for \(i \geq 0\) and by their duals.
\end{definition}

\subsection{Reshetikhin-Turaev invariant}
\label{sec:ribbon-category}

Since the category of \(U_q(\gl_{m|n})\)--representations is ribbon, given a tangle diagram, we can label its strands by finite dimensional \(U_q(\gl_{m|n})\)--representations and associate to it some equivariant map which is an invariant of the given tangle (see \cite{MR1036112,MR1881401} for the details).

We will not actually need the full ribbon calculus of this category; it will be enough for our purposes to restrict to labelings of the strands by representations \(\bigwedgeq^a V\) for \(a \geq 1\). In particular, we will just label our strands by positive integers, and to a strand labeled by \(a\) we associate the representation \(\bigwedgeq^a V\). In this way we get the Reshetikhin-Turaev functor \(\smash{\RT_{m|n} \colon \catTangles^\lab \mapto \catRep_{m|n}}\).

We will sometimes use the same picture for an element of \(\catTangles^\lab_q\) and for its image under the functor \(\RT_{m|n}\). This should cause no confusion.
For example, the braiding of \(V \otimes V\) and its inverse will be represented by
\begingroup
\begin{equation}
  \label{eq:40}
\tikzset{frecciona/.style={double,gray,->,double distance=0.5mm},every picture/.style={yscale=0.5}}
      \begin{tikzpicture}[smallnodes,anchorbase]
\node[gray,below] at (1,0) {\(V\)};
\node[gray,below] at (0,0) {\(V\)};
\node[gray,below] at (0.5,0) {\(\vphantom{V}\otimes\)};
\node[gray,above] at (1,2) {\(V\)};
\node[gray,above] at (0,2) {\(V\)};
\node[gray,above] at (0.5,2) {\(\vphantom{V}\otimes\)};
\draw[frecciona] (0.5,0) -- ++(0,2);
   \draw[uno,cross line] (1,0)  -- ++(0,0.3) \midarrow .. controls ++(0,0.7) and ++(0,-0.7) .. ++(-1,1.4)  node[left] {$1$} -- ++(0,0.3)  \midarrow;
   \draw[uno,cross line] (0,0)  -- ++(0,0.3) \midarrow .. controls ++(0,0.7) and ++(0,-0.7) .. ++(1,1.4)  node[right] {$1$}  -- ++(0,0.3)  \midarrow;
    \end{tikzpicture} \qquad \text{and} \qquad
    \begin{tikzpicture}[smallnodes,anchorbase]
\node[gray,below] at (1,0) {\(V\)};
\node[gray,below] at (0,0) {\(V\)};
\node[gray,below] at (0.5,0) {\(\vphantom{V}\otimes\)};
\node[gray,above] at (1,2) {\(V\)};
\node[gray,above] at (0,2) {\(V\)};
\node[gray,above] at (0.5,2) {\(\vphantom{V}\otimes\)};
\draw[frecciona] (0.5,0) -- ++(0,2);
   \draw[uno,cross line] (0,0)  -- ++(0,0.3)  \midarrow .. controls ++(0,0.7) and ++(0,-0.7) .. ++(1,1.4) node[right] {$1$} -- ++(0,0.3)  \midarrow;
   \draw[uno,cross line] (1,0)  -- ++(0,0.3) \midarrow .. controls ++(0,0.7) and ++(0,-0.7)    .. ++(-1,1.4) node[left] {$1$} -- ++(0,0.3)  \midarrow;
    \end{tikzpicture}
\end{equation}
\endgroup
respectively.
We also have evaluation and coevaluation maps, which for \(V\) and \(V^*\) are graphically represented by
\begin{equation}
\label{eq:41}
\tikzset{frecciona/.style={double,gray,->,double distance=0.5mm}}
  \begin{tikzpicture}[smallnodes,anchorbase]
\node[gray,below] at (1,0) {\(V^*\)};
\node[gray,below] at (0,0) {\(V\)};
\node[gray,below] at (0.5,0) {\(\vphantom{V^*}\otimes\)};
\node[gray,above] at (0.5,0.7) {\(\C_q\)};
\draw[frecciona] (0.5,0) -- ++(0,0.7);
    \draw[uno,cross line] (0,0) .. controls ++(0,+0.5) and ++(0,+0.5) .. ++(1,0)  \nendarrow node[near end,above right] {$1$} ;
  \end{tikzpicture}\qquad
  \begin{tikzpicture}[smallnodes,anchorbase]
\node[gray,below] at (1,0) {\(V\)};
\node[gray,below] at (0,0) {\(V^*\)};
\node[gray,below] at (0.5,0) {\(\vphantom{V^*}\otimes\)};
\node[gray,above] at (0.5,0.7) {\(\C_q\)};
\draw[frecciona] (0.5,0) -- ++(0,0.7);
    \draw[uno,cross line] (0,0) .. controls ++(0,+0.5) and ++(0,+0.5) .. ++(1,0)  \nstartarrowrev node[near start,above left] {$1$} ;
  \end{tikzpicture}\qquad
  \begin{tikzpicture}[smallnodes,anchorbase]
\node[gray,above] at (1,1) {\(V\)};
\node[gray,above] at (0,1) {\(V^*\)};
\node[gray,above] at (0.5,1) {\(\vphantom{V^*}\otimes\)};
\node[gray,below] at (0.5,0.3) {\(\C_q\)};
\draw[frecciona] (0.5,0.3) -- ++(0,0.7);
    \draw[uno,cross line] (0,1) .. controls ++(0,-0.5) and ++(0,-0.5) .. ++(1,0)  \nendarrow node[near end,below right] {$1$} ;
  \end{tikzpicture} \qquad 
  \begin{tikzpicture}[smallnodes,anchorbase]
\node[gray,above] at (1,1) {\(V^*\)};
\node[gray,above] at (0,1) {\(V\)};
\node[gray,above] at (0.5,1) {\(\vphantom{V^*}\otimes\)};
\node[gray,below] at (0.5,0.3) {\(\C_q\)};
\draw[frecciona] (0.5,0.3) -- ++(0,0.7);
    \draw[uno,cross line] (0,1) .. controls ++(0,-0.5) and ++(0,-0.5) .. ++(1,0)  \nstartarrowrev node[near start,below left] {$1$} ;
  \end{tikzpicture}
\end{equation}

We recall the following relations in \(\catRep_{m|n}\), which are a consequence of \eqref{eq:35} (cf.\ also \cite[Lemma~3.1]{MR2640994}):
\begingroup
\tikzset{every picture/.style={yscale=0.7}}
  \begin{equation}\label{eq:42}
    \begin{tikzpicture}[smallnodes,anchorbase]
   \draw[uno] (0.8,1)  .. controls ++(0,-0.3) and ++(0.3,0) .. ++(-0.3,-0.6)  .. controls ++(-0.5,0) and ++(0,-0.7) .. ++(-0.5,1.6) \nendarrow node[yshift=-0.1cm,right] {$1$} ;
   \draw[uno, cross line] (0.8,1)  .. controls ++(0,0.3) and ++(0.3,0) .. ++(-0.3,0.6) \startarrowrev .. controls ++(-0.5,0) and ++(0,0.7) .. ++(-0.5,-1.6) \nendarrowrev;
    \end{tikzpicture} \; = q^{-d} \;
    \begin{tikzpicture}[smallnodes,anchorbase]
      \draw[uno] (0,0) 
      -- ++(0,2) \midarrow node[yshift=-0.1cm,right] {$1$};
    \end{tikzpicture}
\qquad \text{and} \qquad
    \begin{tikzpicture}[smallnodes,anchorbase]
   \draw[uno, cross line] (0.8,1)  .. controls ++(0,0.3) and ++(0.3,0) .. ++(-0.3,0.6) \startarrowrev .. controls ++(-0.5,0) and ++(0,0.7) .. ++(-0.5,-1.6) \nendarrowrev;
   \draw[uno, cross line] (0.8,1)  .. controls ++(0,-0.3) and ++(0.3,0) .. ++(-0.3,-0.6)  .. controls ++(-0.5,0) and ++(0,-0.7) .. ++(-0.5,1.6) \nendarrow node[yshift=-0.1cm,right] {$1$} ;
    \end{tikzpicture} \; = q^{+d} \;
    \begin{tikzpicture}[smallnodes,anchorbase]
      \draw[uno] (0,0) 
 -- ++(0,2) \midarrow node[yshift=-0.1cm,right] {${1}$};
    \end{tikzpicture},
  \end{equation}
and more generally
  \begin{equation}\label{eq:43}
    \begin{tikzpicture}[smallnodes,anchorbase]
   \draw[int] (0.8,1)  .. controls ++(0,-0.3) and ++(0.3,0) .. ++(-0.3,-0.6)  .. controls ++(-0.5,0) and ++(0,-0.7) .. ++(-0.5,1.6) \nendarrow node[yshift=-0.1cm,right] {${a}$} ;
   \draw[int, cross line] (0.8,1)  .. controls ++(0,0.3) and ++(0.3,0) .. ++(-0.3,0.6) \startarrowrev .. controls ++(-0.5,0) and ++(0,0.7) .. ++(-0.5,-1.6) \nendarrowrev;
    \end{tikzpicture} \; = q^{-ad+a(a-1)} \;
    \begin{tikzpicture}[smallnodes,anchorbase]
      \draw[int] (0,0) 
      -- ++(0,2) \midarrow node[yshift=-0.1cm,right] {${a}$};
    \end{tikzpicture}
\qquad \text{and} \qquad
    \begin{tikzpicture}[smallnodes,anchorbase]
   \draw[int, cross line] (0.8,1)  .. controls ++(0,0.3) and ++(0.3,0) .. ++(-0.3,0.6) \startarrowrev .. controls ++(-0.5,0) and ++(0,0.7) .. ++(-0.5,-1.6) \nendarrowrev ;
   \draw[int, cross line] (0.8,1)  .. controls ++(0,-0.3) and ++(0.3,0) .. ++(-0.3,-0.6)  .. controls ++(-0.5,0) and ++(0,-0.7) .. ++(-0.5,1.6) \nendarrow node[yshift=-0.1cm,right] {${a}$} ;
    \end{tikzpicture} \; = q^{ad-a(a-1)} \;
    \begin{tikzpicture}[smallnodes,anchorbase]
      \draw[int] (0,0) 
      -- ++(0,2) \midarrow node[yshift=-0.1cm,right] {${a}$};
    \end{tikzpicture}.
  \end{equation}
\endgroup
Moreover, we have
\begin{equation}
  \label{eq:44}
    \begin{tikzpicture}[smallnodes,anchorbase]
      \draw[uno] (0,0) arc (0:360:0.5cm) node[midway,left]  {$1$};
      \draw [uno] (0,-.01) -- (0,.01) \midarrowrev;
    \end{tikzpicture} \; = [d] ,
\end{equation}
where \(d=m-n\) is the superdimension of \(\C^{m|n}\). We stress that these are not equalities in \(\catTangles_q\), but only in \(\catRep_{m|n}\). 

\section{Super skew Howe duality}
\label{sec:super-skew-howe}

We consider now simultaneously two quantum superalgebras \(\smash{U_q(\gl_{m|n})}\) and \(\smash{U_q(\gl_{k|l})}\) for some \(m,n,k,l \geq 0\). We let \(\smash{V_1=\C_q^{m|n}}\) and \(\smash{V_2=\C_q^{k|l}}\) be their vector representations, with bases \(\{x_a \suchthat 1 \leq a \leq m+n\}\) and \(\{y_b \suchthat 1 \leq b \leq k+l\}\) respectively. 

We will consider the representation \(V_1 \otimes V_2\) of \(U_q(\gl_{m|n}) \otimes U_q(\gl_{k|l})\), whose standard basis we denote by \(z_{ab} = x_a \otimes y_b\). Our first goal is to define exterior powers of this representation, generalizing \cite{MR2386232}. Notice that the tensor product algebra \(U_q(\gl_{m|n}) \otimes U_q(\gl_{k|l})\) is the quantum enveloping superalgebra of \(\gl_{m|n} \oplus \gl_{k|l}\). It follows by general theory that its representation category is braided, and its universal \(R\)--matrix is the tensor product of the two \(R\)--matrices of \(U_q(\gl_{m|n})\) and \(U_q(\gl_{k|l})\).

Let  \(\sigma_{23} \colon V_1 \otimes V_1 \otimes V_2 \otimes V_2 \mapto V_1 \otimes V_2 \otimes V_1 \otimes V_2\) be the superpermutation which swaps the two middle tensor factors.
The operator \(\calR = \sigma_{23} \circ (\checkR_{V_1,V_1} \otimes \checkR_{V_2,V_2}) \circ \sigma_{23}^{-1}\) is then the braiding on \(V_1 \otimes V_2\). (Although this follows by general statements, it is  straightforward to see directly that it satisfies the braid relation.)

Since the minimal polynomial of both \(\checkR_{V_1,V_1}\) and of \(\checkR_{V_2,V_2}\) is \((t-q^{-1})(t+q)\), the minimal polynomial of \(\calR\) is \((t-q^{-2})(t-q^2)(t+1)\).  
We set \(I\) to be  the sum of the eigenspaces of \(\calR\) with eigenvalues \(q^{-2}\) and \(q^2\) (i.e.\ with \emph{positive} eigenvalues). It is easy to see that \(I\) is spanned by
\begin{equation}
\label{eq:56}
  \begin{aligned}
   & z_{ab} \otimes z_{ab} &&  \text{for } \abs{a}=\abs{b},\\
   & z_{ac} \otimes z_{bc} + (-1)^{\abs{a}\abs{b}} q^{-1} z_{bc} \otimes z_{ac}  &&  \text{for } a < b,\abs{c}=0,\\
   & z_{ac} \otimes z_{ad} + (-1)^{\abs{c}\abs{d}} q^{-1} z_{ad} \otimes z_{ac}  &&  \text{for } \abs{a}=0, c<d,\\
   & \begin{multlined}[b]z_{ac} \otimes z_{bd} + (-1)^{\abs{a}\abs{b}} q^{-1} z_{bc} \otimes z_{ad} \\ + (-1)^{\abs{c}\abs{d}} q^{-1} z_{ad} \otimes z_{bc} + (-1)^{\abs{a}\abs{b}+\abs{c}\abs{d}}q^{-2} z_{bd} \otimes z_{ac} \end{multlined}  &&  \text{for } a < b,c<d,\\
   & z_{ac} \otimes z_{ad} - (-1)^{\abs{c}\abs{d}} q z_{ad} \otimes z_{ac}  &&  \text{for } \abs{a}=1, c<d,\\
   & z_{ac} \otimes z_{bc} - (-1)^{\abs{a}\abs{b}} q z_{bc} \otimes z_{ac}  &&  \text{for } a < b,\abs{c}=1,\\
   & \begin{multlined}[b]z_{ac} \otimes z_{bd} - (-1)^{\abs{a}\abs{b}} q z_{bc} \otimes z_{ad} \\ - (-1)^{\abs{c}\abs{d}} q z_{ad} \otimes z_{bc} + (-1)^{\abs{a}\abs{b}+\abs{c}\abs{d}}q^{2} z_{bd} \otimes z_{ac} \end{multlined}  &&  \text{for } a < b,c<d.
  \end{aligned}
\end{equation}

We let \(\bigwedgeq^\bullet(V_1 \otimes V_2) = \bigT(V_1 \otimes V_2)/\langle I \rangle\):
by definition, this is the quotient of the free algebra on the set \(\{z_{ab}\}\)  with multiplication \(\wedge_q\) modulo the relations
\begin{equation}
  \label{eq:57}
  \begin{aligned}
   & z_{ab} \wedgeq z_{ab} =0 &&  \text{for } \abs{a}=\abs{b},\\
   & z_{ac} \wedgeq z_{bc} = - (-1)^{\abs{a}\abs{b}} q^{-1} z_{bc} \wedgeq z_{ac}  &&  \text{for } a < b,\abs{c}=0,\\
   & z_{ac} \wedgeq z_{ad} = - (-1)^{\abs{c}\abs{d}} q^{-1} z_{ad} \wedgeq z_{ac}  &&  \text{for } \abs{a}=0, c<d,\\
   & z_{ac} \wedgeq z_{bd} = - (-1)^{\abs{a}\abs{b}+\abs{c}\abs{d}} z_{bd} \wedgeq z_{ac} && \text{for } a<b, c<d,\\
   &  \begin{multlined}[b] z_{ac} \wedgeq z_{bd} = - (-1)^{\abs{a}\abs{b}+\abs{c}\abs{d}} z_{bd} \wedgeq z_{ac} \\ + (-1)^{\abs{a}\abs{b}} (q^{-1}-q) z_{bc} \wedgeq z_{ad} \end{multlined} && \text{for } a>b, c<d,\\
   & z_{ac} \wedgeq z_{ad} = (-1)^{\abs{c}\abs{d}} q z_{ad} \wedgeq z_{ac} && \text{for } \abs{a}=1, c<d,\\
   & z_{ac} \wedgeq z_{bc} = (-1)^{\abs{a}\abs{b}} q z_{bc} \wedgeq z_{ac}  &&  \text{for } a < b,\abs{c}=1.\\
  \end{aligned}
\end{equation}

The tensor algebra \(\bigT(V_1 \otimes V_2)\) is  a \(\big(U_q(\gl_{m|n}) \otimes U_q(\gl_{k|l})\big)\)--module, and the ideal \(\langle I\rangle\) is by construction an invariant subspace. In particular, \(\bigwedgeq^\bullet (V_1 \otimes V_2)\) is also a \(\big(U_q(\gl_{m|n}) \otimes U_q(\gl_{k|l})\big)\)--module, and it decomposes as \(\bigoplus_{N=0}^\infty \bigwedgeq^N (V_1 \otimes V_2)\).

We want to show that \(V_1 \otimes V_2\) is flat, that is, the dimension of \(\bigwedgeq^N (V_1 \otimes V_2)\) coincides with the dimension of the usual exterior power \(\bigwedge^N (V_1 \otimes V_2)\). Unfortunately, it is not easy to see this from the presentation~\eqref{eq:57}. Hence we will give now an alternative construction of \(\bigwedgeq^\bullet (V_1 \otimes V_2)\). This will come only with an action of \(U_q(\gl_{m|n})\), and not of \(U_q(\gl_{k|l})\). However, the flatness will be obvious. 

Let \(A,B\) be any locally finite \(U_q(\gl_{m|n})\)--module algebras with multiplications \(\mu_A\) and \(\mu_B\), respectively. We recall from \cite[Section~2]{MR2753673} that  one can define a \(U_q(\gl_{m|n})\)--module algebra structure on \(A \otimes B\) with multiplication
\begin{equation}
  \label{eq:11}
  \mu_{A \otimes B} = (\mu_A \otimes \mu_B) (\id_A \otimes \checkR^{-1}_{A,B} \otimes \id_B),
\end{equation}
This is well-defined (\cite[Theorem~2.3]{MR2753673}), the construction can be iterated with a third module algebra \(C\) and is associative (\cite[Lemma~2.5]{MR2753673}).

We apply this construction iteratively to the algebras \(\bigsymmq^\bullet (V_1\langle 1 \rangle)\) and \(\bigsymmq^\bullet (V_1)\), where \(\langle 1 \rangle\) denotes a shift by one in the \(\Z/2\Z\)--parity, and we define in this way a \(U_q(\gl_{m|n})\)--module algebra structure on
\begin{equation}
A_{k|l} = \big(\bigsymmq^\bullet (V_1\langle 1 \rangle) \big)^{\otimes k} \otimes \big( \bigsymmq^\bullet (V_1) \big)^{\otimes l},\label{eq:25}
\end{equation}
where \(\langle 1 \rangle\) denotes a shift by one in the \(\Z/2\Z\)--degree. As in \cite[Theorem~2.8]{MR2753673}, it follows that \(A_{k|l}\) is a flat deformation of the symmetric algebra \(\bigsymm \big((V_1\langle 1 \rangle)^{\oplus k} \oplus V_1^{\oplus l} )\), i.e.\ of the exterior algebra \(\bigwedge \big( V_1 ^{\oplus k} \oplus (V_1\langle 1 \rangle)^{\oplus l}\big)\). This means that the graded dimensions of \(A_{k|l}\) and of \(\bigwedge \big( V_1 ^{\oplus k} \oplus (V_1\langle 1 \rangle)^{\oplus l}\big)\) coincide.

For all \(a=1,\dots,m+n\) and \(i=1,\dots,k+l\) let \(X_{ai} = 1 \otimes \dots \otimes 1 \otimes v_a \otimes 1 \otimes \dots \otimes 1 \in A_{k|l}\), the entry \(v_a\) being at position \(i\). It is easy to see that the \(X_{ai}\)'s generate \(A_{k|l}\) as an algebra (this follows for example from~\cite[Lemma~2.9]{MR2753673}), and by construction they are subject to the relations
\begin{equation}
  \label{eq:27}
  X_{aj} X_{bi} = -(-1)^{\abs{i}\abs{j}} \sum_{a',b'=1}^{m+n} \checkR^{-1}_{a'a,b'b} X_{a'i}X_{b'j} \qquad \text{for all } i < j.
\end{equation}
Here \(\checkR^{-1}_{a'a,b'b}\) denotes the matrix coefficient of the inverse of the braiding \(\checkR\), and the sign \(- (-1)^{\abs{i}\abs{j}}\) appears since we shifted the first \(k\) copies of \(\bigsymmq^\bullet V_1\) by one in the \(\Z/2\Z\)--degree.

\begin{lemma}
  \label{lem:5}
  The map \(z_{ai} \mapsto X_{ai}\) defines an algebra isomorphism
  \(\bigwedgeq^\bullet (V_1 \otimes V_2) \rightarrow  A_{k|l} \).
  For all \(N \geq 0\), a basis of \(\bigwedgeq^N (V_1 \otimes V_2)\) is given by the elements
  \begin{equation}
    \label{eq:59}
   \left\{ z_{a_1 b_1} \wedgeq \dotsb \wedgeq z_{a_N b_N} \,\left|\,
     \begin{aligned}
       & (a_1,b_1) \leq (a_2,b_2) \leq \dotsb \leq (a_N,b_N)\\
       &\text{in the anti-lexicographic order, and}\\ 
       & \text{if \((a_i,b_i) = (a_{i+1},b_{i+1})\) then \(\abs{z_{a_i b_i}}=1\)}
     \end{aligned}\right.\right\}.
  \end{equation}
\end{lemma}

We point out that this isomorphism does not preserve the \(\Z/2\Z\)--parity. However, one gets a parity-preserving isomorphism  \(\bigwedgeq^\bullet (V_1 \otimes V_2) \rightarrow \mathbb{P} A_{k|l}\), see remark~\ref{rem:9}.

\begin{proof}
  First, let us notice that the relations~\eqref{eq:57} hold for the \(X_{ai}\)'s, and hence the map \(z_{ai} \mapsto X_{ai}\) is well-defined. Indeed, the first two relations hold simply because they hold in \(\bigwedgeq^\bullet (V_1)\), while the other relations are directly obtained from~\eqref{eq:27} (being very careful with signs). 

It is a straightforward computation using the relations~\eqref{eq:57} to show that~\eqref{eq:59} span \(\bigwedgeq^N (V_1 \otimes V_2)\). Since the \(X_{ai}\)'s generate \(A_{k|l}\) as an algebra, the map \(z_{ai} \mapsto X_{ai}\) is surjective and the images of the elements~\eqref{eq:59} span \(A_{k|l}\). On the other side, one checks directly that the cardinality of the set~\eqref{eq:59} coincides with the dimension of \(\bigwedge^N \big( V_1 ^{\oplus k} \oplus (V_1\langle 1 \rangle)^{\oplus l}\big)\), which by construction is the same as the dimension of the graded part of \(A_{k|l}\) of degree \(N\). It follows that the images of the elements~\eqref{eq:59} must be linearly independent and the map \(z_{ai} \mapsto X_{ai}\) is an isomorphism of algebras.
\end{proof}

The following result is the quantum version of super skew Howe duality. The corresponding non-quantized result can be found in \cite[Chapter~5]{MR3012224}. In the non-super case, the result has already been proven in \cite[Theorem~4.2.2]{CKM}.
A generalization to the super case, but only from one side, already appeared in \cite[Theorem~2.2]{MR2510063}.

\begin{theorem}[Quantized super skew Howe duality]
  \label{thm:1}
 For all \(N>0\) the two actions
  \begin{equation}
    \label{eq:60}
    U_q(\gl_{m|n}) \acts \bigwedgeq^N (\C^{m|n}_q \otimes \C^{k|l}_q) \actsb U_q(\gl_{k|l})
  \end{equation}
  commute and generate each other's centralizer. Moreover, we have
  \begin{equation}
    \label{eq:61}
    \bigwedgeq^N (\C^{m|n}_q \otimes \C^{k|l}_q) \cong \displaystyle\bigoplus_{\uplambda \in H(m|n) \cap H(l|k)} L_{q,{m|n}}(\uplambda) \otimes L_{q, {k|l}}( \uplambda^T)
  \end{equation}
  as \(U_q(\gl_{m|n}) \otimes U_q(\gl_{k|l})\)--modules.
\end{theorem}

\begin{proof}
  It is clear that the two actions in \eqref{eq:60} commute with each other. It is then enough to prove the decomposition \eqref{eq:61}, which implies the first assertion. We remark that in the non-quantized case, the analogous  of \eqref{eq:61} holds (see \cite[Theorem~5.18 and Remark~5.20]{MR3012224}).

The space \(W=\bigwedgeq^N(\C^{m|n}_q \otimes  \C^{k|l}_q)\), being a quotient of  \(\bigotimes^N(\C^{m|n}_q \otimes \C^{k|l}_q)\) which is a semisimple \(\big(U_q(\gl_{m|n}) \otimes U_q(\gl_{k|l})\big)\)--module by proposition~\ref{prop:1}, is semisimple. Moreover it must decompose as \(\bigoplus_{\uplambda \in H_{m|n},\upmu \in H_{k|l}} \big(L_{q,m|n}(\uplambda) \otimes L_{q,k|l}(\upmu)\big)^{\bigoplus \kappa_{\uplambda,\upmu}}\). As in finite-dimensional representation theory of semisimple Lie algebras, it is very easy to argue that the multiplicities \(\kappa_{\uplambda,\upmu}\) are uniquely determined by the dimensions of the weight spaces of \(W\), that is by the character of \(W\). The character of \(W\) coincides with the character of \(\bigwedge^N(\C^{m|n} \otimes \C^{k|l})\), hence the multiplicities \(\kappa_{\uplambda,\upmu}\) are the same as in the non-quantized case, and \eqref{eq:61} follows.
\end{proof}

\begin{prop}
  \label{prop:2}
  As a \(U_q(\gl_{m|n})\)--representation, the module \eqref{eq:60} decomposes as
  \begin{equation}
\bigwedgeq^N(\C_q^{m|n} \otimes \C_q^{k|l}) \cong \displaystyle \bigoplus_{\substack{\bolda \in \mathrm{Comp}(N_1,k), \\ \boldb \in \mathrm{Comp}(N_2,l), \\ N_1+N_2=N}} \bigwedgeq^{\bolda} \C_q^{m|n} \otimes \bigwedgeq^{\boldb} (\C_q^{m|n}\langle 1 \rangle) ,\label{eq:62}
\end{equation}
where \(\mathrm{Comp}(M,h) = \{\bolda = (a_1,\dots,a_h)   \suchthat a_i \in \Z_{\geq 0}, \, a_1+\dots+a_k=M\}\) denotes the set of all compositions of \(M\) with \(h\) parts,  
and for \(\bolda=(a_1,\dotsc,a_k)\) set
  \begin{equation}
    \label{eq:63}
        \bigwedgeq^{\bolda} \C_q^{m|n} = \bigwedgeq^{a_1}\C_q^{m|n} \otimes \cdots \otimes \bigwedgeq^{a_k} \C_q^{m|n}.
  \end{equation}
  When  considered as a \(U_q(\gl_k)\)--module,  \eqref{eq:62} is the weight space decomposition, and \(\bigwedgeq^\bolda \C_q^{m|n} \otimes \bigwedgeq^\boldb ( \C_q^{m|n} \langle 1 \rangle)\) is the weight space of weight \((\bolda,\boldb)=a_1 \epsilon_1 + \dotsb + a_k \epsilon_k +b_1 \epsilon_{k+1} + \dotsb + b_l \epsilon_{k+l}\).
\end{prop}

We recall that \(\bigwedgeq^\boldb (\C_q^{m|n} \langle 1 \rangle)\) is just the symmetric braided algebra, but with shifted \(\Z/2\Z\)--parity, see Remark~\ref{rem:9}.

\begin{proof}
  The isomorphism~\eqref{eq:62} is just the isomorphism of \(\bigwedgeq^N(\C_q^{m|n} \otimes \C_q^k)\) with the graded part of \(A_{k|l}\) of degree \(N\). It is straightforward to check that this isomorphism is \(U_q(\gl_{m|n})\)--equivariant, and that the last assertion on the weight spaces holds.
\end{proof}

From now on we  set \(l=0\) and consider the case of \(V_2=\C_q^k\).
We will mainly be interested in the action of \(U_q(\gl_k)\) on the \(U_q(\gl_{m|n})\)--representation \(\smash{\bigoplus_{\bolda} \bigwedge_q^\bolda \C_q^{m|n}}\). Hence from now on the symbols \(E_i,F_i\) will denote the generators of \(U_q(\gl_k)\), and will act on \(\bigoplus_{\bolda} \bigwedge_q^\bolda \C_q^{m|n}\).

 We identify the weight lattice \(\sfP\) of \(\gl_k\) with \(\Z^k\). We can then reinterpret  skew Howe duality by saying that for all \(\bolda \in \N^k\) we have morphisms 
 \begin{align}
   \label{eq:65}
   E_i \colon \bigwedgeq^\bolda \C_q^{m|n} & \longmapto \bigwedgeq^{\bolda + \alpha_i} \C_q^{m|n},\\
   F_i \colon \bigwedgeq^\bolda \C_q^{m|n}& \longmapto \bigwedgeq^{\bolda-\alpha_i} \C_q^{m|n},\\
   \quantumq^h \colon \bigwedgeq^\bolda \C_q^{m|n} & \longmapto \bigwedgeq^{\bolda} \C_q^{m|n},
 \end{align}
satisfying the defining relations of \(U_q(\gl_k)\), where we set \(\bigwedgeq^\bolda \C_q^{m|n}=0\) in case \(a_i <0\) for some \(i=1,\dotsb,k\). By the last claim of proposition~\ref{prop:2}, the element \(\quantumq^h\) acts on \(\bigwedgeq^\bolda \C_q^{m|n}\) by \(q^{\langle h, \bolda\rangle}\).

We notice that
the action given by skew Howe duality is local, in the following sense:

\begin{prop}
  \label{prop:3}
  Fix \(k \geq 2\), let \(\bolda=(a_1,\dotsc,a_k)\) and consider \(\bigwedgeq^\bolda \C_q^{m|n}\). Then we have
  \begin{equation}
    \label{eq:66}
    E_i = \id^{\otimes (i-1)} \otimes E \otimes \id^{\otimes(k-i-1)} 
  \end{equation}
  as morphisms \(\bigwedgeq^\bolda \C_q^{m|n} \mapsto \bigwedgeq^{\bolda + \alpha_i}\C_q^{m|n}\), where
  \begin{equation}
E \colon \bigwedgeq^{a_i}\C_q^{m|n} \otimes \bigwedgeq^{a_{i+1}}\C_q^{m|n} \mapto \bigwedgeq^{a_i+1}\C_q^{m|n} \otimes \bigwedgeq^{a_{i+1}-1}\C_q^{m|n}\label{eq:67}
\end{equation}
is the generator \(E_1 \) of \(U_q(\gl_2)\). The analogous statement holds for the \(F_i\)'s.
\end{prop}

\begin{proof}
  This is straightforward.
\end{proof}

Moreover, we state the following easy result:

\begin{lemma}
  \label{lem:2}
  Let \(k=2\), and consider skew Howe duality applied to \(\bigwedgeq^N(\C_q^{m|n} \otimes \C_q^2)\). Then 
  \begin{align}
E_1^{(N)} &\colon \bigwedgeq^0(\C_q^{m|n}) \otimes \bigwedgeq^N(\C_q^{m|n}) \mapto \bigwedgeq^N(\C_q^{m|n}) \otimes \bigwedgeq^0(\C_q^{m|n})\quad \text{and}\\
F_1^{(N)} & \colon \bigwedgeq^N(\C_q^{m|n}) \otimes \bigwedgeq^0(\C_q^{m|n}) \mapto \bigwedgeq^0(\C_q^{m|n}) \otimes \bigwedgeq^N(\C_q^{m|n})
  \end{align}
 are the identity map after the natural identification of both spaces with \(\bigwedgeq^N(\C_q^{m|n})\).
\end{lemma}

\section{The category \texorpdfstring{$\Spider^+$}{of Spiders}}
\label{sec:graphical-version}

Following \cite{CKM}, we introduce the category \(\Spider^+\) which gives a graphical interpretation of skew Howe duality.

\subsection{The category \texorpdfstring{$\widetilde \Spider^+$}{Sp+}}
\label{sec:categ-widet-spid}

First, we are going to define a monoidal category \(\widetilde\Spider^+\). Objects of \(\widetilde \Spider^+\) are sequences of non-negative integers. The tensor product is given by concatenation of sequences. Morphisms are monoidally generated by identities and by the two elements
\begin{align}
  \label{eq:68}
  \sfE^{(r)}=\sfE^{(r)}\idem_{a \otimes b} \colon & a \otimes b \longmapto a+r \otimes b-r\\
  \sfF^{(r)} = \sfF^{(r)} \idem_{a \otimes b} \colon & a \otimes b \longmapto a-r \otimes b+r\label{eq:69}
\end{align}
for all \(a,b \in \Z_{\geq 0}\) and \(1 \leq r\leq b\) such that \(a+r,b-r \in \Z_{\geq 0}\) or \(a-r,b+r \in \Z_{\geq 0}\), respectively. By abuse of notation, we omit sometimes to write \(\idem_{a \otimes b}\) on the right of \(\sfE^{(r)}\) and \(\sfF^{(r)}\).

Let us denote by \(\varnothing\) the empty sequence, which is the unity of the monoidal structure. We also want to have an isomorphism \(0 \cong \varnothing\) which allows us to identify \(0\)  with the monoidal unity. We will implicitly use this isomorphism to let \(0\)'s appear and disappear.
We impose moreover
that \(\sfE^{(r)}\idem_{0 \otimes r} \colon r \cong 0 \otimes r \mapto r \otimes 0 \cong r\) 
and  \(\sfF^{(r)}\idem_{r \otimes 0}
\colon r \cong r \otimes 0 \mapto 0 \otimes r \cong r\) are the identity morphisms. 

 Graphically, we represent the identity morphism \(a \mapto a\) and the morphisms \eqref{eq:68} and \eqref{eq:69} by
 \begin{equation}\label{eq:70}
\id_{a} = 
\begin{tikzpicture}[int,smallnodes,anchorbase,xscale=1,yscale=1]
   \draw (0,0) node[below] {\(a\)} -- (0,1) node[above] {\(a\)} \midarrow;
 \end{tikzpicture}\;, \qquad
 E^{(r)}\idem_{a \otimes b} = 
\begin{tikzpicture}[int,smallnodes,anchorbase,xscale=1,yscale=1.3]
   \draw (0,0) node[below] {\(a\)} -- (0,1) node[above] {\(a{+}r\)} \nendarrow \nstartarrow;
   \draw (1,0) node[below] {\(b\)} -- (1,1) node[above] {\(b{-}r\)} \nendarrow \nstartarrow;
   \draw (1,0.4) -- node[above] {$r$} (0,0.6) \midarrow;
 \end{tikzpicture}
  \qquad \text{and}\qquad
 F^{(r)} \idem_{a \otimes b} = 
\begin{tikzpicture}[int,smallnodes,anchorbase,xscale=1,yscale=1.3]
   \draw (0,0) node[below] {\(a\)} -- (0,1) node[above] {\(a{-}r\)} \nendarrow \nstartarrow;
   \draw (1,0) node[below] {\(b\)} -- (1,1) node[above] {\(b{+}r\)} \nendarrow \nstartarrow;
   \draw (0,0.4) -- node[above] {$r$} (1,0.6) \midarrow;
 \end{tikzpicture}
\end{equation}
respectively. Notice that we read the pictures always from the bottom to the top. The tensor product of two morphisms is given by horizontal juxtaposition, while composition of morphisms is given by stacking them one on top of the other.

We adopt the following graphical convention: we use thick strands (as in \eqref{eq:70}) when they are labeled by an integer \(a > 0\), we use thin strands when they are labeled by \(1\), and we use dashed strands when they are labeled by \(0\). In particular, \(\sfE \idem_{a \otimes b}=\sfE^{(1)} \idem_{a \otimes b}\) and \(\sfF \idem_{a \otimes b}=\sfF^{(1)} \idem_{a \otimes b}\) are represented by the diagrams
 \begin{equation}
 \begin{tikzpicture}[smallnodes,anchorbase,xscale=1,yscale=1.3]
   \draw[int] (0,0) node[below] {\(\vphantom{b}a\)} -- (0,1) node[above] {\(\mathclap{a{+}1}\)} \nstartarrow \nendarrow;
   \draw[int] (1,0) node[below] {\(b\)} -- (1,1) node[above] {\(\mathclap{b{-}1}\)} \nstartarrow \nendarrow;
   \draw[uno] (1,0.4) -- node[above] {$1$} (0,0.6) \midarrow;
 \end{tikzpicture}
 \qquad \text{and} \qquad
\begin{tikzpicture}[smallnodes,anchorbase,xscale=1,yscale=1.3]
   \draw[int] (0,0) node[below] {\(\vphantom{b}a\)} -- (0,1) node[above] {\(\mathclap{a{-}1}\)} \nstartarrow \nendarrow;
   \draw[int] (1,0) node[below] {\(b\)} -- (1,1) node[above] {\(\mathclap{b{+}1}\)} \nstartarrow \nendarrow;
   \draw[uno] (0,0.4) -- node[above]{$1$} (1,0.6) \midarrow;
 \end{tikzpicture}\;.\label{eq:71}
\end{equation}

We enlarge the category \(\widetilde \Spider^+\) by taking as new homomorphism spaces the \(\fieldk\)--vector spaces spanned by the old ones, and by taking then the additive envelope of the resulting category. In this way \(\widetilde \Spider^+\) becomes an additive \(\fieldk\)--linear category. For brevity in the following, when we will write \(\bolda \in \widetilde \Spider^+\), we will mean a sequence of non-negative integers (although a generic object of \(\widetilde \Spider^+\) is a formal direct sum of such sequences). We will use a similar convention for the categories \(\Spider^+\), \(\Spider\) and \(\Spider(\beta)\) which will be introduced later on.

\subsection{The category \texorpdfstring{$\Spider^+$}{Sp+}}
\label{sec:category-spider+}

We define a category \(\Spider^+\) by imposing the graphical versions of the relations of the generators of \(U_q(\gl_k)\):

\begin{definition}[\cite{CKM}]
  \label{def:5}
  We define the category \(\Spider^+\) to be the quotient of the category \(\widetilde \Spider^+\) modulo the following relations:
\begin{subequations}
\allowdisplaybreaks
\begin{itemize}
\item Far away elements commute:
\begingroup
\tikzset{every picture/.style={yscale=0.8,xscale=0.8}}
\begin{gather}
  \begin{tikzpicture}[int,smallnodes,anchorbase]
   \draw (0,0) node[below] {$\mathclap{a}\vphantom{b}$} -- ++(0,2) node[above] {$\mathclap{a{-}r}$} \midarrow;
   \draw (1,0) node[below] {$\mathclap{b}\vphantom{b}$} -- ++(0,2) node[above] {$\mathclap{b{+}r}$} \midarrow;
   \node at (1.5,1) {$\cdots$};
   \draw (2,0) node[below] {$\mathclap{c}\vphantom{b}$} -- ++(0,2) node[above] {$\mathclap{c{-}s}$} \midarrow;
   \draw (3,0) node[below] {$\mathclap{d}\vphantom{b}$}-- ++(0,2) node[above] {$\mathclap{d{+}s}$} \midarrow;
   \draw (0,0.4) -- node[above] {$r$} ++(1,0.2) \midarrow;
   \draw (2,1.4) -- node[above] {$s$} ++(1,0.2) \midarrow;
 \end{tikzpicture} \; =\;
  \begin{tikzpicture}[int,smallnodes,anchorbase]
   \draw (0,0) node[below] {$\mathclap{a}\vphantom{b}$} -- ++(0,2) node[above] {$\mathclap{a{-}r}$} \midarrow;
   \draw (1,0) node[below] {$\mathclap{b}\vphantom{b}$} -- ++(0,2) node[above] {$\mathclap{b{+}r}$} \midarrow;
   \node at (1.5,1) {$\cdots$};
   \draw (2,0) node[below] {$\mathclap{c}\vphantom{b}$} -- ++(0,2) node[above] {$\mathclap{c{-}s}$} \midarrow;
   \draw (3,0) node[below] {$\mathclap{d}\vphantom{b}$}-- ++(0,2) node[above] {$\mathclap{d{+}s}$} \midarrow;
   \draw (0,1.4) -- node[above] {$r$} ++(1,0.2) \midarrow;
   \draw (2,0.4) -- node[above] {$s$} ++(1,0.2) \midarrow;
  \end{tikzpicture}\;,\label{eq:72}\\
  \begin{tikzpicture}[int,smallnodes,anchorbase]
   \draw (0,0) node[below] {$\mathclap{a}\vphantom{b}$} -- ++(0,2) node[above] {$\mathclap{a{+}r}$} \midarrow;
   \draw (1,0) node[below] {$\mathclap{b}\vphantom{b}$} -- ++(0,2) node[above] {$\mathclap{b{-}r}$} \midarrow;
   \node at (1.5,1) {$\cdots$};
   \draw (2,0) node[below] {$\mathclap{c}\vphantom{b}$} -- ++(0,2) node[above] {$\mathclap{c{+}s}$} \midarrow;
   \draw (3,0) node[below] {$\mathclap{d}\vphantom{b}$}-- ++(0,2) node[above] {$\mathclap{d{-}s}$} \midarrow;
   \draw (1,0.4) -- node[above] {$r$} ++(-1,0.2) \midarrow;
   \draw (3,1.4) -- node[above] {$s$} ++(-1,0.2) \midarrow;
 \end{tikzpicture} \; =\;
  \begin{tikzpicture}[int,smallnodes,anchorbase]
   \draw (0,0) node[below] {$\mathclap{a}\vphantom{b}$} -- ++(0,2) node[above] {$\mathclap{a{+}r}$} \midarrow;
   \draw (1,0) node[below] {$\mathclap{b}\vphantom{b}$} -- ++(0,2) node[above] {$\mathclap{b{-}r}$} \midarrow;
   \node at (1.5,1) {$\cdots$};
   \draw (2,0) node[below] {$\mathclap{c}\vphantom{b}$} -- ++(0,2) node[above] {$\mathclap{c{+}s}$} \midarrow;
   \draw (3,0) node[below] {$\mathclap{d}\vphantom{b}$}-- ++(0,2) node[above] {$\mathclap{d{-}s}$} \midarrow;
   \draw (1,1.4) -- node[above] {$r$} ++(-1,0.2) \midarrow;
   \draw (3,0.4) -- node[above] {$s$} ++(-1,0.2) \midarrow;
 \end{tikzpicture}\;.
\end{gather}
\endgroup
\item The morphisms \(\sfE^{(r)}\) and \(\sfF^{(r)}\) commute also when they are near to each other:
\begingroup
\tikzset{every picture/.style={yscale=0.8,xscale=0.9}}
\begin{equation}
\begin{tikzpicture}[int,smallnodes,anchorbase]
   \draw (0,0) node[below] {$\mathclap{a}\vphantom{b}$} -- ++(0,2) node[above] {$\mathclap{a{-}r}$} \midarrow;
   \draw (1,0) node[below] {$\mathclap{b}\vphantom{b}$} -- ++(0,2) node[above] {$\mathclap{b{+}r{+}s}$} \midarrow;
   \draw (2,0) node[below] {$\mathclap{c}\vphantom{b}$}-- ++(0,2) node[above] {$\mathclap{c{-}s}$} \midarrow;
   \draw (0,0.4) -- node[above] {$r$} (1,0.6) \midarrow;
   \draw (2,1.4) -- node[above] {$s$} (1,1.6) \midarrow;
 \end{tikzpicture} \; =\;
  \begin{tikzpicture}[int,smallnodes,anchorbase]
   \draw (0,0) node[below] {$\mathclap{a}\vphantom{b}$} -- ++(0,2) node[above] {$\mathclap{a{-}r}$} \midarrow;
   \draw (1,0) node[below] {$\mathclap{b}\vphantom{b}$} -- ++(0,2) node[above] {$\mathclap{b{+}r{+}s}$} \midarrow;
   \draw (2,0) node[below] {$\mathclap{c}\vphantom{b}$}-- ++(0,2) node[above] {$\mathclap{c{-}s}$} \midarrow;
   \draw (0,1.4) -- node[above] {$r$} (1,1.6) \midarrow;
   \draw (2,0.4) -- node[above] {$s$} (1,0.6) \midarrow;
 \end{tikzpicture}
\quad\text{and}\quad
  \begin{tikzpicture}[int,smallnodes,anchorbase]
   \draw (0,0) node[below] {$\mathclap{a}\vphantom{b}$} -- ++(0,2) node[above] {$\mathclap{a{+}r}$} \midarrow;
   \draw (1,0) node[below] {$\mathclap{b}\vphantom{b}$} -- ++(0,2) node[above] {$\mathclap{b{-}r{-}s}$} \midarrow;
   \draw (2,0) node[below] {$\mathclap{c}\vphantom{b}$}-- ++(0,2) node[above] {$\mathclap{c{+}s}$} \midarrow;
   \draw (1,0.4) -- node[above] {$r$} (0,0.6) \midarrow;
   \draw (1,1.4) -- node[above] {$s$} (2,1.6) \midarrow;
 \end{tikzpicture}  = 
  \begin{tikzpicture}[int,smallnodes,anchorbase]
   \draw (0,0) node[below] {$\mathclap{a}\vphantom{b}$} -- ++(0,2) node[above] {$\mathclap{a{+}r}$} \midarrow;
   \draw (1,0) node[below] {$\mathclap{b}\vphantom{b}$} -- ++(0,2) node[above] {$\mathclap{b{-}r{-}s}$} \midarrow;
   \draw (2,0) node[below] {$\mathclap{c}\vphantom{b}$}-- ++(0,2) node[above] {$\mathclap{c{+}s}$} \midarrow;
   \draw (1,1.4) -- node[above] {$r$} (0,1.6) \midarrow;
   \draw (1,0.4) -- node[above] {$s$} (2,0.6) \midarrow;
 \end{tikzpicture}\;.\label{eq:74}
\end{equation}
\item Divided powers multiply as usual:
\begin{equation}
   \begin{tikzpicture}[int,smallnodes,anchorbase,xscale=1.3]
   \draw (0,0) node[below] {$\mathclap{a}\vphantom{b}$} -- ++(0,2) node[above] {$\mathclap{a{-}r{-}s}$} \midarrow;
   \draw (1,0) node[below] {$\mathclap{b}\vphantom{b}$} -- ++(0,2) node[above] {$\mathclap{b{+}r{+}s}$} \midarrow;
   \draw (0,0.4) -- node[above] {$r$} (1,0.6) \midarrow;
   \draw (0,1.4) -- node[above] {$s$} (1,1.6) \midarrow;
 \end{tikzpicture}  \;= \qbin{r+s}{r}\;
  \begin{tikzpicture}[int,smallnodes,anchorbase,xscale=1.3]
   \draw (0,0) node[below] {$\mathclap{a}\vphantom{b}$} -- ++(0,2) node[above] {$\mathclap{a{-}r{-}s}$} \nstartarrow \nendarrow;
   \draw (1,0) node[below] {$\mathclap{b}\vphantom{b}$} -- ++(0,2) node[above] {$\mathclap{b{+}r{+}s}$} \nstartarrow \nendarrow;
   \draw (0,0.9) -- node[above] {$r+s$} (1,1.1) \midarrow;
 \end{tikzpicture}
\quad\text{and}\quad
   \begin{tikzpicture}[int,smallnodes,anchorbase,xscale=1.3]
   \draw (0,0) node[below] {$\mathclap{a}\vphantom{b}$} -- ++(0,2) node[above] {$\mathclap{a{+}r{+}s}$} \midarrow;
   \draw (1,0) node[below] {$\mathclap{b}\vphantom{b}$} -- ++(0,2) node[above] {$\mathclap{b{-}r{-}s}$} \midarrow;
   \draw (1,0.4) -- node[above] {$r$} (0,0.6) \midarrow;
   \draw (1,1.4) -- node[above] {$s$} (0,1.6) \midarrow;
 \end{tikzpicture}  \;= \qbin{r+s}{r}\;
  \begin{tikzpicture}[int,smallnodes,anchorbase,xscale=1.3]
   \draw (0,0) node[below] {$\mathclap{a}\vphantom{b}$} -- ++(0,2) node[above] {$\mathclap{a{+}r{+}s}$} \nstartarrow \nendarrow;
   \draw (1,0) node[below] {$\mathclap{b}\vphantom{b}$} -- ++(0,2) node[above] {$\mathclap{b{-}r{-}s}$} \nstartarrow \nendarrow;
   \draw (1,0.9) -- node[above] {$r+s$} (0,1.1) \midarrow;
 \end{tikzpicture}\;.\label{eq:76}
\end{equation}
\endgroup
\item The commuting relation between \(\sfE\) and \(\sfF\) is
\begingroup
\tikzset{every picture/.style={yscale=0.8,xscale=0.8}}
\begin{gather}
\label{eq:77}
   \begin{tikzpicture}[smallnodes,anchorbase]
   \draw[int] (0,0) node[below] {$\mathclap{a}\vphantom{b}$} -- ++(0,2) node[above] {$\mathclap{a}$} \midarrow;
   \draw[int] (1,0) node[below] {$\mathclap{b}\vphantom{b}$}-- ++(0,2) node[above] {$\mathclap{b}$} \midarrow;
   \draw[uno] (0,0.4) -- node[above] {$1$} (1,0.6) \midarrow;
   \draw[uno] (1,1.4) -- node[above] {$1$} (0,1.6) \midarrow;
 \end{tikzpicture} \;-\;
   \begin{tikzpicture}[smallnodes,anchorbase]
   \draw[int] (0,0) node[below] {$\mathclap{a}\vphantom{b}$} -- ++(0,2) node[above] {$\mathclap{a}$} \midarrow;
   \draw[int] (1,0) node[below] {$\mathclap{b}\vphantom{b}$}-- ++(0,2) node[above] {$\mathclap{b}$} \midarrow;
   \draw[uno] (1,0.4) -- node[above] {$1$} (0,0.6) \midarrow;
   \draw[uno] (0,1.4) -- node[above] {$1$} (1,1.6) \midarrow;
 \end{tikzpicture} \;  = [a-b]\;
   \begin{tikzpicture}[smallnodes,anchorbase,xscale=0.7]
   \draw[int] (0,0) node[below] {$\mathclap{a}\vphantom{b}$} -- ++(0,2) node[above] {$\mathclap{a}$} \midarrow;
   \draw[int] (1,0) node[below] {$\mathclap{b}\vphantom{b}$}-- ++(0,2) node[above] {$\mathclap{b}$} \midarrow;
 \end{tikzpicture}\;.
\end{gather}
\item And finally, we have the Serre relations:
\begin{gather}
  \begin{tikzpicture}[smallnodes,anchorbase]
   \draw[int] (0,0) node[below] {$\mathclap{a}\vphantom{b}$} -- ++(0,2) node[above] {$\mathclap{a{+}2}$} \midarrow;
   \draw[int] (1,0) node[below] {$\mathclap{b}\vphantom{b}$} -- ++(0,2) node[above] {$\mathclap{b{-}1}$} \midarrow;
   \draw[int] (2,0) node[below] {$\mathclap{c}\vphantom{b}$}-- ++(0,2) node[above] {$\mathclap{c{-}1}$} \midarrow;
   \draw[uno] (1,0.6) -- node[above] {$1$} ++(1,-0.2) \midarrowrev;
   \draw[int] (0,1.6) -- node[above] {$2$} ++(1,-0.2) \midarrowrev;
 \end{tikzpicture} \;-\;
  \begin{tikzpicture}[smallnodes,anchorbase]
   \draw[int] (0,0) node[below] {$\mathclap{a}\vphantom{b}$} -- ++(0,2) node[above] {$\mathclap{a{+}2}$} \midarrow;
   \draw[int] (1,0) node[below] {$\mathclap{b}\vphantom{b}$} -- ++(0,1) \nstartarrow \nendarrow -- ++(0,1)  node[above] {$\mathclap{b{-}1}$}  \nendarrow;
   \draw[int] (2,0) node[below] {$\mathclap{c}\vphantom{b}$}-- ++(0,2) node[above] {$\mathclap{c{-}1}$} \nstartarrow \nendarrow;
   \draw[uno] (0,0.6) -- node[above] {$1$} ++(1,-0.2) \midarrowrev;
   \draw[uno] (1,1.1) -- node[above] {$1$} ++(1,-0.2) \midarrowrev;
   \draw[uno] (0,1.6) -- node[above] {$1$} ++(1,-0.2) \midarrowrev;
 \end{tikzpicture}\; +\;
  \begin{tikzpicture}[smallnodes,anchorbase]
   \draw[int] (0,0) node[below] {$\mathclap{a}\vphantom{b}$} -- ++(0,2) node[above] {$\mathclap{a{+}2}$} \midarrow;
   \draw[int] (1,0) node[below] {$\mathclap{b}\vphantom{b}$} -- ++(0,2) node[above] {$\mathclap{b{-}1}$} \midarrow;
   \draw[int] (2,0) node[below] {$\mathclap{c}\vphantom{b}$}-- ++(0,2) node[above] {$\mathclap{c{-}1}$} \midarrow;
   \draw[int] (0,0.6) -- node[above] {$2$} ++(1,-0.2) \midarrowrev;
   \draw[uno] (1,1.6) -- node[above] {$1$} ++(1,-0.2) \midarrowrev;
 \end{tikzpicture} \; = 0,\\
  \begin{tikzpicture}[smallnodes,anchorbase]
   \draw[int] (0,0) node[below] {$\mathclap{a}\vphantom{b}$} -- ++(0,2) node[above] {$\mathclap{a{+}1}$} \midarrow;
   \draw[int] (1,0) node[below] {$\mathclap{b}\vphantom{b}$} -- ++(0,2) node[above] {$\mathclap{b{+}1}$} \midarrow;
   \draw[int] (2,0) node[below] {$\mathclap{c}\vphantom{b}$}-- ++(0,2) node[above] {$\mathclap{c{-}2}$} \midarrow;
   \draw[uno] (0,0.6) -- node[above] {$1$} ++(1,-0.2) \midarrowrev;
   \draw[int] (1,1.6) -- node[above] {$2$} ++(1,-0.2) \midarrowrev;
 \end{tikzpicture} \;-\;
  \begin{tikzpicture}[smallnodes,anchorbase]
   \draw[int] (0,0) node[below] {$\mathclap{a}\vphantom{b}$} -- ++(0,2) node[above] {$\mathclap{a{+}1}$} \nstartarrow \nendarrow;
   \draw[int] (1,0) node[below] {$\mathclap{b}\vphantom{b}$} -- ++(0,1) \nstartarrow  -- ++(0,1)  node[above] {$\mathclap{b{+}1}$}  \nstartarrow \nendarrow;
   \draw[int] (2,0) node[below] {$\mathclap{c}\vphantom{b}$}-- ++(0,2) node[above] {$\mathclap{c{-}2}$} \midarrow;
   \draw[uno] (1,0.6) -- node[above] {$1$} ++(1,-0.2) \midarrowrev;
   \draw[uno] (0,1.1) -- node[above] {$1$} ++(1,-0.2) \midarrowrev;
   \draw[uno] (1,1.6) -- node[above] {$1$} ++(1,-0.2) \midarrowrev;
 \end{tikzpicture}\; +\;
  \begin{tikzpicture}[smallnodes,anchorbase]
   \draw[int] (0,0) node[below] {$\mathclap{a}\vphantom{b}$} -- ++(0,2) node[above] {$\mathclap{a{+}1}$} \midarrow;
   \draw[int] (1,0) node[below] {$\mathclap{b}\vphantom{b}$} -- ++(0,2) node[above] {$\mathclap{b{+}1}$} \midarrow;
   \draw[int] (2,0) node[below] {$\mathclap{c}\vphantom{b}$}-- ++(0,2) node[above] {$\mathclap{c{-}2}$} \midarrow;
   \draw[int] (1,0.6) -- node[above] {$2$} ++(1,-0.2) \midarrowrev;
   \draw[uno] (0,1.6) -- node[above] {$1$} ++(1,-0.2) \midarrowrev;
 \end{tikzpicture} \; = 0, \label{eq:190}\\
  \begin{tikzpicture}[smallnodes,anchorbase]
   \draw[int] (0,0) node[below] {$\mathclap{a}\vphantom{b}$} -- ++(0,2) node[above] {$\mathclap{a{-}2}$} \midarrow;
   \draw[int] (1,0) node[below] {$\mathclap{b}\vphantom{b}$} -- ++(0,2) node[above] {$\mathclap{b{+}1}$} \midarrow;
   \draw[int] (2,0) node[below] {$\mathclap{c}\vphantom{b}$}-- ++(0,2) node[above] {$\mathclap{c{+}1}$} \midarrow;
   \draw[uno] (1,0.4) -- node[above] {$1$} ++(1,0.2) \midarrow;
   \draw[int] (0,1.4) -- node[above] {$2$} ++(1,0.2) \midarrow;
 \end{tikzpicture} \;-\;
  \begin{tikzpicture}[smallnodes,anchorbase]
   \draw[int] (0,0) node[below] {$\mathclap{a}\vphantom{b}$} -- ++(0,2) node[above] {$\mathclap{a{-}2}$} \midarrow;
   \draw[int] (1,0) node[below] {$\mathclap{b}\vphantom{b}$} -- ++(0,1) \nstartarrow -- ++(0,1)  node[above] {$\mathclap{b{+}1}$} \nstartarrow \nendarrow;
   \draw[int] (2,0) node[below] {$\mathclap{c}\vphantom{b}$}-- ++(0,2) node[above] {$\mathclap{c{+}1}$} \nstartarrow \nendarrow;
   \draw[uno] (0,0.4) -- node[above] {$1$} ++(1,0.2) \midarrow;
   \draw[uno] (1,0.9) -- node[above] {$1$} ++(1,0.2) \midarrow;
   \draw[uno] (0,1.4) -- node[above] {$1$} ++(1,0.2) \midarrow;
 \end{tikzpicture}\; +\;
  \begin{tikzpicture}[smallnodes,anchorbase]
   \draw[int] (0,0) node[below] {$\mathclap{a}\vphantom{b}$} -- ++(0,2) node[above] {$\mathclap{a{-}2}$} \midarrow;
   \draw[int] (1,0) node[below] {$\mathclap{b}\vphantom{b}$} -- ++(0,2) node[above] {$\mathclap{b{+}1}$} \midarrow;
   \draw[int] (2,0) node[below] {$\mathclap{c}\vphantom{b}$}-- ++(0,2) node[above] {$\mathclap{c{+}1}$} \midarrow;
   \draw[int] (0,0.4) -- node[above] {$2$} (1,0.6) \midarrow;
   \draw[uno] (1,1.4) -- node[above] {$1$} (2,1.6) \midarrow;
 \end{tikzpicture} \; = 0, \label{eq:189}\\
  \begin{tikzpicture}[smallnodes,anchorbase]
   \draw[int] (0,0) node[below] {$\mathclap{a}\vphantom{b}$} -- ++(0,2) node[above] {$\mathclap{a{-}1}$} \midarrow;
   \draw[int] (1,0) node[below] {$\mathclap{b}\vphantom{b}$} -- ++(0,2) node[above] {$\mathclap{b{-}1}$} \midarrow;
   \draw[int] (2,0) node[below] {$\mathclap{c}\vphantom{b}$}-- ++(0,2) node[above] {$\mathclap{c{+}2}$} \midarrow;
   \draw[uno] (0,0.4) -- node[above] {$1$} ++(1,0.2) \midarrow;
   \draw[int] (1,1.4) -- node[above] {$2$} ++(1,0.2) \midarrow;
 \end{tikzpicture} \;-\;
  \begin{tikzpicture}[smallnodes,anchorbase]
   \draw[int] (0,0) node[below] {$\mathclap{a}\vphantom{b}$} -- ++(0,2) node[above] {$\mathclap{a{-}1}$} \nstartarrow \nendarrow;
   \draw[int] (1,0) node[below] {$\mathclap{b}\vphantom{b}$} -- ++(0,1) \nstartarrow \nendarrow -- ++(0,1)  node[above] {$\mathclap{b{-}1}$} \nendarrow;
   \draw[int] (2,0) node[below] {$\mathclap{c}\vphantom{b}$}-- ++(0,2) node[above] {$\mathclap{c{+}2}$} \midarrow;
   \draw[uno] (1,0.4) -- node[above] {$1$} ++(1,0.2) \midarrow;
   \draw[uno] (0,0.9) -- node[above] {$1$} ++(1,0.2) \midarrow;
   \draw[uno] (1,1.4) -- node[above] {$1$} ++(1,0.2) \midarrow;
 \end{tikzpicture}\; +\;
  \begin{tikzpicture}[smallnodes,anchorbase]
   \draw[int] (0,0) node[below] {$\mathclap{a}\vphantom{b}$} -- ++(0,2) node[above] {$\mathclap{a{-}1}$} \midarrow;
   \draw[int] (1,0) node[below] {$\mathclap{b}\vphantom{b}$} -- ++(0,2) node[above] {$\mathclap{b{-}1}$} \midarrow;
   \draw[int] (2,0) node[below] {$\mathclap{c}\vphantom{b}$}-- ++(0,2) node[above] {$\mathclap{c{+}2}$} \midarrow;
   \draw[int] (1,0.4) -- node[above] {$2$} (2,0.6) \midarrow;
   \draw[uno] (0,1.4) -- node[above] {$1$} (1,1.6) \midarrow;
 \end{tikzpicture} \; = 0.\label{eq:78}
\end{gather}
\endgroup
\end{itemize}
\end{subequations}
\end{definition}

We define now a monoidal functor \(\smash{\calG^+_{m|n}} \colon \widetilde \Spider^+ \mapto \smash{\catRep^+_{m|n}}\). On objects, \(\smash{\calG^+_{m|n}}\) is given by \(\calG^+_{m|n}(a) = \bigwedgeq^a \C_q^{m|n}\), while on the morphisms \eqref{eq:68} and \eqref{eq:69} we set
\begin{align}
\calG^+_{m|n}\big(\sfE^{(r)}\big) &= E^{(r)} \colon \bigwedgeq^a \C_q^{m|n}\otimes \bigwedgeq^b\C_q^{m|n} \longrightarrow \bigwedgeq^{a+r}\C_q^{m|n} \otimes \bigwedgeq^{b-r}\C_q^{m|n}, \\
\calG^+_{m|n}\big(\sfF^{(r)}\big) & = F^{(r)} \colon\bigwedgeq^a\C_q^{m|n} \otimes \bigwedgeq^b\C_q^{m|n} \longrightarrow \bigwedgeq^{a-r}\C_q^{m|n} \otimes \bigwedgeq^{b+r}\C_q^{m|n},\label{eq:79}
\end{align}
 where \(E\) and \(F\), the generators of \(U_q(\gl_2)\), act on \(\bigoplus_{a+b=N} \bigwedge^a \C_q^{m|n} \otimes \bigwedge^b \C_q^{m|n}\) by skew Howe duality. It follows from lemma~\ref{lem:2} that the functor is well-defined.

We get then automatically as a consequence of theorem~\ref{thm:1}:

\begin{theorem}[See also \cite{CKM}]
  \label{thm:2}
  The functor \(\calG^+_{m|n}\) descends to a full functor \(\calG^+_{m|n} \colon \Spider^+ \mapto \catRep^+_{m|n}\).
\end{theorem}

In the non-super case there is actually an easy way to turn \(\smash{\calG^+_{m|n}}\) into an equivalence of categories by passing to a quotient of \(\Spider^+\), as proven in \cite{CKM} (and as also follows by the presentation of the Schur algebra given in \cite{MR1920169}):

\begin{definition}
  \label{def:6}
  We define \(\Spider^+_{\leq m}\) to be the quotient of the category \(\Spider^+\) modulo all objects \(a > m\) (formally, we quotient out all morphisms that factor through a tensor product of objects with at least one tensor factor greater than \(m\)).
\end{definition}

Then:

\begin{theorem}[{See \cite[Theorem~4.4.1]{CKM}}]
  \label{thm:3}
  In case \(n=0\), the functor \(\smash{\calG^+_{m|n}}\) descends to an equivalence of categories \(\calG^+_{m|0} \colon \Spider^+_{\leq m} \mapto  \catRep^+_{m}\).
\end{theorem}

For a sequence \(\bolda = (a_1,\dotsc, a_r) \in \Spider^+\) we let  \(\abs{\bolda}=a_1+\dotsb+a_r\).
Notice that, by construction, \(\Hom_{\Spider^+}(\bolda,\boldb) \neq 0\) if and only if \(\abs{\bolda}=\abs{\boldb}\). Hence we have: 

\begin{lemma}
  \label{lem:3}
  Let \(\bolda,\boldb \in \Spider^+\), and let \(m\geq \abs{\bolda}\). Then the quotient functor \(\Spider^+ \mapto \Spider^+_{\leq m}\) induces an isomorphism \(\Hom_{\Spider^+}(\bolda,\boldb) \cong \Hom_{\Spider^+_{\leq m}} (\bolda, \boldb)\).
\end{lemma}

We recall also the following lemma:

\begin{lemma}
  \label{lem:4}
  Let \(\bolda \in \Spider^+\) and let \(N=\abs{\bolda}\). There are a monomorphism \(\iota_\bolda \colon \bolda \into 1^{\otimes N}\) and an epimorphism \(\pi_\bolda \colon 1^{\otimes N} \surto \bolda\) with \(\pi_\bolda \circ \iota_\bolda = \id_{\bolda}\).
\end{lemma}

\begin{proof}
  It is sufficient to prove the claim for an object \(a \in \Z_{\geq 0}\). Let 
  \begin{equation}
\iota_a = \sfF^{(1)}_{a-1} \dotsm \sfF^{(a-2)}_2 \sfF^{(a-1)}_1 \colon a \cong a \otimes 0^{\otimes a-1} \mapto 1^{\otimes a}\label{eq:202}
\end{equation}
 and 
 \begin{equation}
\pi_a = \sfE_1^{(a-1)} \dotsm \sfE^{(r)}_{a-2} \sfE^{(1)}_{a-1} \colon 1^{\otimes a} \mapto a \otimes 0^{\otimes a-1} \cong a.\label{eq:114}
\end{equation}
It follows from the relation~\eqref{eq:77} that \(\pi_a \circ \iota_a \) is a multiple of the identity, whence the claim.
\end{proof}

As a consequence, for all objects \(\bolda,\boldb \in \Spider^+\) with \(\abs{\bolda}=\abs{\boldb}=N\) we can factor the identity of \(\Hom_{\Spider^+}(\bolda,\boldb)\) as
\begin{equation}
  \label{eq:81}
  \Hom_{\Spider^+}(\bolda,\boldb) \into \Hom_{\Spider^+} (1^{\otimes N}, 1^{\otimes N}) \surto \Hom_{\Spider^+} (\bolda,\boldb).
\end{equation}

\subsection{The functor \texorpdfstring{$\calI^+$}{I+}}
\label{sec:functor-cali+}

Set in \(\Spider^+\)
\begingroup
\tikzset{every picture/.style={scale=0.5}}
  \begin{align}
\label{eq:82}
   c_{1,1} &= (q^{-1}-\sfF \sfE)\idem_{1 \otimes 1}
    =  q^{-1} \;
    \begin{tikzpicture}[smallnodes,anchorbase]
   \draw[uno] (1,0) node[below] {$\mathclap{1}$} -- ++(0,2) \midarrow  node[above] {$\mathclap{1}$} ;
   \draw[uno] (2,0) node[below] {$\mathclap{1}$} -- ++(0,2) \midarrow  node[above] {$\mathclap{1}$} ;
    \end{tikzpicture} \;  -\;
    \begin{tikzpicture}[smallnodes,anchorbase]
   \draw[uno] (0,0) node[below]  {$\mathclap{1}$} -- ++(0,0.7) \midarrow ++(0,0.6) -- ++(0,0.7) \midarrow node[above] {$\mathclap{1}$} ;
   \draw[uno] (1,0) node[below]  {$\mathclap{1}$} -- ++(0,0.5) \midarrow -- ++(-1,0.2) \midarrow;
   \draw[uno] (0,1.3) -- ++(1,0.2) \midarrow -- ++(0,0.5) \midarrow node[above] {$\mathclap{1}$};
   \draw[zero] (1,0.5) -- ++(0,1);
   \draw[int] (0,0.7) -- ++(0,0.6) \midarrow;
    \end{tikzpicture}\;,\\
   c_{1,1}^{-1} &= (q-\sfF\sfE)\idem_{1 \otimes 1}
    = q \;
    \begin{tikzpicture}[smallnodes,anchorbase]
   \draw[uno] (1,0) node[below] {$\mathclap{1}$} -- ++(0,2) \midarrow  node[above] {$\mathclap{1}$} ;
   \draw[uno] (2,0) node[below] {$\mathclap{1}$} -- ++(0,2) \midarrow  node[above] {$\mathclap{1}$} ;
    \end{tikzpicture} \;  -\;
    \begin{tikzpicture}[smallnodes,anchorbase]
   \draw[uno] (0,0) node[below]  {$\mathclap{1}$} -- ++(0,0.7) \midarrow ++(0,0.6) -- ++(0,0.7) \midarrow node[above] {$\mathclap{1}$} ;
   \draw[uno] (1,0) node[below]  {$\mathclap{1}$} -- ++(0,0.5) \midarrow -- ++(-1,0.2) \midarrow;
   \draw[uno] (0,1.3) -- ++(1,0.2) \midarrow -- ++(0,0.5) \midarrow node[above] {$\mathclap{1}$};
   \draw[zero] (1,0.5) -- ++(0,1);
   \draw[int] (0,0.7) -- ++(0,0.6) \midarrow;
    \end{tikzpicture}\;.\label{eq:193}
  \end{align}
\endgroup
The notation is motivated by the fact (easy to check) that \(c_{1,1}^{-1}\)  is the inverse of \(c_{1,1}\).

We define a monoidal functor \(\calI^+ \colon \catH \mapto \Spider^+\). On objects, we set \(\calI^+(1) = 1\). On morphisms, we set
\begingroup
\tikzset{every picture/.style={scale=0.5}}
  \begin{equation}
    \label{eq:203}
    \calI^+ \left(\;\begin{tikzpicture}[smallnodes,anchorbase]
   \draw[uno] (1,0) node[below] {$\mathclap{1}$} -- ++(0,0.3) \midarrow .. controls ++(0,0.7) and ++(0,-0.7) .. ++(-1,1.4) -- ++(0,0.3)  \midarrow node[above] {$\mathclap{1}$} ;
   \draw[uno,cross line] (0,0) node[below] {$\mathclap{1}$} -- ++(0,0.3) \midarrow .. controls ++(0,0.7) and ++(0,-0.7) .. ++(1,1.4) -- ++(0,0.3)  \midarrow node[above] {$\mathclap{1}$} ;
    \end{tikzpicture}\;\right) =  c_{1,1}
    \quad \text{and} \quad
    \calI^+\left(\;    \begin{tikzpicture}[smallnodes,anchorbase]
   \draw[uno] (0,0) node[below] {$\mathclap{1}$} -- ++(0,0.3) \midarrow .. controls ++(0,0.7) and ++(0,-0.7) .. ++(1,1.4) -- ++(0,0.3)  \midarrow node[above] {$\mathclap{1}$} ;
   \draw[uno,cross line] (1,0) node[below] {$\mathclap{1}$} -- ++(0,0.3) \midarrow .. controls ++(0,0.7) and ++(0,-0.7) .. ++(-1,1.4) -- ++(0,0.3)  \midarrow node[above] {$\mathclap{1}$} ;
    \end{tikzpicture} \;\right) = c_{1,1}^{-1}.
  \end{equation}
\endgroup

It is easy to check that \(\calI^+\) is well-defined, i.e.\ that \(c_{1,1}\) satisfies the defining relations \eqref{eq:53} of the Hecke algebra. We will use the pictures in \eqref{eq:203} to denote the morphisms \(c_{1,1}\) and \(c_{1,1}^{-1}\) of \(\Spider^+\).

\begin{remark}
  Alternatively, one can deduce from
  lemma~\ref{lem:6} below that for each fixed \(m \geq 0\)
  the functor \(\calI^+\),
  restricted to the full subcategory \(\calH_m\)
  and composed with the quotient functor
  \(\Spider^+ \mapto \Spider^+_{\leq m}\),
  is equal to the composition \((\calG^+_{m|0})^{-1} \circ \RT_{m|0}\),
  where \((\calG^+_{m|0})^{-1}\)
  is an essential inverse to the equivalence of categories
  \(\calG^+_{m|0} \colon \Spider^+_{\leq m} \mapto \catRep^+_m\),
  and hence \(\calI^+\) is well-defined.
\end{remark}

\begin{lemma}
  \label{lem:6}
  The composition \(\calG^+_{m|n} \circ \calI^+\) gives the Reshetikhin-Turaev invariant \(\RT_{m|n}\), that is we have
\begingroup
\tikzset{every picture/.style={yscale=0.5,xscale=0.5}}
  \begin{equation}
    \label{eq:86}
\calG^+_{m|n}\left(    \begin{tikzpicture}[smallnodes,anchorbase]
   \draw[uno] (1,0) node[below]  {$\mathclap{1}$} -- ++(0,0.7) \midarrow ++(0,0.6) -- ++(0,0.7) \midarrow node[above] {$\mathclap{1}$} ;
   \draw[uno] (0,0) node[below]  {$\mathclap{1}$} -- ++(0,0.5) \midarrow -- ++(1,0.2) \midarrow;
   \draw[uno] (1,1.3) -- ++(-1,0.2) \midarrow -- ++(0,0.5) \midarrow node[above] {$\mathclap{1}$};
   \draw[zero] (0,0.5) -- ++(0,1);
   \draw[int] (1,0.7) -- ++(0,0.6) \midarrow;
    \end{tikzpicture}\right) \; = - \;
   \RT_{m|n}\left( \begin{tikzpicture}[smallnodes,anchorbase]
   \draw[uno] (1,0) node[below] {$\mathclap{1}$} -- ++(0,0.3) \midarrow .. controls ++(0,0.7) and ++(0,-0.7) .. ++(-1,1.4) -- ++(0,0.3)  \midarrow node[above] {$\mathclap{1}$} ;
   \draw[uno,cross line] (0,0) node[below] {$\mathclap{1}$} -- ++(0,0.3) \midarrow .. controls ++(0,0.7) and ++(0,-0.7) .. ++(1,1.4) -- ++(0,0.3)  \midarrow node[above] {$\mathclap{1}$} ;
    \end{tikzpicture} \right)\; + q^{-1} \;
   \RT_{m|n}\left( \begin{tikzpicture}[smallnodes,anchorbase]
   \draw[uno] (1,0) node[below] {$\mathclap{1}$} -- ++(0,2) \midarrow  node[above] {$\mathclap{1}$} ;
   \draw[uno] (2,0) node[below] {$\mathclap{1}$} -- ++(0,2) \midarrow  node[above] {$\mathclap{1}$} ;
    \end{tikzpicture}\right) 
  \end{equation}
\endgroup
as morphisms in \(\catRep^+_{m|n}\).
\end{lemma}

\begin{proof}
 Let \(x_a \otimes x_b \in \C^{m|n}_q \otimes \C^{m|n}_q\) be a standard basis vector. Under the isomorphism \eqref{eq:62} it corresponds to \(z_{a1} \wedgeq z_{b2} \in \bigwedgeq^2 (\C_q^{m|n} \otimes \C_q^k)\). We can now compute
\begin{equation}
  \label{eq:87}
  EF(z_{a1} \wedgeq z_{b2}) = E (z_{a2} \wedgeq z_{b2}) = z_{a2}\wedgeq z_{b1} + q z_{a1} \wedgeq z_{b2}.
\end{equation}
We have now four cases:
\begin{itemize}
\item if \(a=b\) and \(\abs{a}=0\), then \(z_{a2} \wedgeq z_{a1}=-q z_{a1} \wedgeq z_{a2}\) and hence \eqref{eq:87} gives \(0\);
\item if \(a=b\) and  \(\abs{a}=1\), then \(z_{a2} \wedgeq z_{a1}=q^{-1} z_{a1} \wedgeq z_{a2}\) and hence \eqref{eq:87} gives \([2] z_{a1} \wedgeq z_{b2}\);
\item if \(a<b\), then \(z_{a2} \wedgeq z_{b1} =- (-1)^{\abs{a}\abs{b}} z_{b1} \wedgeq z_{a2} + (q^{-1} -q) z_{a1} \wedgeq z_{b2}\) and hence \eqref{eq:87} is equal to \(-(-1)^{\abs{a}\abs{b}} z_{b1} \wedgeq z_{a2} + q^{-1} z_{a1} \wedgeq z_{b2}\);
\item if \(a>b\), then \(z_{a2} \wedgeq z_{b1} = -(-1)^{\abs{a}\abs{b}} z_{b1} \wedgeq z_{a2}\) and hence \eqref{eq:87} is equal to \(-(-1)^{\abs{a}\abs{b}} z_{b1} \wedgeq z_{a2} + q z_{a1} \wedgeq z_{b2}\).
\end{itemize}
An immediate comparison with the formulas \eqref{eq:30} for the \(R\)--matrix proves the claim.
\end{proof}

\begin{prop}
  \label{prop:4}
  The functor \(\calI^+\) is fully faithful.
\end{prop}

\begin{proof}
  We must show that for all \(r \geq 0\) the functor \(\calI^+\) induces an isomorphism \(\ucalH_r \cong \End_{\Spider^+}(1^{\otimes r})\). Choose \(m\geq r\) and \(n=0\). Then, by Schur-Weyl duality, the composition \(\calG^+_{m|0} \circ \calI^+\) induces an isomorphism \(\ucalH_r \cong \End_{\catRep^+}\big((\C_q^m )^{\otimes r}\big)\). By theorem~\ref{thm:3} together with lemma~\ref{lem:3}, on the other side, \(\calG^+_{m|0}\) induces an isomorphism \(\End_{\Spider^+}(1^{\otimes r}) \cong \End_{\catRep^+}\big((\C_q^m)^{\otimes r}\big)\), and the claim follows.
\end{proof}

We recall the following immediate consequence:

\begin{lemma}\label{lem:7}
  The homomorphism space \(\End_{\Spider^+}(1^{\otimes N})\) is monoidally generated by \(c_{1,1}\).
\end{lemma}

\subsection{Braided structure}
\label{sec:braided-structure}

As a reference for the definition of a braided category, we refer to \cite[\textsection{}1.2]{MR1292673}, from where we also copy the notation.

The category \(\Spider^+\) is a sort of inverse limit of the categories \(\Spider^+_{\leq m}\). Each one of the categories \(\Spider^+_{\leq m}\), being equivalent to \(\catRep^+_m\), is braided. This allows to induce a braided structure on \(\Spider^+\):

\begin{prop}
  \label{prop:5}
  The category \(\Spider^+\) is a braided category, with braiding on \(1 \otimes 1\)  given by \(c_{1,1}\).
\end{prop}

\begin{proof}
  Given objects \(\bolda,\boldb \in \Spider^+\), choose \(m \geq \abs{\bolda}+\abs{\boldb} \). In the category \(\Spider^+_{\leq m}\), which is braided since it is equivalent to \(\catRep^+_{m}\), there is a braiding morphism, which  via the isomorphism given by lemma~\ref{lem:3} gives an element \(c_{\bolda,\boldb} \in \Hom_{\Spider^+}(\bolda,\boldb)\). This element does not depend on the choice of \(m\) (this is a consequence of the naturality of the construction, and can be proved using \eqref{eq:82} and lemma~\ref{lem:8} below). It follows that the elements \(c_{\bolda,\boldb}\) define a braided structure (the axioms hold since they hold in \(\Spider^+_{\leq m}\) for all big enough \(m\)).
\end{proof}

Actually, we can also give a direct proof, which does not use the braided structure of \(\catRep^+_{m}\) (and we could even recover the braided structure of \(\catRep^+_m\) from it).

\begin{proof}[Alternative proof]
  We  define the braiding on \(1 \otimes 1\) using \eqref{eq:82}. By lemma~\ref{lem:8} below, there is a unique possible way to extend it to a candidate braiding on the whole category, using \eqref{eq:90}. We need then to check that such extension is natural, i.e.\ that for \(\bolda,\bolda',\boldb,\boldb' \in \Spider^+\) and \(f\colon \bolda \mapto\bolda'\), \(g \colon \boldb \mapto \boldb'\) we have
  \begin{equation}
(f \otimes g)c_{\bolda,\boldb}=c_{\bolda',\boldb'}(f \otimes g).\label{eq:88}
\end{equation}
 We can suppose \(\abs{\bolda}=\abs{\bolda'}=M\) and \(\abs{\boldb}=\abs{\boldb'}=N\) (otherwise the claim is trivial). 
We have then
\begin{equation}
  \label{eq:89}
  c_{\bolda',\boldb'}(f \otimes g) = (\pi_{\bolda'} \otimes \pi_{\boldb'}) c_{1^{\otimes M},1^{\otimes N}} (\iota_{\bolda'} \otimes \iota_{\boldb'}) (f \otimes g) (\pi_\bolda \otimes \pi_\boldb)(\iota_\bolda \otimes \iota_\boldb).
\end{equation}
If we set \(\tilde f = \iota_{\bolda'} f \pi_\bolda\) and \(\tilde g = \iota_{\boldb'} g \pi_\boldb\) and we suppose that we can prove  naturality for \(c_{1^{\otimes M},1^{\otimes N}}\), then we get \(c_{1^{\otimes M},1^{\otimes N}} (\tilde f \otimes \tilde g)=(\tilde f \otimes \tilde g) c_{1^{\otimes M},1^{\otimes N}}\) and \eqref{eq:88} follows.

So we reduced the problem to the case \(\bolda=\bolda'=1^{\otimes M}\) and \(\boldb=\boldb'=1^{\otimes N}\). By lemma~\ref{lem:7}, we can then assume that \(f\) and \(g\) are just \(c_{1,1}\) tensored with identities, and the claim follows then immediately from the braid relation for \(c_{1,1}\).
\end{proof}

We recall the following observation, which we used in the proof of the previous proposition.

\begin{lemma}[{See also \cite[\textsection{}6.2]{CKM}}]\label{lem:8}
  A braided structure on \(\Spider^+\) is uniquely determined by the braiding of \(1 \otimes 1\). In other words, suppose we have two braided structures on \(\Spider^+\), with braidings \(c_{\bolda,\boldb}\) and \(c'_{\bolda,\boldb}\) respectively. If \(c_{1,1}=c_{1,1}'\) then the two braided structures coincide.
\end{lemma}

\begin{proof}
  We need to show that \(c_{\bolda,\boldb} = c'_{\bolda,\boldb}\) for all objects \(\bolda,\boldb \in \Spider^+\). Since the braiding is compatible with the monoidal structure (see \cite[(1.2.b) and (1.2.c)]{MR1292673}), it is sufficient to prove that \(c_{a,b}=c'_{a,b}\) for all \(a,b \in \Z_{\geq 0}\). Using naturality \cite[(1.2.d)]{MR1292673} we get
  \begin{multline}
    \label{eq:90}
    c_{a,b} = c_{a,b} (\pi_a \otimes \pi_b)(\iota_a \otimes \iota_b) = (\pi_a \otimes \pi_b) c_{1^{\otimes N}, 1^{\otimes N'}} (\iota_a \otimes \iota_b) \\= (\pi_a \otimes \pi_b) c'_{1^{\otimes N}, 1^{\otimes N'}} (\iota_a \otimes \iota_b) = c'_{a,b} (\pi_a \otimes \pi_b)(\iota_a \otimes \iota_b) = c'_{a,b}
  \end{multline}
  and we are done.
\end{proof}

Up to rescaling, our formula for \(c_{1,1}\) coincides with Lusztig's symmetry \cite[5.2.1]{Lus4} (cf.\ also \cite{Queff_aff}). In \cite[\textsection{}6.1]{CKM} a braided structure on \(\Spider^+\) is constructed using these Lusztig's symmetries. Then by lemma~\ref{lem:8} we can deduce that the braiding defined in proposition~\ref{prop:5} is the same as the one in \cite[\textsection{}6.1]{CKM}, and we have the following explicit formula:
\begin{equation}
c_{a,b}=(-1)^{\delta_{a\geq b}\cdot (a-b)}q^{-a}\sum_{-k+l-m=a-b}(-q)^{-km+l}E^{(k)}F^{(l)}E^{(m)}\idem_{a \otimes b}.
\end{equation}
It can be shown \cite[Lemma 6.1.1]{CKM} that this formula is equivalent to the following ones:
\begin{align}
c_{a,b} =
  \begin{cases}
    q^{-a}\sum_{s\geq 0}(-q)^s E^{(b-a+s)}F^{(s)}\idem_{a \otimes b}  & \text{if } a-b\leq 0, \\
    q^{-b}\sum_{s\geq 0}
    (-q)^sF^{(a-b+s)}E^{(s)}\idem_{a \otimes b} & \text{if }    a-b\geq 0.
  \end{cases}
\end{align}

Lusztig's symmetries  have been largely studied, in particular in \cite{Lus4}. We recall  some formulas from \cite[Section~37.1.3]{Lus4} in  our conventions:

\begin{align} \label{eq:EqW1}
 \begin{tikzpicture} [scale=.5,smallnodes,anchorbase]
\draw [int] (1,0) node[below] {$\vphantom{b}\mathclap{b}$}-- ++(0,1) .. controls ++(0,0.5) and ++(0,-0.5) .. ++(-1,1) -- ++(0,0.1)  \nendarrow;
\draw [int, cross line] (0,0) node[below] {$\vphantom{b}\mathclap{a}$}  -- ++(0,1) .. controls ++(0,0.5) and ++(0,-0.5) .. ++(1,1) -- ++(0,0.1)  \nendarrow;
\draw [uno] (0,0.5) -- ++ (1,0.2) \midarrow;
 \end{tikzpicture}
\;& =  (-1)^{1+\delta_{a=b+1}}q^{a-b-1} \;
 \begin{tikzpicture} [scale=.5,smallnodes,anchorbase]
\draw [int] (1,0) node[below] {$\vphantom{b}\mathclap{b}$}-- ++(0,0.1) .. controls ++(0,0.5) and ++(0,-0.5) .. ++(-1,1) -- ++(0,0.9) -- ++(0,0.1)  \nendarrow;
\draw [int, cross line] (0,0) node[below] {$\vphantom{b}\mathclap{a}$}  -- ++(0,0.1) .. controls ++(0,0.5) and ++(0,-0.5) .. ++(1,1) -- ++(0,0.9) -- ++(0,0.1)  \nendarrow;
\draw [uno] (1,1.5) -- ++ (-1,0.2) \midarrow;
 \end{tikzpicture}\;,
&\;\;
 \begin{tikzpicture} [scale=.5,smallnodes,anchorbase]
\draw [int] (1,0) node[below] {$\vphantom{b}\mathclap{b}$}-- ++(0,1) .. controls ++(0,0.5) and ++(0,-0.5) .. ++(-1,1) -- ++(0,0.1)  \nendarrow;
\draw [int, cross line] (0,0) node[below] {$\vphantom{b}\mathclap{a}$}  -- ++(0,1) .. controls ++(0,0.5) and ++(0,-0.5) .. ++(1,1) -- ++(0,0.1)  \nendarrow;
\draw [uno] (1,0.5) -- ++ (-1,0.2) \midarrow;
 \end{tikzpicture}
\;& =  (-1)^{1+\delta_{a=b-1}}q^{b-a-1}\;
 \begin{tikzpicture} [scale=.5,smallnodes,anchorbase]
\draw [int] (1,0) node[below] {$\vphantom{b}\mathclap{b}$}-- ++(0,0.1) .. controls ++(0,0.5) and ++(0,-0.5) .. ++(-1,1) -- ++(0,0.9) -- ++(0,0.1)  \nendarrow;
\draw [int, cross line] (0,0) node[below] {$\vphantom{b}\mathclap{a}$}  -- ++(0,0.1) .. controls ++(0,0.5) and ++(0,-0.5) .. ++(1,1) -- ++(0,0.9) -- ++(0,0.1)  \nendarrow;
\draw [uno] (0,1.5) -- ++ (1,0.2) \midarrow;
 \end{tikzpicture}\;,\\
 \label{eq:EqW2}
\begin{tikzpicture} [scale=.5,smallnodes,anchorbase]
\draw [int, cross line] (0,0) node[below] {$\vphantom{b}\mathclap{a}$}  -- ++(0,1) .. controls ++(0,0.5) and ++(0,-0.5) .. ++(1,1) -- ++(0,0.1)  \nendarrow;
\draw [int,cross line] (1,0) node[below] {$\vphantom{b}\mathclap{b}$}-- ++(0,1) .. controls ++(0,0.5) and ++(0,-0.5) .. ++(-1,1) -- ++(0,0.1)  \nendarrow;
\draw [uno] (0,0.5) -- ++ (1,0.2) \midarrow;
 \end{tikzpicture}
 \; &= (-1)^{1+\delta_{a=b+1}}q^{b-a+1}
 \begin{tikzpicture} [scale=.5,smallnodes,anchorbase]
\draw [int] (0,0) node[below] {$\vphantom{b}\mathclap{a}$}  -- ++(0,0.1) .. controls ++(0,0.5) and ++(0,-0.5) .. ++(1,1) -- ++(0,0.9) -- ++(0,0.1)  \nendarrow;
\draw [int,cross line] (1,0) node[below] {$\vphantom{b}\mathclap{b}$}-- ++(0,0.1) .. controls ++(0,0.5) and ++(0,-0.5) .. ++(-1,1) -- ++(0,0.9) -- ++(0,0.1)  \nendarrow;
\draw [uno] (1,1.5) -- ++ (-1,0.2) \midarrow;
 \end{tikzpicture}\;,
&
 \begin{tikzpicture} [smallnodes,anchorbase,scale=.5]
\draw [int] (0,0) node[below] {$\vphantom{b}\mathclap{a}$}  -- ++(0,1) .. controls ++(0,0.5) and ++(0,-0.5) .. ++(1,1) -- ++(0,0.1)  \nendarrow;
\draw [int, cross line] (1,0) node[below] {$\vphantom{b}\mathclap{b}$}-- ++(0,1) .. controls ++(0,0.5) and ++(0,-0.5) .. ++(-1,1) -- ++(0,0.1)  \nendarrow;
\draw [uno] (1,0.5) -- ++ (-1,0.2) \midarrow;
 \end{tikzpicture}
 \;& =  (-1)^{1+\delta_{b=a+1}}q^{a-b+1}
 \begin{tikzpicture} [scale=.5, smallnodes,anchorbase]
\draw [int] (0,0) node[below] {$\vphantom{b}\mathclap{a}$}  -- ++(0,0.1) .. controls ++(0,0.5) and ++(0,-0.5) .. ++(1,1) -- ++(0,0.9) -- ++(0,0.1)  \nendarrow;
\draw [int, cross line] (1,0) node[below] {$\vphantom{b}\mathclap{b}$}-- ++(0,0.1) .. controls ++(0,0.5) and ++(0,-0.5) .. ++(-1,1) -- ++(0,0.9) -- ++(0,0.1)  \nendarrow;
\draw [uno] (0,1.5) -- ++ (1,0.2) \midarrow;
 \end{tikzpicture}\;.
\end{align}

\section{Graphical calculus for duals}
\label{sec:graph-calc-duals}

\subsection{The category \texorpdfstring{$\Spider$}{Sp}}
\label{sec:category-spider}

We define now a new monoidal category \(\Spider\) by adding to \(\Spider^+\) left and right duals. 
In detail, objects of \(\Spider\) are (formal direct sums of) sequences of integers and the tensor product of objects is given by concatenation of sequences. 
We set \(a^* = -a\) for all \(a \in \Z\).  Morphisms are monoidally generated by the identity morphisms \(\id_a\) for \(a \in \Z\), by the morphisms of \(\Spider^+\) (subject to the relations of \(\Spider^+\)) and by adjunction morphisms turning \(a^*\) into a left and right dual of \(a\) for all \(a \in \Z_{>0}\).

We use orientations to distinguish in our pictures between objects \(a >0\) and their duals. Graphically, we represent the identity morphism \(a^* \mapto a^*\) for \(a > 0\) by
\begin{equation}
  \label{eq:100}
\begin{tikzpicture}[int,smallnodes,anchorbase,xscale=1.3,yscale=1]
   \draw (0,0) node[below] {\(a^*\)} -- (0,1) node[above] {\(a^*\)} \midarrowrev;
 \end{tikzpicture}
\end{equation}
As usual, we represent the adjunction morphisms for \(a > 0\) by the pictures
\begin{equation}
  \label{eq:101}
  \begin{tikzpicture}[smallnodes,anchorbase]
    \draw[int] (0,0) node[below] {$\vphantom{a^*}a$} .. controls ++(0,+0.5) and ++(0,+0.5) .. ++(1,0)  \nendarrow node[below] {$a^*$} ;
  \end{tikzpicture}\qquad
  \begin{tikzpicture}[smallnodes,anchorbase]
    \draw[int] (0,0) node [below] {$a^*$}.. controls ++(0,+0.5) and ++(0,+0.5) .. ++(1,0)  \nstartarrowrev node[below] {$\vphantom{a^*}a$};
  \end{tikzpicture}\qquad
  \begin{tikzpicture}[smallnodes,anchorbase]
    \draw[int] (0,1) node[above] {$a^*$} .. controls ++(0,-0.5) and ++(0,-0.5) .. ++(1,0)  \nendarrow node[above] {$a$} ;
  \end{tikzpicture} \qquad 
  \begin{tikzpicture}[smallnodes,anchorbase]
    \draw[int] (0,1) node[above] {$a$} .. controls ++(0,-0.5) and ++(0,-0.5) .. ++(1,0)  \nstartarrowrev node[above] {$a^*$} ;
  \end{tikzpicture}
\end{equation}
Saying that they turn \(a^*\) into a left and right dual of \(a\) amounts to imposing  the duality relations
\begingroup
\tikzset{every picture/.style={yscale=0.8,xscale=0.5}}
\begin{equation}
  \label{eq:102}
\begin{tikzpicture}[int,smallnodes,anchorbase]
   \draw (-1,1.5) node[above] {\(\mathclap{a}\)}  -- ++(0,-1) \midarrowrev .. controls ++(0,-0.5) and ++(0,-0.5) .. ++(1,0) \midarrowrev -- ++(0,0.5) \midarrowrev  .. controls ++(0,0.5) and ++(0,0.5) .. ++(1,0) \midarrowrev -- ++(0,-1) \midarrowrev node[below] {\(\mathclap{a}\)} ;
 \end{tikzpicture}\; = \;
 \begin{tikzpicture}[int,smallnodes,anchorbase]
   \draw (3,1.5) node[above] {\(\mathclap{a}\)}  -- ++(0,-1) \midarrowrev .. controls ++(0,-0.5) and ++(0,-0.5) .. ++(-1,0) \midarrowrev -- ++(0,0.5) \midarrowrev  .. controls ++(0,0.5) and ++(0,0.5) .. ++(-1,0) \midarrowrev -- ++(0,-1) \midarrowrev node[below] {\(\mathclap{a}\)};
 \end{tikzpicture}  \; = \;
\begin{tikzpicture}[int,smallnodes,anchorbase,xscale=1.3]
   \draw (0,0) node[below] {\(a\)} -- (0,1.5) node[above] {\(a\)} \midarrow;
 \end{tikzpicture}
\quad \text{and} \quad
\begin{tikzpicture}[int,smallnodes,anchorbase]
   \draw (-1,1.5) node[above] {\(\mathclap{a^*}\)}  -- ++(0,-1) \midarrow .. controls ++(0,-0.5) and ++(0,-0.5) .. ++(1,0) \midarrow -- ++(0,0.5) \midarrow  .. controls ++(0,0.5) and ++(0,0.5) .. ++(1,0) \midarrow -- ++(0,-1) \midarrow node[below] {\(\mathclap{a^*}\)} ;
 \end{tikzpicture}\; = \;
 \begin{tikzpicture}[int,smallnodes,anchorbase]
   \draw (3,1.5) node[above] {\(\mathclap{a^*}\)}  -- ++(0,-1) \midarrow .. controls ++(0,-0.5) and ++(0,-0.5) .. ++(-1,0) \midarrow -- ++(0,0.5) \midarrow  .. controls ++(0,0.5) and ++(0,0.5) .. ++(-1,0) \midarrow -- ++(0,-1) \midarrow node[below] {\(\mathclap{a^*}\)};
 \end{tikzpicture}  \; = \;
\begin{tikzpicture}[int,smallnodes,anchorbase,xscale=1.3]
   \draw (0,0) node[below] {\(a^*\)} -- (0,1.5) node[above] {\(a^*\)} \midarrowrev;
 \end{tikzpicture}
\end{equation}\;.
\endgroup
Notice that every object \(\bolda = (a_1,\dotsc,a_\ell) \in \Spider\) has a left and right dual \(\bolda^* = (a_\ell^*,\dotsc,a_1^*)\), with adjunction morphisms obtained by nesting (graphically) the generating adjunction morphisms \eqref{eq:101}. Hence, by adjunction, we have right and left duals of the generating morphisms \(\sfE^{(r)}\) and \(\sfF^{(r)}\). We impose that these left and right duals coincide. In pictures:
\begingroup
\tikzset{every picture/.style={yscale=0.7,xscale=0.6}}
\begin{equation}
\label{eq:104}
\begin{tikzpicture}[int,smallnodes,anchorbase]
   \draw (-1,1.5) node[aboverotated] {\({(b{-}r)^*}\)}  -- ++(0,-1.5) \midarrow .. controls ++(0,-0.5) and ++(0,-0.5) .. ++(1,0) \midarrow -- ++(0,1) \nstartarrow \nendarrow  .. controls ++(0,1.5) and ++(0,1.5) .. ++(3,0) \midarrow -- ++(0,-1.5) \midarrow node[below] {\(\mathclap{b^*}\)} ;
   \draw (-2,1.5) node[aboverotated] {${(a{+}r)^*}$}  -- ++(0,-1.5) \midarrow .. controls ++(0,-1.5) and ++(0,-1.5) .. ++(3,0) \midarrow -- ++ (0,1) \nstartarrow \nendarrow .. controls ++(0,0.5) and ++(0,0.5) .. ++(1,0) \midarrow -- ++(0,-1.5) node[below] {\(\vphantom{b^*}\mathclap{a^*}\)} \midarrow;
   \draw (1,0.4) -- node[above] {$r$} (0,0.6) \midarrow;
 \end{tikzpicture} \; = \;
\begin{tikzpicture}[int,smallnodes,anchorbase]
   \draw (3,1.5) node[aboverotated] {\({(b{-}r)^*}\)}  -- ++(0,-1.5) \midarrow .. controls ++(0,-1.5) and ++(0,-1.5) .. ++(-3,0) \midarrow -- ++(0,1) \nstartarrow \nendarrow  .. controls ++(0,0.5) and ++(0,0.5) .. ++(-1,0) \midarrow -- ++(0,-1.5) \midarrow node[below] {\(\mathclap{b^*}\)};
   \draw (2,1.5) node[aboverotated] {${(a{+}r)^*}$}  -- ++(0,-1.5) \midarrow .. controls ++(0,-0.5) and ++(0,-0.5) .. ++(-1,0) \midarrow -- ++ (0,1) \nstartarrow \nendarrow .. controls ++(0,1.5) and ++(0,1.5) .. ++(-3,0) \midarrow -- ++(0,-1.5) node[below] {\(\mathclap{a^*}\vphantom{b^*}\)} \midarrow;
   \draw (1,0.4) -- node[above] {$r$} (0,0.6) \midarrow;
 \end{tikzpicture} 
\end{equation}
for \(\sfE^{(r)}\idem_{(b-r,a+r)}\) and
\begin{equation}
\label{eq:105}
\begin{tikzpicture}[int,smallnodes,anchorbase]
   \draw (-1,1.5) node[aboverotated] {\({(b{+}r)^*}\)}  -- ++(0,-1.5) \midarrow .. controls ++(0,-0.5) and ++(0,-0.5) .. ++(1,0) \midarrow -- ++(0,1) \nstartarrow \nendarrow  .. controls ++(0,1.5) and ++(0,1.5) .. ++(3,0) \midarrow -- ++(0,-1.5) \midarrow node[below] {\(\mathclap{b^*}\)} ;
   \draw (-2,1.5) node[aboverotated] {${(a{-}r)^*}$}  -- ++(0,-1.5) \midarrow .. controls ++(0,-1.5) and ++(0,-1.5) .. ++(3,0) \midarrow -- ++ (0,1) \nstartarrow \nendarrow .. controls ++(0,0.5) and ++(0,0.5) .. ++(1,0) \midarrow -- ++(0,-1.5) node[below] {\(\vphantom{b^*}\mathclap{a^*}\)} \midarrow;
   \draw (0,0.4) -- node[above] {$r$} (1,0.6) \midarrow;
 \end{tikzpicture} \; = \;
\begin{tikzpicture}[int,smallnodes,anchorbase]
   \draw (3,1.5) node[aboverotated] {\({(b{+}r)^*}\)}  -- ++(0,-1.5) \midarrow .. controls ++(0,-1.5) and ++(0,-1.5) .. ++(-3,0) \midarrow -- ++(0,1) \nstartarrow \nendarrow  .. controls ++(0,0.5) and ++(0,0.5) .. ++(-1,0) \midarrow -- ++(0,-1.5) \midarrow node[below] {\(\mathclap{b^*}\)};
   \draw (2,1.5) node[aboverotated] {${(a{-}r)^*}$}  -- ++(0,-1.5) \midarrow .. controls ++(0,-0.5) and ++(0,-0.5) .. ++(-1,0) \midarrow -- ++ (0,1) \nstartarrow \nendarrow .. controls ++(0,1.5) and ++(0,1.5) .. ++(-3,0) \midarrow -- ++(0,-1.5) node[below] {\(\mathclap{a^*}\vphantom{b^*}\)} \midarrow;
   \draw (0,0.4) -- node[above] {$r$} (1,0.6) \midarrow;
 \end{tikzpicture} 
\end{equation}
\endgroup
for \(\sfF^{(r)}\idem_{(b+r,a-r)}\). Since \(\sfE^{(r)}\) and \(\sfF^{(r)}\), together with identities and the adjunction morphisms, generate the morphisms of \(\Spider\), it follows that the left and right duals of \emph{every} morphism coincide.

We denote
\begin{equation}
  \label{eq:106}
  \sfE^{(r)} \idem_{a^* \otimes b^*} = \big(\sfF^{(r)}\idem_{a+r \otimes b-r}\big)^* \qquad \text{and} \qquad   \sfF^{(r)} \idem_{a^* \otimes b^*} = \big(\sfE^{(r)}\idem_{a-r \otimes b+r}\big)^* 
\end{equation}
and we use the  pictures
\begingroup 
\tikzset{every picture/.style={xscale=0.7,yscale=0.7}}
\begin{equation}\label{eq:107}
\sfF^{(r)} \idem_{a^* \otimes b^*} = \begin{tikzpicture}[int,smallnodes,baseline={([yshift=-0.5ex]0,0.7)},xscale=1.3,yscale=2]
   \draw (0,0) node[below] {\(\vphantom{b^*}\mathclap{a^*}\)} -- (0,1) node[aboverotated] {\({(a{+}r)^*}\)} \nendarrowrev \nstartarrowrev;
   \draw (1,0) node[below] {\(\mathclap{b^*}\)} -- (1,1) node[aboverotated] {\({(b{-}r)^*}\)} \nendarrowrev \nstartarrowrev;
   \draw (1,0.4) -- node[above] {$r$} (0,0.6) \midarrowrev;
 \end{tikzpicture}
 \qquad \text{and} \qquad
 \sfE^{(r)} \idem_{a^* \otimes b^*} = 
\begin{tikzpicture}[int,smallnodes,baseline={([yshift=-0.5ex]0,0.7)},xscale=1.3,yscale=2]
   \draw (0,0) node[below] {\(\mathclap{a^*}\vphantom{b^*}\)} -- (0,1) node[aboverotated] {\({(a-r)^*}\)} \nendarrowrev \nstartarrowrev;
   \draw (1,0) node[below] {\(\mathclap{b^*}\)} -- (1,1) node[aboverotated] {\({(b{+}r)^*}\)} \nendarrowrev \nstartarrowrev;
   \draw (0,0.4) -- node[above] {$r$} (1,0.6) \midarrowrev;
 \end{tikzpicture}
\end{equation}
\endgroup
for \eqref{eq:104} and \eqref{eq:105}, respectively.

Since the category \(\catRep_{m|n}\) is ribbon, we have

\begin{prop}
  \label{prop:6}
  The functor \(\calG \colon \Spider^+ \mapto \catRep_{m|n}^+\) extends to a functor \(\calG \colon \Spider \mapto \catRep_{m|n}\).
\end{prop}

\begin{proof}
  This is almost straightforward. It is sufficient to map the adjunction morphisms \eqref{eq:101} to the duality morphisms of the ribbon structure of \(\catRep_{m|n}\) under the functor \(\calG\), and to notice that since \(\catRep_{m|n}\) is braided, left and right duals of morphisms agree and hence the relations~\eqref{eq:104} and \eqref{eq:105} are automatically satisfied.
\end{proof}

\begin{prop}
  \label{prop:7}
  The functor \(\calG\colon \Spider \mapto \catRep_{m|n}\) is full.
\end{prop}

\begin{proof}
  Let \(\bolda=(a_1,\dotsc,a_\ell), \boldb=(b_1,\dotsc,b_\kappa)\) be two objects  of \(\Spider\), and suppose without loss of generality that \(a_i,b_j \neq 0\) for all \(i,j\). Suppose first that the following condition is satisfied:
  \begin{itemize}
  \item[(*)] there are indexes \(1 \leq \ell' \leq \ell\) and \(1
    \leq \kappa' \leq \kappa\) such that \(a_i > 0 \) if and only if \(i \leq \ell'\) while
    \(b_j > 0 \) if and only if \(j > \kappa'\).
  \end{itemize}
 Let \(\bolda'=(a_1,\dotsc,a_{\ell'})\) and \(\bolda''=(a_{\ell'+1},\dotsc,a_\ell)\), and \(\boldb'=(b_1,\dotsc,b_{\kappa'})\) and \(\boldb''=(b_{\kappa'+1},\dotsc,b_{\kappa})\). 
Let us define sequences \(\boldc\) and \(\boldd\) of objects of \(\Spider^+\) by 
\(\boldc=(\boldb')^* \otimes \bolda'\) and \(\boldd= \boldb'' \otimes (\bolda'')^*\). 
Then we have an isomorphism
  \begin{equation}\label{eq:111}
    \begin{aligned}
      \Hom_{\catRep_{m|n}}(\bigwedgeq^\bolda , \bigwedgeq^\boldb) & \longmapto \Hom_{\catRep^+_{m|n}}(\bigwedgeq^{\boldc} ,\bigwedgeq^\boldd)\\
      \gamma & \longmapsto \hat \gamma =    (\ev_{\boldb^{\prime *}} \otimes \id_{\boldd}) (\id_{{\boldb^{\prime *}}} \otimes \gamma  \otimes \id_{{{\bolda^{\prime\prime*}}}} ) (\id_{\boldc} \otimes \coev_{\bolda^{\prime \prime}} )
    \end{aligned}
  \end{equation}
whose inverse is given by
  \begin{equation}\label{eq:112}
    \begin{aligned}
      \Hom_{\catRep^+_{m|n}}(\bigwedgeq^\boldc , \bigwedgeq^\boldd) & \longmapto \Hom_{\catRep}(\bigwedgeq^{\bolda} ,\bigwedgeq^\boldb)\\
      \gamma & \longmapsto \tilde \gamma =   (\id_{\boldb} \otimes \ev_{\bolda^{\prime\prime *}} ) (\id_{{\boldb^{\prime}}} \otimes \gamma  \otimes \id_{{{\bolda^{\prime\prime}}}} ) (\coev_{\boldb'} \otimes \id_{\bolda} ),
    \end{aligned}
  \end{equation}
where for simplicity we omitted the symbol \(\bigwedgeq\) in the subscripts.

Let now \(\gamma \colon \bigwedgeq^\bolda \mapto \bigwedgeq^\boldb\) be a morphism in \(\catRep_{m|n}\). Since \(\calG^+_{m|n}\) is full, there is an element \(\phi \in \Hom_{\Spider^+}(\boldc,\boldd)\) such that \(\calG^+_{m|n}(\phi)=\hat \gamma\). Let
  \begin{equation}\label{eq:204}
 \tilde \phi =    (\id_{\boldb} \otimes \ev_{\bolda^{\prime\prime *}} ) (\id_{{\boldb^{\prime}}} \otimes \phi  \otimes \id_{{{\bolda^{\prime\prime}}}} ) (\coev_{\boldb'} \otimes \id_{\bolda} ) \in \Spider,
  \end{equation}
where now \(\ev_{\bolda^{\prime\prime *}}\) and \(\coev_{\boldb'}\) denote the adjunction morphisms in \(\Spider\). Then \(\calG_{m|n} (\tilde \phi) = \tilde{\hat \gamma}=\gamma\), hence \(\gamma\) is in the image of \(\calG_{m|n}\).

Let now \(\bolda\) and \(\boldb\) be general, and let \(\gamma \in \Hom_\catRep(\bigwedgeq^\bolda,\bigwedgeq^\boldb)\). Choose permutations \(w\) and \(z\) such that \(w\bolda\) and \(z \boldb\) satisfy (*). Corresponding to some reduced expressions of \(w\) and \(z\) there are isomorphisms \(\Psi_1 \colon \bigwedgeq^\bolda \mapto \bigwedgeq^{w \bolda}\) and \(\Psi_2 \colon \bigwedgeq^{\boldb} \mapto \bigwedgeq^{z \boldb}\) which are given by the braiding of \(\catRep\). Although the category \(\Spider\) is not braided, it is straightforward to see that these isomorphisms are in the image of \(\calG_{m|n}\), since the braiding morphisms of \(\catRep_{m|n}\) are obtained by taking the duals (graphically ``rotating'') of the braiding morphisms of \(\catRep_{m|n}^+\) (cf.\ also the pictures \eqref{eq:143} and \eqref{eq:145} below). So there exist morphisms \(\widetilde{\Psi_1}\) and \(\widetilde{\Psi_2^{-1}}\) in \(\Spider\) such that \(\calG(\widetilde{\Psi_1})=\Psi_1\) and \(\calG(\widetilde{\Psi_2^{-1}})=\Psi_2^{-1}\). By the first part of the proof, there is a morphism \(\phi\) in \(\Spider\) such that \(\calG(\phi)=\Psi_2 \circ \gamma \circ \Psi_1^{-1}\).  Hence \(\calG(\widetilde{\Psi_2^{-1}}\circ \phi\circ \widetilde{\Psi_1}) = \gamma\), and we are done.
\end{proof}

\begin{remark}
  \label{rem:2}
  Notice that proposition~\ref{prop:7} is a special case of the following more
  general statement. Consider a full strong monoidal functor
  \(\calG\colon \calM^+ \mapto \calB\) from a monoidal
  category \(\calM^+\) to a ribbon category \(\calB\). Suppose that the images of objects of \(\calM^+\) together with their duals in \(\calB\) generate the objects of \(\calB\). Suppose also that
  the category \(\calM^+\) is a subcategory of a rigid
  category \(\calM\) and the functor \(\calG\) extends to a functor
  \(\tilde\calG\colon\calM \mapto \calB\). Then \(\tilde\calG\) is also full.
\end{remark}

We state also the following result, which seems actually to be known, but we have not been able to find a reference in full generality.  The analogous statement in the non-quantized setting can be found in \cite[Theorem~7.8]{MR2955190}.

\begin{prop}[Super mixed quantum Schur-Weyl duality]
  \label{prop:11}
  The Reshetikhin-Turaev functor \(\RT_{m|n}\) induces a full functor \(\uBr(d) \mapto \catRep_{m|n}\). In particular, for any \(r,s \geq 0\)  we have a surjective map
  \begin{equation}
    \label{eq:1}
    \Br_{\boldeta_{r,s}}(d) \longrightarrow \End_{\catRep_{m|n}}\big((\C_q^{m|n})^{\otimes r} \otimes (\C_q^{m|n})^{* \otimes s} \big).
  \end{equation}
  Moreover, if \((m+1)(n+1) > r+s\) then this map is also injective.
\end{prop}

\begin{proof}
  The fullness of the functor \( \uBr(d) \mapto \catRep_{m|n}\) follows from the fullness of the functor \(\catH \mapto \catRep^+_{m|n}\) exactly  as in the proof of proposition~\ref{prop:7} (cf.\ also remark~\ref{rem:2}). If \((m+1)(n+1)>r+s\) then it follows that \eqref{eq:1} is faithful by comparing dimensions, since \(\dim \Br_{\boldeta_{r,s}}(d) = (r+s)! = \dim \ucalH_{r+s}\) and
  \begin{equation}
    \label{eq:10}
    \dim \End_{\catRep_{m|n}}\big((\C_q^{m|n})^{\otimes r} \otimes (\C_q^{m|n})^{* \otimes s} \big) = \dim \End_{\catRep_{m|n}}\big((\C_q^{m|n})^{\otimes r+s} \big) = \dim \ucalH_{r+s}.
  \end{equation}
 The last equation follows by super Schur-Weyl duality (proposition~\ref{prop:1}), since all partitions of \(r+s\) are in \(H_{m|n}\) if \((m+1)(n+1)>r+s\).
\end{proof}

We define also the following morphisms:
\begin{align}
   \sfE^{(r)}\idem_{a \otimes b^*} & =\;
        \begin{tikzpicture}[int,smallnodes,baseline={([yshift=-0.5ex]0,1)},xscale=0.5]
   \draw (0,0.25) node[below] {$\vphantom{b^*}\mathclap{a}$} -- ++(0,1.5) node[aboverotated] {$a{+}r$} \nendarrow \nstartarrow;
   \draw (1,0.25) node[below] {$\vphantom{b}\mathclap{b^*}$} -- ++(0,1.5) node[aboverotated] {$(b{+}r)^*$} \nstartarrowrev \nendarrowrev;
     \draw (0,1) ++(0,+0.1) .. controls ++(0.5,-0.2) and ++(-0.5,-0.2)  .. node[above] {\(r\)} ++(1,0) \midarrowrev; 
     \coordinate (left) at (0,0);
     \coordinate (right) at (1,0);
    \coordinate (top) at (current bounding box.north);
    \coordinate (bottom) at (current bounding box.south);
    \pgfresetboundingbox;
    \path [use as bounding box] (left|-bottom) rectangle (right|-top);
    \end{tikzpicture} \; = \; \begin{tikzpicture}[int,smallnodes,baseline={([yshift=-0.5ex]0,1.25)},xscale=0.5]
   \draw (0,0.5) node[below] {$\vphantom{b^*}\mathclap{a}$} -- ++(0,1.5) node[aboverotated] {${a{+}r}$} \midarrow;
   \draw (3,0.5) node[below] {$\vphantom{b^*}\mathclap{b^*}$} -- ++(0,1.5) node[aboverotated] {${(b{+}r)^*}$} \midarrowrev;
     \draw (0,1.5) -- node[above] {$r$} ++(1,-0.3) \midarrowrev -- ++(0,-0.4) \midarrowrev .. controls ++(0.5,-0.2) and ++(-0.5,-0.2)  .. node[above] {\(r\)} ++(1,0) \midarrowrev -- ++(0,0.4) \midarrowrev -- node[above] {$r$} ++(1,0.3) \midarrowrev; 
     \draw[zero] (1,1.2)  -- ++(0,0.8);
     \draw[zero] (2,1.2)  -- ++(0,0.8);
     \draw[zero] (1,0.5)  -- ++(0,0.3);
     \draw[zero] (2,0.5)  -- ++(0,0.3);
     \coordinate (left) at (0,0);
     \coordinate (right) at (3,0);
    \coordinate (top) at (current bounding box.north);
    \coordinate (bottom) at (current bounding box.south);
    \pgfresetboundingbox;
    \path [use as bounding box] (left|-bottom) rectangle (right|-top);
    \end{tikzpicture}\;,
&\qquad
 \sfF^{(r)}\idem_{a \otimes b^*} & = \; 
        \begin{tikzpicture}[int,smallnodes,baseline={([yshift=-0.5ex]0,1)},xscale=0.5]
   \draw (0,0.25) node[below] {$\vphantom{b^*}\mathclap{a}$} -- ++(0,1.5) node[aboverotated] {${a{-}r}$} \nendarrow \nstartarrow;
   \draw (1,0.25) node[below] {$\vphantom{b}\mathclap{b^*}$} -- ++(0,1.5) node[aboverotated] {${(b{-}r)^*}$} \nstartarrowrev \nendarrowrev;
     \draw (0,1) ++(0,-0.1) .. controls ++(0.5,+0.2) and ++(-0.5,+0.2)  .. node[above] {\(r\)} ++(1,0) \midarrow; 
     \coordinate (left) at (0,0);
     \coordinate (right) at (1,0);
    \coordinate (top) at (current bounding box.north);
    \coordinate (bottom) at (current bounding box.south);
    \pgfresetboundingbox;
    \path [use as bounding box] (left|-bottom) rectangle (right|-top);
    \end{tikzpicture} \; = \;
      \begin{tikzpicture}[int,smallnodes,baseline={([yshift=-0.5ex]0,0.75)},xscale=0.5]
   \draw (0,0) node[below] {$\vphantom{b^*}\mathclap{a}$} -- ++(0,1.5) node[aboverotated] {${a{-}r}$} \midarrow;
   \draw (3,0) node[below] {$\vphantom{b^*}\mathclap{b^*}$} -- ++(0,1.5) node[aboverotated] {${(b{-}r)^*}$} \midarrowrev;
     \draw (0,0.5) -- node[above] {$r$} ++(1,0.3) \midarrow -- ++(0,0.4) \midarrow .. controls ++(0.5,0.2) and ++(-0.5,0.2)  .. node[above] {\(r\)} ++(1,0) \midarrow -- ++(0,-0.4) \midarrow -- node[above] {$r$} ++(1,-0.3) \midarrow; 
     \draw[zero] (1,1.2)  -- ++(0,0.3);
     \draw[zero] (2,1.2)  -- ++(0,0.3);
     \draw[zero] (1,0)  -- ++(0,0.8);
     \draw[zero] (2,0)  -- ++(0,0.8);
     \coordinate (left) at (0,0);
     \coordinate (right) at (3,0);
    \coordinate (top) at (current bounding box.north);
    \coordinate (bottom) at (current bounding box.south);
    \pgfresetboundingbox;
    \path [use as bounding box] (left|-bottom) rectangle (right|-top);
    \end{tikzpicture}\;, \label{eq:116}\\
    \sfE^{(r)}\idem_{a^* \otimes b} & =\;
        \begin{tikzpicture}[int,smallnodes,baseline={([yshift=-0.5ex]0,1)},xscale=0.5]
   \draw (0,0.25) node[below] {$\vphantom{b}\mathclap{a^*}$} -- ++(0,1.5) node[aboverotated] {${(a{-}r)^*}$} \nendarrowrev \nstartarrowrev;
   \draw (1,0.25) node[below] {$\vphantom{b}\mathclap{b}$} -- ++(0,1.5) node[aboverotated] {${b-r}$} \nstartarrow \nendarrow;
     \draw (0,1) ++(0,-0.1) .. controls ++(0.5,+0.2) and ++(-0.5,+0.2)  .. node[above] {\(r\)} ++(1,0) \midarrowrev; 
     \coordinate (left) at (0,0);
     \coordinate (right) at (1,0);
    \coordinate (top) at (current bounding box.north);
    \coordinate (bottom) at (current bounding box.south);
    \pgfresetboundingbox;
    \path [use as bounding box] (left|-bottom) rectangle (right|-top);
    \end{tikzpicture} \; = \;
      \begin{tikzpicture}[int,smallnodes,baseline={([yshift=-0.5ex]0,0.75)},xscale=0.5]
   \draw (0,0) node[below] {$\vphantom{b}\mathclap{a^*}$} -- ++(0,1.5) node[aboverotated] {${(a{-}r)^*}$} \midarrowrev;
   \draw (3,0) node[below] {$\vphantom{b^*}\mathclap{b^*}$} -- ++(0,1.5) node[aboverotated] {${b{-}r}$} \midarrow;
     \draw (0,0.5) -- node[above] {$r$} ++(1,0.3) \midarrowrev -- ++(0,0.4) \midarrowrev .. controls ++(0.5,0.2) and ++(-0.5,0.2)  .. node[above] {\(r\)} ++(1,0) \midarrowrev -- ++(0,-0.4) \midarrowrev -- node[above] {$r$} ++(1,-0.3) \midarrowrev; 
     \draw[zero] (1,1.2)  -- ++(0,0.3);
     \draw[zero] (2,1.2)  -- ++(0,0.3);
     \draw[zero] (1,0)  -- ++(0,0.8);
     \draw[zero] (2,0)  -- ++(0,0.8);
     \coordinate (left) at (0,0);
     \coordinate (right) at (3,0);
    \coordinate (top) at (current bounding box.north);
    \coordinate (bottom) at (current bounding box.south);
    \pgfresetboundingbox;
    \path [use as bounding box] (left|-bottom) rectangle (right|-top);
    \end{tikzpicture}\;, & 
    \sfF^{(r)} \idem_{a^* \otimes b}       & =\; \begin{tikzpicture}[int,smallnodes,baseline={([yshift=-0.5ex]0,1)},xscale=0.5]
   \draw (0,0.25) node[below] {$\vphantom{b}\mathclap{a^*}$} -- ++(0,1.5) node[aboverotated] {${(a{+}r)^*}$} \nendarrowrev \nstartarrowrev;
   \draw (1,0.25) node[below] {$\vphantom{b}\mathclap{b}$} -- ++(0,1.5) node[aboverotated] {${b{+}r}$} \nstartarrow \nendarrow;
     \draw (0,1) ++(0,+0.1) .. controls ++(0.5,-0.2) and ++(-0.5,-0.2)  .. node[above] {\(r\)} ++(1,0) \midarrow; 
     \coordinate (left) at (0,0);
     \coordinate (right) at (1,0);
    \coordinate (top) at (current bounding box.north);
    \coordinate (bottom) at (current bounding box.south);
    \pgfresetboundingbox;
    \path [use as bounding box] (left|-bottom) rectangle (right|-top);
    \end{tikzpicture} \; = \; \begin{tikzpicture}[int,smallnodes,baseline={([yshift=-0.5ex]0,1.25)},xscale=0.5]
   \draw (0,0.5) node[below] {$\vphantom{b}\mathclap{a^*}$} -- ++(0,1.5) node[aboverotated] {${(a{+}r)^*}$} \midarrowrev;
   \draw (3,0.5) node[below] {$\vphantom{b^*}\mathclap{b}$} -- ++(0,1.5) node[aboverotated] {${b{+}r}$} \midarrow;
     \draw (0,1.5) -- node[above] {$r$} ++(1,-0.3) \midarrow -- ++(0,-0.4) \midarrow .. controls ++(0.5,-0.2) and ++(-0.5,-0.2)  .. node[above] {\(r\)} ++(1,0) \midarrow -- ++(0,0.4) \midarrow -- node[above] {$r$} ++(1,0.3) \midarrow; 
     \draw[zero] (1,1.2)  -- ++(0,0.8);
     \draw[zero] (2,1.2)  -- ++(0,0.8);
     \draw[zero] (1,0.5)  -- ++(0,0.3);
     \draw[zero] (2,0.5)  -- ++(0,0.3);
     \coordinate (left) at (0,0);
     \coordinate (right) at (3,0);
    \coordinate (top) at (current bounding box.north);
    \coordinate (bottom) at (current bounding box.south);
    \pgfresetboundingbox;
    \path [use as bounding box] (left|-bottom) rectangle (right|-top);
    \end{tikzpicture}\;. \label{eq:118}
\end{align}

\subsection{The category \texorpdfstring{$\Spider(\beta)$}{Sp(b)}}
\label{sec:category-spiderbeta}

As we have seen, via the functor \(\calG_{m|n}\) the category \(\Spider\) gives all morphisms of \(\catRep_{m|n}\). However, \(\Spider\) is somehow too big (any non-trivial homomorphism space is infinite dimensional, since it contains the diagram consisting of a disjoint union of an arbitrary number of bubbles). We define hence a quotient \(\Spider(\beta)\), depending on the choice of our element \(q^\beta\).

\begin{definition}
  \label{def:9}
  We define the category \(\Spider(\beta)\) to be the quotient of the category \(\Spider\) modulo the relations
\begin{subequations}
\label{eq:84}
\begingroup
\allowdisplaybreaks
\tikzset{every picture/.style={scale=0.7}}
  \begin{equation}
    \begin{tikzpicture}[smallnodes,anchorbase]
   \draw[uno] (0,0) node[below]  {$\mathclap{1}$} -- ++(0,0.6) \midarrow;
   \draw[uno] (0,2) node[above]  {$\mathclap{1}$} -- ++(0,-0.6) \midarrowrev .. controls ++(0.5,0.2) and ++(0,+0.5)  .. ++(1,-0.4) \midarrow node[midway,above right] {$1$} .. controls ++(0,-0.5) and ++(0.5,-0.2) .. ++(-1,-0.4) \midarrow;
   \draw[int] (0,0.6) -- ++(0,0.8) \midarrow;
    \end{tikzpicture} \; =  [\beta-1] \;
    \begin{tikzpicture}[uno,smallnodes,anchorbase]
   \draw[uno,arrowinthemiddle] (0,0) node[below] {$\mathclap{1}$} -- ++(0,2) node[above] {$\mathclap{1}$};
    \end{tikzpicture}\;,  \qquad \qquad 
    \begin{tikzpicture}[smallnodes,anchorbase]
   \draw[uno] (0,0) node[below]  {$\mathclap{1}$} -- ++(0,0.6) \midarrow;
   \draw[uno] (0,2) node[above]  {$\mathclap{1}$} -- ++(0,-0.6) \midarrowrev .. controls ++(-0.5,0.2) and ++(0,+0.5)  .. ++(-1,-0.4) \midarrow node[midway,above left] {$1$} .. controls ++(0,-0.5) and ++(-0.5,-0.2) .. ++(1,-0.4) \midarrow;
   \draw[int] (0,0.6) -- ++(0,0.8) \midarrow;
    \end{tikzpicture} \; =  [\beta-1] \;
    \begin{tikzpicture}[uno,smallnodes,anchorbase]
   \draw[uno,arrowinthemiddle] (0,0) node[below] {$\mathclap{1}$} -- ++(0,2) node[above] {$\mathclap{1}$};
    \end{tikzpicture}\;,\label{eq:119}
\end{equation}
and
\begin{align} 
    \begin{tikzpicture}[smallnodes,anchorbase]
      \draw[uno,->] (0,0) arc (0:360:0.5cm)  node[midway,left]  {$1$};
    \end{tikzpicture} \; = \;
    \begin{tikzpicture}[smallnodes,anchorbase]
      \draw[uno,<-] (0,0) arc (0:360:0.5cm) node[midway,left]  {$1$};
    \end{tikzpicture} \; &= [\beta], \label{eq:120}\\
   \begin{tikzpicture}[smallnodes,anchorbase]
   \draw[uno] (0,0) node[below]  {$\vphantom{1^*}\mathclap{1}$} -- ++(0,0.5) \midarrow .. controls ++(0.5,0.2) and ++(-0.5,0.2)  .. ++(1,0) \midarrow -- ++(0,-0.5) \midarrow node[below] {$\mathclap{1^*}$};
   \draw[uno] (0,2) node[above]  {$\vphantom{1^*}\mathclap{1}$} -- ++(0,-0.5) \midarrowrev .. controls ++(0.5,-0.2) and ++(-0.5,-0.2)  .. ++(1,0) \midarrowrev -- ++(0,0.5) \midarrowrev node[above] {$\mathclap{1^*}$};
    \end{tikzpicture} \; -  \;
    \begin{tikzpicture}[smallnodes,anchorbase]
   \draw[uno] (0,0) node[below]  {$\vphantom{1^*}\mathclap{1}$} -- ++(0,0.6) \midarrow .. controls ++(0.5,-0.2) and ++(-0.5,-0.2)  .. ++(1,0) node[below,midway] {$1$} \midarrowrev -- ++(0,-0.6) \midarrow node[below] {$\mathclap{1^*}$};
   \draw[uno] (0,2) node[above]  {$\vphantom{1^*}\mathclap{1}$} -- ++(0,-0.6) \midarrowrev .. controls ++(0.5,0.2) and ++(-0.5,0.2)  .. ++(1,0) node[above,midway] {$1$} \midarrow -- ++(0,0.6) \midarrowrev node[above] {$\mathclap{1^*}$};
   \draw[int] (0,0.6) -- ++(0,0.8) \midarrow;
   \draw[int] (1,0.6) -- ++(0,0.8) \midarrowrev;
    \end{tikzpicture} \; &=  [2-\beta] \;
    \begin{tikzpicture}[uno,smallnodes,anchorbase]
   \draw[arrowinthemiddle] (0,0) node[below] {$\mathclap{1}$} -- ++(0,2) node[above] {$\vphantom{1^*}\mathclap{1}$};
   \draw[arrowinthemiddlerev] (1,0) node[below] {$\mathclap{1}$} -- ++(0,2) node[above] {$\mathclap{1^*}$};
    \end{tikzpicture}\;,\label{eq:121}\\
    \begin{tikzpicture}[smallnodes,anchorbase]
   \draw[uno] (0,0) node[below]  {$\mathclap{1}\vphantom{1^*}$} -- ++(0,0.5) \midarrow .. controls ++(-0.5,0.2) and ++(0.5,0.2)  .. ++(-1,0) \midarrow -- ++(0,-0.5) \midarrow node[below] {$\mathclap{1^*}$};
   \draw[uno] (0,2) node[above]  {$\mathclap{1}\vphantom{1^*}$} -- ++(0,-0.5) \midarrowrev .. controls ++(-0.5,-0.2) and ++(0.5,-0.2)  .. ++(-1,0) \midarrowrev -- ++(0,0.5) \midarrowrev node[above] {$\mathclap{1^*}$};
    \end{tikzpicture} \; -  \;
    \begin{tikzpicture}[smallnodes,anchorbase]
   \draw[uno] (0,0) node[below]  {$\vphantom{1^*}\mathclap{1}$} -- ++(0,0.6) \midarrow .. controls ++(-0.5,-0.2) and ++(0.5,-0.2)  .. ++(-1,0) node[below,midway] {$1$} \midarrowrev -- ++(0,-0.6) \midarrow node[below] {$\mathclap{1^*}$};
   \draw[uno] (0,2) node[above]  {$\vphantom{1^*}\mathclap{1}$} -- ++(0,-0.6) \midarrowrev .. controls ++(-0.5,0.2) and ++(0.5,0.2)  .. ++(-1,0) node[above,midway] {$1$} \midarrow -- ++(0,0.6) \midarrowrev node[above] {$\mathclap{1^*}$};
   \draw[int] (0,0.6) -- ++(0,0.8) \midarrow;
   \draw[int] (-1,0.6) -- ++(0,0.8) \midarrowrev;
    \end{tikzpicture} \; &=  [2-\beta] \;
    \begin{tikzpicture}[uno,smallnodes,anchorbase]
   \draw[arrowinthemiddle] (0,0) node[below] {$\vphantom{1^*}\mathclap{1}$} -- ++(0,2) node[above] {$\vphantom{1^*}\mathclap{1}$};
   \draw[arrowinthemiddlerev] (-1,0) node[below] {$\mathclap{1^*}$} -- ++(0,2) node[above] {$\mathclap{1^*}$};
    \end{tikzpicture}\;.\label{eq:122}
  \end{align}
\endgroup
\end{subequations}
\end{definition}

Recall that we set \(d=m-n\). We have then:

\begin{theorem}
  \label{thm:6}
  The functor \(\calG_{m|n} \colon  \Spider \mapto \catRep_{m|n}\) descends to a full functor 
  \begin{equation}
\calG_{m|n} \colon \Spider(d) \mapto \catRep_{m|n}. \label{eq:205}
\end{equation}
\end{theorem}

\begin{proof}
  We need to check that relations~\eqref{eq:84} are satisfied by morphisms in \(\catRep_{m|n}\). Relation~\eqref{eq:120} is \eqref{eq:44}. Let us check the first one of \eqref{eq:119} (the second one is analogous). Using \eqref{eq:86} we have in \(\catRep_{m|n}\):
\begingroup
\tikzset{every picture/.style={yscale=0.7}}
  \begin{equation}
    \label{eq:123}
     \begin{tikzpicture}[smallnodes,anchorbase,xscale=0.8]
   \draw[uno] (0,0) node[below]  {$\mathclap{\vphantom{1^*}1}$} -- ++(0,0.6) \midarrow;
   \draw[uno] (0,2) node[above]  {$\mathclap{\vphantom{1^*}1}$} -- ++(0,-0.6) \midarrowrev .. controls ++(0.5,0.2) and ++(0,+0.5)  .. ++(1,-0.4) \midarrow node[midway,above right] {$1$} .. controls ++(0,-0.5) and ++(0.5,-0.2) .. ++(-1,-0.4) \midarrow;
   \draw[int] (0,0.6) -- ++(0,0.8) \midarrow;
    \end{tikzpicture}\;  =  - \;
    \begin{tikzpicture}[smallnodes,anchorbase]
   \draw[uno] (0.8,1)  .. controls ++(0,-0.3) and ++(0.3,0) .. ++(-0.3,-0.6)  .. controls ++(-0.5,0) and ++(0,-0.7) .. ++(-0.5,1.6) \nendarrow node[above] {$\mathclap{\vphantom{1^*}1}$} ;
   \draw[uno, cross line] (0.8,1)  .. controls ++(0,0.3) and ++(0.3,0) .. ++(-0.3,0.6) \startarrowrev .. controls ++(-0.5,0) and ++(0,0.7) .. ++(-0.5,-1.6) \nendarrowrev node[below] {$\mathclap{\vphantom{1^*}1}$};
    \end{tikzpicture} \; +q^{-1} \;
    \begin{tikzpicture}[smallnodes,anchorbase]
      \draw[uno] (0,0) node[below] {$\mathclap{\vphantom{1^*}1}$} -- ++(0,2) \midarrow node[above] {$\mathclap{\vphantom{1^*}1}$};
      \draw[uno] (0.2,1)   .. controls ++(0,-0.3) and ++(-0.3,0) .. ++(0.3,-0.6) .. controls ++(0.3,0) and ++(0,-0.3) .. ++(0.3,0.6) \startarrowrev .. controls ++(0,0.3) and ++(0.3,0) .. ++(-0.3,0.6) node [above left] {$1$} .. controls ++(-0.3,0) and ++(0,0.3) .. ++(-0.3,-0.6) \startarrowrev;
    \end{tikzpicture} \; = -q^{-d} + q^{-1}[d]\;
    \begin{tikzpicture}[smallnodes,anchorbase]
      \draw[uno] (0,0) node[below] {$\mathclap{\vphantom{1^*}1}$} -- ++(0,2) \midarrow node[above] {$\mathclap{\vphantom{1^*}1}$};
    \end{tikzpicture}\; = [d-1]\;
    \begin{tikzpicture}[smallnodes,anchorbase]
      \draw[uno] (0,0) node[below] {$\mathclap{\vphantom{1^*}1}$} -- ++(0,2) \midarrow node[above] {$\mathclap{\vphantom{1^*}1}$};
    \end{tikzpicture}\;.
  \end{equation}
We now check \eqref{eq:121}. Again using \eqref{eq:86} we have in \(\catRep_{m|n}\):
  \begin{multline}
    \label{eq:124}
    \begin{tikzpicture}[smallnodes,anchorbase,xscale=0.8]
   \draw[uno] (0,0) node[below]  {$\mathclap{\vphantom{1^*}1}$} -- ++(0,0.6) \midarrow .. controls ++(0.5,-0.2) and ++(-0.5,-0.2)  .. ++(1,0) \midarrowrev -- ++(0,-0.6) \midarrow node[below] {$\mathclap{1^*}$};
   \draw[uno] (0,2) node[above]  {$\mathclap{\vphantom{1^*}1}$} -- ++(0,-0.6) \midarrowrev .. controls ++(0.5,0.2) and ++(-0.5,0.2)  .. ++(1,0) \midarrow -- ++(0,0.6) \midarrowrev node[above] {$\mathclap{1^*}$};
   \draw[int] (0,0.6) -- ++(0,0.8) \midarrow;
   \draw[int] (1,0.6) -- ++(0,0.8) \midarrowrev;
    \end{tikzpicture}   =  \;
    \begin{tikzpicture}[smallnodes,anchorbase,xscale=0.8]
   \draw[uno, cross line] (0,2) node[above]  {$\mathclap{\vphantom{1^*}1}$} .. controls ++(0,-0.7) and ++(-0.5,0) .. ++(0.5,-1.6) \nstartarrowrev .. controls ++(0.5,0) and ++(0,-0.7) .. ++(0.5,1.6) \nendarrowrev node[above] {$\mathclap{1^*}$} ;
   \draw[uno, cross line] (0,0) node[below]  {$\mathclap{\vphantom{1^*}1}$} .. controls ++(0,0.7) and ++(-0.5,0) .. ++(0.5,1.6) \nstartarrow .. controls ++(0.5,0) and ++(0,0.7) .. ++(0.5,-1.6) \nendarrow node[below] {$\mathclap{1^*}$};
    \end{tikzpicture} \;  -q^{-1} \;
    \begin{tikzpicture}[smallnodes,anchorbase,xscale=0.8]
      \draw[uno] (0,0) node[below] {$\mathclap{\vphantom{1^*}1}$} -- ++(0,2) \midarrow node[above] {$\mathclap{\vphantom{1^*}1}$};
   \draw[uno] (0.2,1)  .. controls ++(0,-0.3) and ++(-0.3,0) .. ++(0.3,-0.6)  .. controls ++(0.5,0) and ++(0,-0.7) .. ++(0.5,1.6) \nendarrowrev node[above] {$\mathclap{1^*}$} ;
   \draw[uno, cross line] (0.2,1)  .. controls ++(0,0.3) and ++(-0.3,0) .. ++(0.3,0.6) \startarrow .. controls ++(0.5,0) and ++(0,0.7) .. ++(0.5,-1.6) \nendarrow node[below] {$\mathclap{1^*}$};
    \end{tikzpicture} \;  -q \;
    \begin{tikzpicture}[smallnodes,anchorbase,xscale=0.8]
      \draw[uno] (1,0) node[below] {$\mathclap{1^*}$} -- ++(0,2) \midarrowrev node[above] {$\mathclap{1^*}$};
   \draw[uno] (0.8,1)  .. controls ++(0,-0.3) and ++(0.3,0) .. ++(-0.3,-0.6)  .. controls ++(-0.5,0) and ++(0,-0.7) .. ++(-0.5,1.6) \nendarrowrev node[above] {$\mathclap{\vphantom{1^*}1}$} ;
   \draw[uno, cross line] (0.8,1)  .. controls ++(0,0.3) and ++(0.3,0) .. ++(-0.3,0.6) \startarrowrev .. controls ++(-0.5,0) and ++(0,0.7) .. ++(-0.5,-1.6) \nendarrow node[below] {$\mathclap{\vphantom{1^*}1}$};
    \end{tikzpicture} \; + \;
    \begin{tikzpicture}[smallnodes,anchorbase]
      \draw[uno] (0,0) node[below] {$\mathclap{\vphantom{1^*}1}$} -- ++(0,2) \midarrow node[above] {$\mathclap{\vphantom{1^*}1}$};
      \draw[uno] (1,0) node[below] {$\mathclap{1^*}$} -- ++(0,2) \midarrowrev node[above] {$\mathclap{1^*}$};
      \draw[uno] (0.2,1)   .. controls ++(0,-0.3) and ++(-0.3,0) .. ++(0.3,-0.6) .. controls ++(0.3,0) and ++(0,-0.3) .. ++(0.3,0.6) \startarrowrev .. controls ++(0,0.3) and ++(0.3,0) .. ++(-0.3,0.6) node [above left] {$1$} .. controls ++(-0.3,0) and ++(0,0.3) .. ++(-0.3,-0.6) \startarrow;
    \end{tikzpicture} \\
    =\;
    \begin{tikzpicture}[smallnodes,anchorbase,xscale=0.6,yscale=0.7]
   \draw[uno, cross line] (0,2) node[above]  {$\mathclap{\vphantom{1^*}1}$} .. controls ++(0,-0.3) and ++(-0.5,0) .. ++(0.5,-0.8) \nstartarrowrev .. controls ++(0.5,0) and ++(0,-0.3) .. ++(0.5,0.8) \nendarrowrev node[above] {$\mathclap{1^*}$} ;
   \draw[uno, cross line] (0,0) node[below]  {$\mathclap{\vphantom{1^*}1}$} .. controls ++(0,0.3) and ++(-0.5,0) .. ++(0.5,0.8) \nstartarrow .. controls ++(0.5,0) and ++(0,0.3) .. ++(0.5,-0.8) \nendarrow node[below] {$\mathclap{1^*}$};
    \end{tikzpicture} \;  -q^{-1+d} \;
    \begin{tikzpicture}[smallnodes,anchorbase,xscale=0.6,yscale=0.7]
      \draw[uno] (0,0) node[below] {$\mathclap{\vphantom{1^*}1}$} -- ++(0,2) \midarrow node[above] {$\mathclap{\vphantom{1^*}1}$};
      \draw[uno] (1,0) node[below] {$\mathclap{1^*}$} -- ++(0,2) \midarrowrev node[above] {$\mathclap{1^*}$};
    \end{tikzpicture} \;  -q^{1-d} \;
    \begin{tikzpicture}[smallnodes,anchorbase,xscale=0.6,yscale=0.7]
      \draw[uno] (0,0) node[below] {$\mathclap{\vphantom{1^*}1}$} -- ++(0,2) \midarrow node[above] {$\mathclap{\vphantom{1^*}1}$};
      \draw[uno] (1,0) node[below] {$\mathclap{1^*}$} -- ++(0,2) \midarrowrev node[above] {$\mathclap{1^*}$};
    \end{tikzpicture} \; + [d] \;
    \begin{tikzpicture}[smallnodes,anchorbase,xscale=0.6,yscale=0.7]
      \draw[uno] (0,0) node[below] {$\mathclap{\vphantom{1^*}1}$} -- ++(0,2) \midarrow node[above] {$\mathclap{\vphantom{1^*}1}$};
      \draw[uno] (1,0) node[below] {$\mathclap{1^*}$} -- ++(0,2) \midarrowrev node[above] {$\mathclap{1^*}$};
    \end{tikzpicture}
  \end{multline}
\endgroup
  and the claim follows since \(-q^{-1+d}-q^{1-d}+[d]=[d-2]\). Analogously one can check \eqref{eq:122}.
\end{proof}

\begin{prop}\label{prop:10}
  The functor \(\calI^+ \colon \catH \mapto \Spider^+\) extends to a fully faithful functor \(\calI_d \colon \uBr(\beta) \mapto \Spider(\beta)\).
\end{prop}

\begin{proof}
  Since the oriented Brauer category \(\uBr(\beta)\) is generated by the braiding from \(\catH\) together with the adjunction maps (cups and caps), there is a unique natural way to extend the functor \(\calI^+\) to \(\uBr(\beta)\). We should check that the defining relations of \(\uBr(\beta)\) hold in the image. Relation~\eqref{eq:52} is a direct consequence of \eqref{eq:120}, while relations~\eqref{eq:54} can be showed by the same computation as in \eqref{eq:123}.

 Now, let us first prove that \(\calI_d\) is full. We will mostly deduce this from the fact that \(\calI\colon\catH\mapto \Spider^+\) is fully faithful (proposition~\ref{prop:4}). Consider a diagram element \(\phi \in \Hom_{\Spider(\beta)}(\bolda,\boldb)\) with \(\bolda\) and \(\boldb\) containing only copies of \(1\) and \(1^*\). We slice \(\phi\) into elementary pieces, as \eqref{eq:70}, \eqref{eq:101} and \eqref{eq:107}, tensored with identities. Of course the labels that appear in the slices can be greater than \(1\). We thus replace, at the bottom and the top of each slice, each higher label \(a\) (respectively, \(a^*\)) by \(1^{\otimes a } \) (respectively, \((1^*)^{\otimes a}\)) by inserting \(\pi_a \iota_a\) (respectively, its dual), where \(\iota_a\) and \(\pi_a\) are the injection and the projection from lemma~\ref{lem:4}.
Moreover, we can suppose that all caps and cups are labeled by \(1\) using the following trick:
\begin{equation}\label{eq:51}
\begin{tikzpicture}[smallnodes,baseline=0.25cm,xscale=1]
  \draw[uno] (-1,0) node[below] {$\mathclap{1}\vphantom{{}^*}$}  --  ++(0,.2)   \midarrow ;
  \draw[uno] (0,0) node[below] {$\mathclap{1}\vphantom{{}^*}$}  --  ++(0,.2)   \midarrow ;
  \node at (-.5,.1) {$\cdots$};
  \draw[thin] (-1.1,.2) rectangle (.1,.5);
  \node at (-.5,.35) {$ \pi_a$};
  \draw[int] (-.5,.5) -- ++(0,0.8) .. controls ++(0,.7) and ++(0,.7) .. ++(2,0)  \midarrow -- ++(0,-0.8);
  \draw (.9,.5) rectangle (2.1,.2);
  \node at (1.5,.35) {$ \iota_a^*$};
  \draw[uno] (1,.2)   --  ++(0,-.2)   \midarrow node[below] {$\mathclap{1^*}$};
  \draw[uno] (2,.2)  --  ++(0,-.2)   \midarrow node[below] {$\mathclap{1^*}$};
  \node at (1.5,.1) {$\cdots$};
    \end{tikzpicture}
\;=\;
\begin{tikzpicture}[smallnodes,baseline=0.25cm,xscale=1]
  \draw[uno] (-1,0) node[below] {$\mathclap{1}\vphantom{{}^*}$}  --  ++(0,.2)   \midarrow ;
  \draw[uno] (0,0) node[below] {$\mathclap{1}\vphantom{{}^*}$}  --  ++(0,.2)   \midarrow ;
  \node at (-.5,.1) {$\cdots$};
  \draw[thin] (-1.1,.2) rectangle (.1,.5);
  \node at (-.5,.35) {$ \pi_a$};
  \draw[int] (-.5,.5)  --  ++(0,.2)   \midarrow ;
  \draw[thin] (-1.1,.7) rectangle (.1,1);
  \node at (-.5,.85) {$ \iota_a$};
   \draw[uno] (-1,1)   .. controls ++(0,1.1) and ++(0,1.1) ..  ++(3,0)   \midarrow -- ++(0,-1)\nendarrow node[below] {$\mathclap{1^*}$};
   \draw[uno] (0,1)   .. controls ++(0,.5) and ++(0,.5) ..  ++(1,0)   \midarrow -- ++(0,-1)\nendarrow node[below] {$\mathclap{1^*}$};
   \node[rotate=90] at (.5,1.6) {$\cdots$};
  \node at (1.5,.1) {$\cdots$};
    \end{tikzpicture}
\end{equation}
(and similarly in the case of cups).

Now we have sliced our diagram so that if one slice contains higher labels then it does not contain cups or caps (so if it contains higher labels then there is no interaction between upwards pointing and downwards pointing strands). The slices containing cups or caps lie in the image of \(\calI_\beta\), obviously. The other slices lie also in the image of \(\calI_\beta\), since \(\calI^+\) is full. Hence \(\calI_\beta\) is full.

In order to prove that \(\calI_\beta\) is faithful, consider first the case \(\beta=d\). For any objects \(\bolda,\boldb \in \uBr(d)\) we can choose \(m,n\) big enough so that the Reshetikhin-Turaev functor \(\RT_{m|n}\) induces an isomorphism \(\Hom_{\uBr(d)}(\bolda,\boldb) \mapto \Hom_{\catRep_{m|n}}(\bigwedgeq^\bolda,\bigwedgeq^\boldb)\) (this is an immediate generalization of proposition~\ref{prop:11}). Since this map factors through \(\Hom_{\uBr(d)}(\bolda,\boldb) \mapto \Hom_{\Spider(d)}(\bolda,\boldb)\), and we just proved that this is surjective, it follows that this is injective. The case \(\beta\) generic follows by specialization (formally, one needs to define the category \(\Spider(\beta)\) over \(\C(q)[q^\beta,q^{-\beta}]\), and then the argument before shows that \(\Hom_{\Spider(d)}(\bolda,\boldb)\) is free of the same rank as for \(\beta=d\)).
\end{proof}

By construction, there is an obvious functor \(\sfi \colon \Spider^+ \mapto \Spider\). We will denote also by \(\sfi \colon \Spider^+ \mapto \Spider(\beta)\) its composition with the quotient functor \(\Spider \mapto \Spider(\beta)\).

\begin{prop}
  \label{prop:12}
  The functor \(\sfi \colon \Spider^+ \mapto \Spider(\beta)\) is fully faithful.
\end{prop}

\begin{proof}
  We shall prove that \(\sfi\) induces an isomorphism \(\Hom_{\Spider^+}(\bolda,\boldb) \cong \Hom_{\Spider(\beta)}(\bolda,\boldb)\)  for all objects \(\bolda,\boldb \in \Spider^+\). First, we claim that it is enough to consider the case \(\bolda=\boldb=1^{\otimes N}\). Indeed, using \eqref{eq:81} we have a commutative diagram
\begin{equation}
  \label{eq:125}
  \begin{tikzpicture}[baseline=(current bounding box.center)]
  \matrix (m) [matrix of math nodes, row sep=3em, column
  sep=3.5em, text height=1.5ex, text depth=0.25ex] {
     \Hom_{\Spider^+} (\bolda, \boldb) & \Hom_{\Spider^+}(1^{\otimes N},1^{\otimes N})  & \Hom_{\Spider^+} (\bolda, \boldb) \\
     \Hom_{\Spider(\beta)} (\bolda, \boldb) & \Hom_{\Spider(\beta)}(1^{\otimes N},1^{\otimes N})  & \Hom_{\Spider(\beta)} (\bolda, \boldb) \\};
  \path[right hook->] (m-1-1) edge node[left] {} (m-1-2);
  \path[->>] (m-1-2) edge node[right] {} (m-1-3);
  \path[right hook->] (m-2-1) edge node[left] {} (m-2-2);
  \path[->>] (m-2-2) edge node[right] {} (m-2-3);
  \path[->] (m-1-1) edge node[left] {} (m-2-1);
  \path[->] (m-1-2) edge node[left] {} (m-2-2);
  \path[->] (m-1-3) edge node[left] {} (m-2-3);
\end{tikzpicture}
\end{equation}
Then supposing that the middle vertical map is an isomorphism, the left square gives the injectivity and the right square gives the surjectivity of the map \(\Hom_{\Spider^+}(\bolda,\boldb) \mapto \Hom_{\Spider(\beta)}(\bolda,\boldb)\), which is hence an isomorphism.

Se we can restrict to the case of \(1^{\otimes N}\). Since the functors \(\calI^+\) and \(\calI_\beta\) are fully faithful, we have isomorphisms \(\Hom_{\Spider^+}(1^{\otimes N},1^{\otimes N}) \cong \ucalH_N \cong \End_{\uBr(\beta)}(1^{\otimes N}) \cong \Hom_{\Spider(\beta)}(1^{\otimes N},1^{\otimes N})\), and we are done.
\end{proof}

As a corollary of proposition~\ref{prop:10} or alternatively of proposition~\ref{prop:12} we have the following important fact:

\begin{corollary}\label{cor:1}
  In \(\Spider(\beta)\) we have \(\End_{\Spider(\beta)}(a) = \C_q\) for all \(a \in \Z_{\geq 0}\).
\end{corollary}

\subsection{Ribbon structure}
\label{sec:ribbon-structure}

In \(\Spider(\beta)\) let us define 
\begingroup
\tikzset{every picture/.style={yscale=0.8}}
\begin{align}
  \label{eq:143}
 c_{1^*,1} & =\; 
\;.
\end{align}
\endgroup

The notation we use is motivated by the following lemma:

\begin{lemma}
 \label{lem:14}
We have
\begingroup
\tikzset{every picture/.style={scale=0.7}}
\begin{equation}
  \label{eq:146}
%
 \; = \id_{1^* \otimes 1}.
  \end{equation} 
\endgroup
\end{lemma}

\begin{proof}
 Let us just check one of the four equalities, the other ones being similar (actually, one of the four even implies the other three using the properties of the braiding \(c_{1,1}\) and of the duality, see \cite{MR1881401} and \cite{MR1292673}). We compute:
\begingroup
\allowdisplaybreaks
\tikzset{every picture/.style={scale=0.5}}
  \begin{equation}
    \label{eq:147}
    \begin{aligned}[t]
        %
\;.
    \end{aligned}
  \end{equation}
  Let us expand the second summand:
  \begin{equation}
    \label{eq:148}
        %
\;.
\end{equation}
  Let us now expand the last summand:
  \begin{align}
    \label{eq:149}
          %
\;.\notag
\end{align}
\endgroup
Putting all together, the claim follows since
\begin{equation}
  \label{eq:150}
  [\beta]-[2][\beta-1]+[\beta-2]=-[\beta-2]+[\beta-2]=0
\end{equation}
and
\begin{equation}
  \label{eq:151}
[2-\beta][2][\beta-1]+[\beta-2][\beta-2]+[\beta-1][\beta-1]=[2-\beta][\beta]+[\beta-1][\beta-1]=1
\end{equation}
(using the formulas from section~\ref{sec:overview}).
\end{proof}

We define more generally candidate braiding elements \(c_{a,b^*}\), \(c_{a,b^*}^{-1}\), \(c_{a^*,b}\) and \(c_{a^*,b}^{-1}\) exactly as we did for \(a=b=1\). It follows from lemma~\ref{lem:14} and from the proof of lemma~\ref{lem:8}, and in particularly from \eqref{eq:90}, that they are inverse to each other.  We define moreover elements \(c_{a^*,b^*}\) and \(c_{a^*,b^*}^{-1}\) by taking the duals of the elements \(c_{a,b}\) and \(c_{a,b}^{-1}\), respectively. 

We define then recursively
\begin{align}
  \label{eq:152}
  c_{\bolda,\boldb \otimes \boldb'}& = (\id_\bolda \otimes c_{\bolda,\boldb'})(c_{\bolda,\boldb} \otimes \id_\boldb),\\
  c_{\bolda \otimes \bolda',\boldb}& = (c_{\bolda,\boldb} \otimes \id_{\bolda'})(\id_\bolda \otimes c_{\bolda',\boldb}).\label{eq:153}
\end{align}
Notice that this is well-defined, i.e.\ when defining \(c_{\bolda\otimes \bolda',\boldb \otimes \boldb'}\) it does not matter in which order we use \eqref{eq:152} and \eqref{eq:153} above.

\begin{lemma} \label{lem:15}
  The family of maps \(\{c_{\bolda,\boldb} \suchthat \bolda,\boldb \in \Spider(\beta)\}\) is natural.
\end{lemma}

\begin{proof}
  Let \(f\colon \bolda \mapto \bolda'\), \(g \colon \boldb \mapto \boldb'\). We need to show that \((g \otimes f) c_{\bolda,\boldb} = c_{\bolda',\boldb'}(f \otimes g)\). Since morphisms in \(\Spider(\beta)\) are generated by the morphisms \(\sfE^{(r)}\) and \(\sfF^{(r)}\) of \(\Spider^+\) and the adjunction morphisms \eqref{eq:101}, it suffices to consider such morphisms. For \(\sfE^{(r)}\) and \(\sfF^{(r)}\) the statement follows by the naturality of the braiding of \(\Spider^+\) (see proposition~\ref{prop:5}). Hence we only need to consider the adjunction morphisms. For them, the claim follows immediately from the fact that \(c_{\bolda,\boldb}^{-1}\) is the inverse of \(c_{\bolda,\boldb}\). 
\end{proof}

We will now define    a family of twists \(\{\theta_\bolda\}_{\bolda \in \Spider(\beta)}\) compatible with the braiding \(\{c_{\bolda,\boldb}\}_{\bolda,\boldb \in \Spider(\beta)}\). For \(a \in \Z_{\geq 0}\) set \(\theta_a \in \End_{\Spider(\beta)}(a)\) to be multiplication by \(q^{-a\beta+a(a-1)}\). In pictures
\begingroup
\tikzset{every picture/.style={yscale=0.8,xscale=0.5}}
\begin{equation}
  \label{eq:154}
\begin{tikzpicture}[smallnodes,anchorbase,xscale=1.3]
   \draw[int] (0,0) node[below] {\(a\)} -- (0,1.4) -- ++(0,0.1) node[above] {\(a\)} \nendarrow;
   \node[rectangle,draw,fill=white] at (0,0.75) {$\theta_{a}$};
 \end{tikzpicture}\; =  q^{-a\beta+a(a-1)} \;
\begin{tikzpicture}[smallnodes,anchorbase,xscale=1.3]
   \draw[int] (0,0) node[below] {\(\mathclap{a}\)} -- (0,1.5) node[above] {\(\mathclap{a}\)} \midarrow;
 \end{tikzpicture}\;.
\end{equation}
The choice of the scalar is made in such a way that the functor \(\calG_{m|n}\) sends \(\theta_a\) to the twist of \(\catRep_{m|n}\), cf.\ \eqref{eq:43}. Moreover we let \(\theta_{a^*}\) be the dual of \(\theta_a\), that is
\begin{equation}
  \label{eq:155}
\begin{tikzpicture}[smallnodes,anchorbase,xscale=1.3]
   \draw[int] (0,0) node[below] {\(a^*\)} -- ++(0,0.1) \nstartarrowrev -- ++(0,1.4)  node[above] {\(a^*\)}  ;
   \node[rectangle,draw,fill=white] at (0,0.75) {$\theta_{a^*}$};
 \end{tikzpicture}\; = \;
  \begin{tikzpicture}[smallnodes,anchorbase]
   \draw[int] (-1,1.5) node[above] {\(\mathclap{a^*}\)}  -- ++(0,-1) \midarrow .. controls ++(0,-0.5) and ++(0,-0.5) .. ++(1,0) \midarrow -- ++(0,0.5) \midarrow  .. controls ++(0,0.5) and ++(0,0.5) .. ++(1,0) \midarrow -- ++(0,-1) \midarrow node[below] {\(\mathclap{a^*}\)} ;
   \node[rectangle,draw,fill=white] at (0,0.75) {$\theta_{a}$};
 \end{tikzpicture}\; = \;
 \begin{tikzpicture}[smallnodes,anchorbase]
   \draw[int] (3,1.5) node[above] {\(\mathclap{a^*}\)}  -- ++(0,-1) \midarrow .. controls ++(0,-0.5) and ++(0,-0.5) .. ++(-1,0) \midarrow -- ++(0,0.5) \midarrow  .. controls ++(0,0.5) and ++(0,0.5) .. ++(-1,0) \midarrow -- ++(0,-1) \midarrow node[below] {\(\mathclap{a^*}\)};
   \node[rectangle,draw,fill=white] at (2,0.75) {$\theta_{a}$};
 \end{tikzpicture} \; = q^{-a\beta+a(a-1)}\;
\begin{tikzpicture}[smallnodes,anchorbase,xscale=1.3]
   \draw[int] (0,0) node[below] {\(\mathclap{a^*}\)} -- (0,1.5) node[above] {\(\mathclap{a^*}\)} \midarrowrev;
 \end{tikzpicture}\;.
\end{equation}
\endgroup
More generally for \(\bolda \in \Spider(\beta)\) define \(\theta_\bolda\) recursively:
\begin{equation}
  \label{eq:156}
  \theta_{\bolda \otimes \bolda'} = c_{\bolda',\bolda} c_{\bolda,\bolda'} (\theta_\bolda \otimes \theta_{\bolda'}) = (\theta_{\bolda'} \otimes \theta_{\bolda})c_{\bolda',\bolda} c_{\bolda,\bolda'}.
\end{equation}
The equality between the two expressions follows by the naturality of the braiding. Again using the naturality of the braiding it is easy to show that this gives a definition of \(\theta_\bolda\) independent of the order in which we perform the recursion.

\begin{lemma}
  The family of maps \(\{\theta_\bolda \suchthat \bolda \in \Spider(\beta)\}\) is natural.
\end{lemma}

\begin{proof}
  Let \(\phi\colon \bolda \mapto \bolda'\). We need to show that \(\phi \theta_\bolda = \theta_{\bolda'} \phi\). Again, it is sufficient to consider generating morphisms of \(\Spider(\beta)\). First, consider the case \(\phi = \sfE^{(r)} \colon a \otimes b \mapsto a+r \otimes b-r\) for \(a,b \in \Z_{\geq 0}\). By~\eqref{eq:76} it is enough to consider the case \(r=1\).
We have:
\begingroup
\tikzset{every picture/.style={yscale=0.8,xscale=0.5}}
\allowdisplaybreaks
\begin{multline}
  \begin{tikzpicture}[smallnodes,anchorbase]
    \draw[int] (0,0) node[below] {\(\vphantom{b}\mathclap{ a}\)}  -- ++ (0,2) --  ++(0,1) node[above] {\(\mathclap{a{+}1}\)} \nendarrow;
    \draw[int] (1.5,0) node[below] {\(\mathclap{ b}\)}  -- ++(0,2) -- ++ (0,1)  node[above] {\(\mathclap{ b{-}1}\)} \nendarrow;
    \draw (1.5,1.8) -- ++ (-1.5,.4) \midarrow;
   \node[rectangle,draw,fill=white,minimum width=1.2cm] at (.75,1) {$\theta_{a\otimes b}$};
 \end{tikzpicture}\; = \;
  \begin{tikzpicture}[smallnodes,anchorbase]
    \draw[int] (1.5,0) node[below] {\(\mathclap{ b}\)} -- ++(0,0.1) .. controls ++(0,0.2) and ++(0,-0.2) .. ++(-1.5,.4) -- ++(0,0.1)  \nendarrow;
    \draw[int,cross line] (0,0) node[below] {\(\vphantom{b}\mathclap{ a}\)} -- ++(0,0.1)  .. controls ++(0,0.2) and ++(0,-0.2) .. ++(1.5,.4) -- ++(0,0.1)  \nendarrow;
    \draw[int] (1.5,.6) node[below] {} -- ++(0,0.1) .. controls ++(0,0.2) and ++(0,-0.2) .. ++(-1.5,.4) -- ++(0,0.1)  \nendarrow;
    \draw[int,cross line] (0,.6) node[below] {} -- ++(0,0.1)  .. controls ++(0,0.2) and ++(0,-0.2) .. ++(1.5,.4) -- ++(0,0.1)  \nendarrow;
    \draw[int] (0,1.2) node[below] {}  -- ++ (0,.8) --  ++(0,1) node[above] {\(\mathclap{ a{+}1}\)} \nendarrow;
    \draw[int] (1.5,1.2) node[below] {}  -- ++(0,.8) -- ++ (0,1)  node[above] {\(\mathclap{ b{-}1}\)} \nendarrow;
    \draw (1.5,2.2) -- ++ (-1.5,.4) \midarrow;
   \node[rectangle,draw,fill=white] at (0,1.5) {$\theta_{a}$};
   \node[rectangle,draw,fill=white] at (1.5,1.5) {$\theta_{b}$};
 \end{tikzpicture}\; = q^{-(a+b)\beta+a(a-1) +b(b{-}1)}\;
  \begin{tikzpicture}[smallnodes,anchorbase]
    \draw[int] (1.5,0) node[below] {\(\mathclap{ b}\)} -- ++(0,0.1) .. controls ++(0,0.3) and ++(0,-0.3) .. ++(-1.5,.6) -- ++(0,0.1)  \nendarrow;
    \draw[int,cross line] (0,0) node[below] {\(\vphantom{b}\mathclap{ a}\)} -- ++(0,0.1)  .. controls ++(0,0.3) and ++(0,-0.3) .. ++(1.5,.6) -- ++(0,0.1)  \nendarrow;
    \draw[int] (1.5,.8) node[below] {} -- ++(0,0.1) .. controls ++(0,0.3) and ++(0,-0.3) .. ++(-1.5,.6) -- ++(0,0.1)  \nendarrow;
    \draw[int,cross line] (0,.8) node[below] {} -- ++(0,0.1)  .. controls ++(0,0.3) and ++(0,-0.3) .. ++(1.5,.6) -- ++(0,0.1)  \nendarrow;
    \draw[int] (0,1.6) node[below] {}  -- ++ (0,.4) --  ++(0,1) node[above] {\(\mathclap{ a{+}1}\)} \nendarrow;
    \draw[int] (1.5,1.6) node[below] {}  -- ++(0,.4) -- ++ (0,1)  node[above] {\(\mathclap{ b{-}1}\)} \nendarrow;
    \draw (1.5,2.2) -- ++ (-1.5,.4) \midarrow;
 \end{tikzpicture} \\
=q^{-(a+b)\beta+a(a-1) +b(b-1)+2(a-b)+2}\;
  \begin{tikzpicture}[smallnodes,anchorbase]
    \draw[int] (1.5,1.4) node[below] {} -- ++(0,0.1) .. controls ++(0,0.3) and ++(0,-0.3) .. ++(-1.5,.6) -- ++(0,0.1) \nendarrow;
   \draw[int,cross line] (0,1.4) node[below] {} -- ++(0,0.1)  .. controls ++(0,0.3) and ++(0,-0.3) .. ++(1.5,.6)  -- ++(0,0.1) \nendarrow;
    \draw[int] (1.5,2.2) node[below] {} -- ++(0,0.1) .. controls ++(0,0.3) and ++(0,-0.3) .. ++(-1.5,.6) \nendarrow -- ++(0,0.1)   node[above] {\(\mathclap{ a{+}1}\)};
    \draw[int,cross line] (0,2.2) node[below] {} -- ++(0,0.1)  .. controls ++(0,0.3) and ++(0,-0.3) .. ++(1.5,.6) \nendarrow -- ++(0,0.1)  node[above] {\(\mathclap{ b{-}1}\)};
    \draw[int] (0,0) node[below] {\(\vphantom{b}\mathclap{ a}\)}  -- ++ (0,1) \nstartarrow --  ++(0,.4) node[above] {};
    \draw[int] (1.5,0) node[below] {\(\mathclap{ b}\)}  -- ++(0,1) \nstartarrow-- ++ (0,.4)  node[above] {};
    \draw (1.5,.8) -- ++ (-1.5,.4) \midarrow;
 \end{tikzpicture} \;=\;
  \begin{tikzpicture}[smallnodes,anchorbase,xscale=1.2]
    \draw[int] (1.5,.8) node[below] {} -- ++(0,0.1) .. controls ++(0,0.2) and ++(0,-0.2) .. ++(-1.5,.4) -- ++(0,0.1)  \nendarrow ;
   \draw[int,cross line] (0,.8) node[below] {} -- ++(0,0.1)  .. controls ++(0,0.2) and ++(0,-0.2) .. ++(1.5,.4) -- ++(0,0.1)  \nendarrow ;
    \draw[int] (1.5,1.4) node[below] {} -- ++(0,0.1) .. controls ++(0,0.2) and ++(0,-0.2) .. ++(-1.5,.4) -- ++(0,0.1) -- ++(0,.8) \nendarrow node[above] {\(\mathclap{ a{+}1}\)};
    \draw[int,cross line] (0,1.4) node[below] {} -- ++(0,0.1)  .. controls ++(0,0.2) and ++(0,-0.2) .. ++(1.5,.4) -- ++(0,0.1)  -- ++(0,.8) \nendarrow node[above] {\(\mathclap{ b{-}1}\)};
    \draw[int] (0,0) node[below] {\(\vphantom{b}\mathclap{ a}\)}  -- ++ (0,.8) \nendarrow  node[above] {};
    \draw[int] (1.5,0) node[below] {\(\mathclap{ b}\)}  -- ++(0,.8) \nendarrow  node[above] {};
    \draw[uno] (1.5,.15) -- ++ (-1.5,.4) \midarrow;
   \node[rectangle,draw,fill=white] at (0,2.3) {$\theta_{a{+}1}$};
   \node[rectangle,draw,fill=white] at (1.5,2.3) {$\theta_{b{-}1}$};
 \end{tikzpicture}\;=\;
  \begin{tikzpicture}[smallnodes,anchorbase]
    \draw[int] (0,0) node[below] {\(\vphantom{b}\mathclap{ a}\)}  -- ++ (0,2) --  ++(0,1) node[above] {\(\mathclap{ a{+}1}\)} \nendarrow;
    \draw[int] (1.5,0) node[below] {\(\mathclap{ b}\)}  -- ++(0,2) -- ++ (0,1)  node[above] {\(\mathclap{ b{-}1}\)} \nendarrow;
    \draw[uno] (1.5,.8) -- ++ (-1.5,.4) \midarrow;
   \node[rectangle,draw,fill=white,minimum width=1.5cm] at (.75,2) {$\theta_{a{+}1\otimes b{-}1}$};
 \end{tikzpicture}\;.
\end{multline}
\endgroup
Here, the first equality is a consequence of the definition \eqref{eq:156} of the twist on a tensor product, the second one is the definition \eqref{eq:155} of the twist on a single strand, the third one is a consequence of \eqref{eq:EqW1} and \eqref{eq:EqW2}, and the last two equalities are symmetric to the two first ones. 
A similar computation yields the same result for \(\sfF^{(r)}\colon a \otimes b \mapto a-r \otimes b-r\) for \(a,b \in \Z_{\geq 0}\).

Second, we need to consider the case in which \(\phi \colon \varnothing \mapto a \otimes a^*\) is the duality map. By writing \(\id_a = \pi_a \otimes \iota_a\) (see lemma~\ref{lem:4}), since we already showed that the twist commutes with \(\pi_a\) and \(\iota_a\), it is enough to consider the case \(a=1\). We have
\begingroup
\tikzset{every picture/.style={yscale=0.8,xscale=0.8}}
\begin{equation}\label{eq:175}
  \begin{tikzpicture}[smallnodes,anchorbase]
    \draw[uno] (0,2.5) node[above] {\(\mathclap{ 1}\)}  -- ++ (0,-2) \nstartarrowrev arc(180:360:0.5) \midarrowrev -- ++(0,2) \nendarrowrev node[above] {\(\mathclap{1^*}\)} ;
   \node[rectangle,draw,fill=white,minimum width=1cm] at (.5,1) {$\theta_{1\otimes 1^*}$};
 \end{tikzpicture}\; = \;
  \begin{tikzpicture}[smallnodes,anchorbase]
    \draw[uno] (0,0) arc(180:360:0.5) \midarrowrev -- ++(0,0.1) .. controls ++(0,0.2) and ++(0,-0.2) .. ++(-1,.4) -- ++(0,0.1)  ;
    \draw[uno,cross line] (0,0)  -- ++(0,0.1)  .. controls ++(0,0.2) and ++(0,-0.2) .. ++(1,.4) -- ++(0,0.1)  ;
    \draw[uno] (1,.6)  -- ++(0,0.1) .. controls ++(0,0.2) and ++(0,-0.2) .. ++(-1,.4) -- ++(0,0.1)  ;
    \draw[uno,cross line] (0,.6)  -- ++(0,0.1)  .. controls ++(0,0.2) and ++(0,-0.2) .. ++(1,.4) -- ++(0,0.1)  ;
    \draw[uno] (0,1.2)   --  ++(0,0.9) node[above] {\(\mathclap{1}\)} \nendarrow;
    \draw[uno] (1,1.2)   -- ++ (0,0.9)  node[above] {\(\mathclap{ 1^*}\)} \nendarrowrev;
   \node[rectangle,draw,fill=white] at (0,1.5) {$\theta_{1}$};
   \node[rectangle,draw,fill=white] at (1,1.5) {$\theta_{1^*}$};
 \end{tikzpicture}
\end{equation}
\endgroup
and the claim follows since 
\begingroup
\tikzset{every picture/.style={yscale=0.7}}
  \begin{equation}
    \label{eq:176}
    \begin{tikzpicture}[smallnodes,anchorbase]
   \draw[uno, cross line] (0.8,1)  .. controls ++(0,0.3) and ++(0.3,0) .. ++(-0.3,0.6) \startarrowrev .. controls ++(-0.5,0) and ++(0,0.7) .. ++(-0.5,-1.6) \nendarrowrev node[below] {$\mathclap{1}$};
   \draw[uno, cross line] (0.8,1)  .. controls ++(0,-0.3) and ++(0.3,0) .. ++(-0.3,-0.6)  .. controls ++(-0.5,0) and ++(0,-0.7) .. ++(-0.5,1.6) \nendarrow node[above] {$\mathclap{1}$} ;
    \end{tikzpicture} \; = q^{\beta} \;
    \begin{tikzpicture}[smallnodes,anchorbase]
      \draw[uno] (0,0) node[below] {$\mathclap{1}$} -- ++(0,2) \midarrow node[above] {$\mathclap{1}$};
    \end{tikzpicture} \; = \;
    \begin{tikzpicture}[smallnodes,anchorbase]
      \draw[uno] (0,0) node[below] {$\mathclap{1}$} -- ++(0,2) \nendarrow node[above] {$\mathclap{1}$};
   \node[rectangle,draw,fill=white] at (0,1) {$\theta_{1}^{-1}$};
    \end{tikzpicture}\;.\qedhere
  \end{equation}
\endgroup
\end{proof}

Altogether we get:

\begin{theorem}
\label{thm:9}
  The category \(\Spider(\beta)\) is ribbon.
\end{theorem}

We collect now a technical computation:

\begin{lemma}\label{lem:19}
We have:
\begingroup
\tikzset{every picture/.style={yscale=0.7}}
\begin{align}\label{eq:177}
\begin{tikzpicture}[smallnodes,anchorbase]
  \draw[int] (0.8,1)  .. controls ++(0,-0.3) and ++(0.3,0) .. ++(-0.3,-0.6)  .. controls ++(-0.5,0) and ++(0,-0.7) .. ++(-0.5,1.6) \nendarrow node[above] {$\mathclap{a}$} ;
  \draw[int, cross line] (0.8,1)  .. controls ++(0,0.3) and ++(0.3,0) .. ++(-0.3,0.6) \startarrowrev .. controls ++(-0.5,0) and ++(0,0.7) .. ++(-0.5,-1.6) \nendarrowrev node[below] {$\mathclap{a}$};
\end{tikzpicture} \; & =\;
q^{-a\beta+a(a-1)}\;
\begin{tikzpicture}[smallnodes,anchorbase]
  \draw[int] (0,-1) node[below] {$\mathclap{a}$} -- (0,1) \nendarrow node[above] {$\mathclap{a}$} ;
\end{tikzpicture} \; = \;
\begin{tikzpicture}[smallnodes,anchorbase]
  \draw[int] (0,-1) node[below] {$\mathclap{a}$} -- (0,1) \nendarrow node[above] {$\mathclap{a}$} ;
   \node[rectangle,draw,fill=white] at (0,0) {$\theta_{a}$};
\end{tikzpicture}\;,\\
\begin{tikzpicture}[smallnodes,anchorbase]
  \draw[int] (0.8,1)  .. controls ++(0,0.3) and ++(0.3,0) .. ++(-0.3,0.6) \startarrowrev .. controls ++(-0.5,0) and ++(0,0.7) .. ++(-0.5,-1.6) \nendarrowrev node[below] {$\mathclap{a}$};
  \draw[int,cross line] (0.8,1)  .. controls ++(0,-0.3) and ++(0.3,0) .. ++(-0.3,-0.6)  .. controls ++(-0.5,0) and ++(0,-0.7) .. ++(-0.5,1.6) \nendarrow node[above] {$\mathclap{a}$} ;
\end{tikzpicture} \;& =\;
q^{a\beta-a(a-1)}\;
\begin{tikzpicture}[smallnodes,anchorbase]
  \draw[int] (0,-1) node[below] {$\mathclap{a}$} -- (0,1) \nendarrow node[above] {$\mathclap{a}$} ;
\end{tikzpicture} \; = \;
\begin{tikzpicture}[smallnodes,anchorbase]
  \draw[int] (0,-1) node[below] {$\mathclap{a}$} -- (0,1) \nendarrow node[above] {$\mathclap{a}$} ;
   \node[rectangle,draw,fill=white] at (0,0) {$\theta_{a}^{-1}$};
\end{tikzpicture}\;.
\end{align}
\endgroup
\end{lemma}

\begin{proof}
We check only the first equation, the second being similar. Using the definition of the braiding, it is easy to deduce:
\begin{equation}\label{eq:178}
\begin{tikzpicture}[smallnodes,anchorbase,yscale=0.7]
  \draw[int] (0.8,1)  .. controls ++(0,-0.3) and ++(0.3,0) .. ++(-0.3,-0.6)  .. controls ++(-0.5,0) and ++(0,-0.7) .. ++(-0.5,1.6) \nendarrow node[above] {$\mathclap{a}$} ;
  \draw[int, cross line] (0.8,1)  .. controls ++(0,0.3) and ++(0.3,0) .. ++(-0.3,0.6) \startarrowrev .. controls ++(-0.5,0) and ++(0,0.7) .. ++(-0.5,-1.6) \nendarrowrev node[below] {$\mathclap{a}$};
\end{tikzpicture} \;=\;
q^{-a}\sum_{s=0}^a(-q)^s\qbin{\beta-s}{a}\qbin{a}{s} \; \begin{tikzpicture}[smallnodes,anchorbase,yscale=0.7]
  \draw[int] (0,-1) node[below] {$\mathclap{a}$} -- (0,1) \nendarrow node[above] {$\mathclap{a}$} ;
\end{tikzpicture}\;.
\end{equation}

First, let us notice that for any \(k \in \N\) we have
\begin{multline}\label{eq:curl_keyEq}
\sum_{s=0}^a(-q)^s\qbin{a}{s}\qbin{\beta-s-k}{a} =\sum_{s=0}^a(-q)^s\qbin{\beta-s-k}{a}\left(q^{a-s}\qbin{a-1}{s-1}+q^{-s}\qbin{a-1}{s}\right) \\
= \sum_{s=0}^{a-1}(-1)^s\qbin{a-1}{s}\left(\qbin{\beta-s-k}{a}-q^a\qbin{\beta-s-k-1}{a}\right)\\
= q^{-\beta+k+a}\sum_{s=0}^{a-1}(-q)^s\qbin{a-1}{s}\qbin{\beta-s-k-1}{a-1}.
\end{multline}
Applying this formula recursively we can compute the r.h.s.\ of~\eqref{eq:178}
\begin{equation}
  \begin{aligned}[t]
 & q^{-a}\sum_{s=0}^a(-q)^s\qbin{\beta-s}{a}\qbin{a}{s}=q^{-a}q^{-\beta+a}\sum_{s=0}^{a-1}(-q)^s\qbin{a-1}{s}\qbin{\beta-s-1}{a-1} \\
   &\qquad =q^{-a}q^{-\beta+a}q^{-\beta+(a-1)+1}\cdots q^{-\beta+1+(a-1)}\sum_{s=0}^0(-q)^s\qbin{0}{s}\qbin{\beta-s-a}{0} \\
   &\qquad =q^{-a\beta+a(a-1)},
  \end{aligned}
\end{equation}
whence the claim.
\end{proof}

The category \(\Spider(d)\) allows us to recover the Reshetikhin-Turaev invariant:

\begin{prop}\label{prop:9}
  There is an obvious functor \(\calQ_\beta\colon\catTangles_q^\lab \mapto \Spider(\beta)\) (where the braiding of \(\catTangles_q^\lab\) goes to the braiding of \(\Spider(\beta)\)). For \(\beta=d\), the composition \(\calG_{m|n} \circ \calQ_d\) gives the Reshetikhin-Turaev functor \(\RT_{m|n}\).
\end{prop}

\begin{proof}
  One has to check that this functor is well-defined, i.e.\ that the relations of \(\catTangles\) are satisfied in the image of \(\calQ_\beta\). We refer to \cite{MR1881401} for the explicit relations of \(\catTangles^\lab\), which we now check. Relations \cite[(3.2) and (3.3)]{MR1881401} are clear, \cite[(3.10)]{MR1881401} is \eqref{eq:102} and \cite[(3.11)]{MR1881401} is implied by \eqref{eq:104} and \eqref{eq:105}. Relation \cite[(3.16)]{MR1881401} is implied by lemma~\ref{lem:19}.
Relations \cite[(3.13) and (3.15)]{MR1881401} follow since \(\Spider^+\) is braided, and relation \cite[(3.14)]{MR1881401} is \eqref{eq:146}. The last claim  is clear.
\end{proof}

\subsection{Equivalence of categories}
\label{sec:equiv-categ}

We constructed a full functor \(\calG \colon \Spider(d) \mapto \catRep_{m|n}\). A natural problem is to ask for additional relations on \(\Spider(d)\) such that this functor becomes an equivalence of categories. The following result tells us that it is enough to look for relations on \(\Spider^+\):

\begin{prop}
  \label{lem:16}
  Suppose that \(S\) is a set of relations on \(\Spider^+\) such that the functor \(\calG^+_{m|n}\) descends to an equivalence of categories \(\Spider^+/S \mapto \catRep^+_{m|n}\). Then the functor \(\calG_{m|n}\) restricts to an equivalence of categories \(\Spider(d)/S \mapto \catRep_{m|n}\). 
\end{prop}

\begin{proof}
  We already know that the functor \(\calG_{m|n}\) is essentially surjective and full. Hence it is enough to show that it is faithful. This follows by comparing dimensions of homomorphism spaces. Let \(\bolda , \boldb \in \Spider(d)/S\) be two objects: we want to show that \(\dim \Hom_{\Spider(d)/S}(\bolda,\boldb) = \dim \Hom_{\catRep_{m|n}}(\bigwedge^\bolda, \bigwedge^\boldb)\). We proceed by induction on the total number \(D\) of duals in the tensor product expansion of \(\bolda\) and \(\boldb\). If \(D=0\) then the claim holds by hypothesis, since \(\bolda,\boldb\) are objects of \(\Spider^+\). Hence suppose \(D>0\). Without loss of generality, suppose that \(\bolda\) contains a dual object. Since both the categories \(\Spider^+(d)/S\) and \(\catRep_{m|n}\) are braided, we can suppose without loss of generality that the dual object appears at the last place in the tensor product decomposition of \(\bolda\), that is \(\bolda = \bolda' \otimes a_r^*\). Then we have
  \begin{multline}
    \label{eq:157}
   \dim \Hom_{\Spider(d)/S} (\bolda' \otimes a_r^* ,\boldb) =    \dim \Hom_{\Spider(d)/S} (\bolda',\boldb \otimes a_r) \\= \dim \Hom_{\catRep_{m|n}}\big(\bigwedgeq^{\bolda'}, \bigwedgeq^{\boldb} \otimes \bigwedgeq^{a_r}\big) =   \dim \Hom_{\catRep_{m|n}}\big(\bigwedgeq^{\bolda'} \otimes \bigwedgeq^{{a_r}^*} , \bigwedgeq^{\boldb} \big),
  \end{multline}
  where the middle equality follows by induction, and we are done.
\end{proof}

In particular, let us define the category \(\Spider(d)_{\leq m}\) to be the quotient of the category \(\Spider(d)\) modulo objects \(a > m\) (compare with definition~\ref{def:6}). Then we have the following corollary:

\begin{corollary}
  \label{cor:2}
  The functor \(\calG_{m|0}\) descends to an equivalence of categories
  \begin{equation}
\calG_{m|0} \colon \Spider(d)_{\leq m } \xrightarrow{\;\sim\;} \catRep_{m|0}.\label{eq:198}
\end{equation}
\end{corollary}

Similarly, combining proposition~\ref{lem:16} with the main result from \cite{Grant}, we obtain:
\begin{corollary}
The category \(\catRep_{1|1}\) is equivalent to the quotient of \(\Spider(0)\) modulo the relation
\begingroup
\tikzset{every picture/.style={yscale=0.6,xscale=0.6}}
\begin{equation}
  \label{eq:158}
[a+1][a][b][b-1] \;
\begin{tikzpicture}[smallnodes,anchorbase,int]
\draw (0,0) node [below] {$\vphantom{b}a$} -- ++(0,1.9) -- ++(0,0.1) \nendarrow;
\draw (1,0) node [below] {$b$} -- ++(0,1.9) -- ++(0,0.1) \nendarrow;
\end{tikzpicture}
\;-\;
[2][a+1][b-1]\;
\begin{tikzpicture}[smallnodes,anchorbase,int]
\draw (0,0) node [below] {$\vphantom{b}a$} -- ++(0,1.9) -- ++(0,0.1) \nendarrow;
\draw (1,0) node [below] {$b$} -- ++(0,1.9) -- ++(0,0.1) \nendarrow;
\draw [uno] (1,.4) -- node[below] { $1$} ++(-1,.4) \midarrow;
\draw [uno] (0,1.2) -- node[above] { $1$} ++(1,.4) \midarrow;
\end{tikzpicture}
+[2]^2\;
\begin{tikzpicture}[smallnodes,anchorbase,int]
\draw (0,0) node [below] {$\vphantom{b}a$} -- ++(0,1.9) -- ++(0,0.1) \nendarrow;
\draw (1,0) node [below] {$b$} -- ++(0,1.9) -- ++(0,0.1) \nendarrow;
\draw  (1,.4) -- node[below] { $2$} ++(-1,.4) \midarrow;
\draw  (0,1.2) -- node[above] { $2$} ++(1,.4) \midarrow;
\end{tikzpicture}
\;=0
\end{equation}
\endgroup
for all \(a,b \in \Z_{\geq 0}\).
\end{corollary}

\begin{remark}\label{rem:8}
  In principle, it is possible to describe  a full set of relations for any \(m,n\). Namely, it follows by Schur-Weyl duality (proposition~\ref{prop:1}) that the kernel of the action of the Hecke algebra is generated by the Young symmetrizer \(Y_{\mathrm{b}}\) corresponding to the partition \(\mathrm{b}\) given by one box of size \((n+1)\times (m+1)\). A standard argument (see the proof of  \cite[Theorem~3.3.12]{miophd2}) shows that \(Y_{\mathrm{b}}\), considered as a morphism in \(\Spider^+\), generates the kernel of the functor \(\calG^+_{m|n}\). By proposition~\ref{lem:16} it follows that \(Y_{\mathrm{b}}\) generates the kernel of \(\calG_{m|n} \colon \Spider(d) \mapto \catRep_{m|n}\). Anyway, although explicit, the relation \(Y_{\mathrm{b}}\) is not handy (already in the case of \(\gl_{1|1}\) it involves four strands, while relation~\eqref{eq:158} above only involves two), and one would like to find nicer relations, involving a minimal number of strands.
\end{remark}

\subsection{More relations in \texorpdfstring{$\Spider(\beta)$}{Sp(beta)}}
\label{sec:more-relat-spid}

We compute in this subsection some further relations which are implied from the defining relations of the categories \(\Spider\) and \(\Spider(\beta)\).

\begin{lemma}
  \label{lem:9}
  In \(\Spider\) we have
  \begin{equation}
    \label{eq:127}
\;.
  \end{equation}
\end{lemma}

\begin{proof}
  The proof is similar in both cases, so we only detail the first one. We have:
  \begin{equation}
    \label{eq:128}
    %
\;.
  \end{equation}
  The first equality follows from the relation
\begingroup
\tikzset{every picture/.style={xscale=0.7}}
  \begin{equation}
    \label{eq:129}
        %
\;,
  \end{equation}
\endgroup
  which is the graphical version of a higher order Serre relation, and follows from the relations defining \(\Spider^+\).
 \end{proof}

\begingroup
\tikzset{every picture/.style={yscale=0.8}}
\begin{lemma}
  \label{lem:10}
  In \(\Spider\) we have 
  \begin{equation}
    \label{eq:130}
    %
\; = 0
  \end{equation}
\end{lemma}
as well as
  \begin{equation}
    \label{eq:135b}
    %
 \;\; = 0.
  \end{equation}
\endgroup

\begin{proof}
As the two formulas are similar, we only detail the first one. We can rewrite the first summand of \eqref{eq:130} as
\begingroup
\tikzset{every picture/.style={xscale=0.8,yscale=0.8}}
\begin{equation}
    \label{eq:131}
    %
 \;,
  \end{equation}
the second summand as
\begin{equation}  \label{eq:132}
    %
   \end{equation}
  and the third summand as
  \begin{equation}
    \label{eq:133}
   %
\;.
  \end{equation}
\endgroup
Now the claim follows by the Serre relation
\begingroup
\tikzset{every picture/.style={yscale=0.8}}
  \begin{equation}
    \label{eq:134}
    %
 \;\; = 0
  \end{equation}
as well as
\begin{equation}
    \label{eq:130b}
    %
\; = 0.
  \end{equation}
\endgroup
\end{lemma}

\begin{proof}
We again only detail the first equation. We can rewrite the second summand as
\begingroup
\tikzset{every picture/.style={yscale=0.8}}
  \begin{equation}
    \label{eq:136}
        %
\;.
  \end{equation}
  Hence, similarly as before, the claim follows by the Serre relation
  \begin{equation}
    \label{eq:137}
   %
 \; = 0
  \end{equation}
as well as
\begin{equation}
  \label{eq:192}
    %
 \; = 0.
  \end{equation}
\endgroup
\end{lemma}

\begin{proof}
  Again, let us only check the first equation. Similarly to \eqref{eq:136}, we can rewrite the second summand as
  \begin{equation}
    \label{eq:139}
        %
\;.
  \end{equation}
  Hence \eqref{eq:138} follows from the Serre relation \eqref{eq:190}.
\end{proof}

\begin{remark}\label{rem:4}
  By swapping all orientations of the uprights, we obtain (with analogous
  proofs) six more relations from
  lemmas~\ref{lem:10}, \ref{lem:11}, and \ref{lem:12}. Moreover, by replacing all
  morphisms \(\sfE_i\)
  with \(\sfF_i\)
  we obtain, again with analogous proofs, other twelve
  relations. For  space reasons, we do not write all of them.
\end{remark}

\begin{lemma}
  \label{lem:13}
  In \(\Spider(\beta)\) we have
\begingroup
\tikzset{every picture/.style={yscale=0.8}}
  \begin{equation}
    \label{eq:140}
    \begin{tikzpicture}[int,smallnodes,anchorbase]
   \draw (0,0) node[below] {$\mathclap{\vphantom{b^*}a}$} -- ++(0,2) \midarrow  node[above] {$\mathclap{a}$};
   \draw (1,0) node[below] {$\mathclap{b^*}$} -- ++(0,2) \midarrowrev node[above] {$\mathclap{b^*}$};
   \draw[uno] (0,1.5)  .. controls ++(0.5,-0.2) and ++(-0.5,-0.2)  .. ++(1,0) node[above,midway] {$1$} \midarrowrev ;
   \draw[uno] (0,0.5) .. controls ++(0.5,0.2) and ++(-0.5,0.2)  .. ++(1,0) node[below,midway] {$1$} \midarrow;
    \end{tikzpicture} \;-  \;
    \begin{tikzpicture}[int,smallnodes,anchorbase]
   \draw (0,0) node[below] {$\mathclap{\vphantom{b^*}a}$} -- ++(0,2) \midarrow  node[above] {$\mathclap{a}$};
   \draw (1,0) node[below] {$\mathclap{b^*}$} -- ++(0,2) \midarrowrev node[above] {$\mathclap{b^*}$};
   \draw[uno] (0,0.5)  .. controls ++(0.5,-0.2) and ++(-0.5,-0.2)  .. ++(1,0) node[below,midway] {$1$} \midarrowrev ;
   \draw[uno] (0,1.5) .. controls ++(0.5,0.2) and ++(-0.5,0.2)  .. ++(1,0) node[above,midway] {$1$} \midarrow;
    \end{tikzpicture} \;=  [a+b-\beta] \;
    \begin{tikzpicture}[int,smallnodes,anchorbase]
   \draw (0,0) node[below] {$\mathclap{\vphantom{b^*}a}$} -- ++(0,2) \midarrow  node[above] {$\mathclap{a}$};
   \draw (1,0) node[below] {$\mathclap{b^*}$} -- ++(0,2) \midarrowrev node[above] {$\mathclap{b^*}$};
    \end{tikzpicture}
  \end{equation}
\endgroup
and
\begingroup
\tikzset{every picture/.style={yscale=0.8}}
  \begin{equation}
    \label{eq:191}
    \begin{tikzpicture}[int,smallnodes,anchorbase]
   \draw (0,0) node[below] {$\mathclap{a^*}$} -- ++(0,2) \midarrowrev  node[above] {$\mathclap{a^*}$};
   \draw (1,0) node[below] {$\mathclap{\vphantom{a^*}b}$} -- ++(0,2) \midarrow node[above] {$\mathclap{b}$};
   \draw[uno] (0,1.5)  .. controls ++(0.5,-0.2) and ++(-0.5,-0.2)  .. ++(1,0) node[above,midway] {$1$} \midarrowrev ;
   \draw[uno] (0,0.5) .. controls ++(0.5,0.2) and ++(-0.5,0.2)  .. ++(1,0) node[below,midway] {$1$} \midarrow;
    \end{tikzpicture} \;-  \;
    \begin{tikzpicture}[int,smallnodes,anchorbase]
   \draw (0,0) node[below] {$\mathclap{a^*}$} -- ++(0,2) \midarrowrev  node[above] {$\mathclap{a^*}$};
   \draw (1,0) node[below] {$\mathclap{\vphantom{a^*}b}$} -- ++(0,2) \midarrow node[above] {$\mathclap{b}$};
   \draw[uno] (0,0.5)  .. controls ++(0.5,-0.2) and ++(-0.5,-0.2)  .. ++(1,0) node[below,midway] {$1$} \midarrowrev ;
   \draw[uno] (0,1.5) .. controls ++(0.5,0.2) and ++(-0.5,0.2)  .. ++(1,0) node[above,midway] {$1$} \midarrow;
    \end{tikzpicture} \;=  [\beta-a-b] \;
    \begin{tikzpicture}[int,smallnodes,anchorbase]
   \draw (0,0) node[below] {$\mathclap{\vphantom{b}a^*}$} -- ++(0,2) \midarrowrev  node[above] {$\mathclap{a^*}$};
   \draw (1,0) node[below] {$\vphantom{a^*}\mathclap{b}$} -- ++(0,2) \midarrow node[above] {$\mathclap{b}$};
    \end{tikzpicture}\;.
  \end{equation}
\endgroup
\end{lemma}

\begin{proof}
  We check only the first relation, the second one being similar. We have
\begingroup
\tikzset{every picture/.style={yscale=0.8,xscale=0.8,baseline={([yshift=-0.5ex]0,0.8)}}}    
  \begin{equation}
    \label{eq:141}
    \begin{aligned}[t]
      & \begin{tikzpicture}[smallnodes]
   \draw[int] (0,0) node[below] {$\mathclap{\vphantom{b^*}a}$} -- ++(0,2) \midarrow ;
   \draw[int] (1,0) node[below] {$\mathclap{b^*}$} -- ++(0,2) \midarrowrev ;
   \drawmiddlee{0,0.5};
   \drawmiddlef{0,1.5};
    \end{tikzpicture} \;=  \;
    \begin{tikzpicture}[smallnodes]
   \draw[int] (0,0) node[below] {$\mathclap{\vphantom{b^*}a}$} -- ++(0,2) \midarrow ;
   \draw[int] (1,0) node[below] {$\mathclap{b^*}$} -- ++(0,2) \midarrowrev;
     \draw[uno] (0,1.5) ++(0,-0.1) -- ++(0.3,+0.1) \midarrow arc (180:0:0.2) \midarrow -- ++(0.3,-0.1)  \midarrow;
     \draw[uno] (0,0.5) ++(0,+0.1) -- ++(0.3,-0.1)  \midarrowrev arc (180:360:0.2) \midarrowrev -- ++(0.3,+0.1) \midarrowrev;
     \draw [dashed] (0.3,0.5) -- ++(0,+1);
     \draw [dashed] (0.7,0.5) -- ++(0,+1);
    \end{tikzpicture} \\  & =  \;
\begin{tikzpicture}[smallnodes]
   \draw[int] (0,0) node[below] {$\mathclap{\vphantom{b^*}a}$} -- ++(0,2) \midarrow ;
   \draw[int] (1,0) node[below] {$\mathclap{b^*}$} -- ++(0,2) \midarrowrev ;
     \draw (0,1.5) ++(0,-0.1) ++(0.3,-0.1) ++(0,0.2) arc (180:0:0.2) \midarrow -- ++(0,-1) \midarrow arc(360:180:0.2) \midarrow -- ++(0,1) \midarrow;
     \draw (0,1.5) ++(0,-0.1) -- ++(0.3,-0.1) \midarrowrev;
     \draw (0,0.5) ++(0,0.1) -- ++(0.3,0.1) \midarrow;
     \draw (1,1.5) ++(0,-0.1) -- ++(-0.3,-0.1) \midarrow;
     \draw (1,0.5) ++(0,0.1) -- ++(-0.3,0.1) \midarrowrev;
     \draw[int] (0.3,1.3) -- ++ (0,-0.6);
     \draw[int] (0.7,1.3) -- ++ (0,-0.6);
    \end{tikzpicture} \; - [a-1]  \;
    \begin{tikzpicture}[smallnodes]
   \draw[int] (0,0) node[below] {$\mathclap{\vphantom{b^*}a}$} -- ++(0,2) \midarrow;
   \draw[int] (1,0) node[below] {$\mathclap{b^*}$} -- ++(0,2) \midarrowrev;
     \draw (0,1.5) ++(0,-0.1) ++(0.3,-0.1) ++(0,0.2) arc (180:0:0.2) \midarrow -- ++(0,-1) \midarrow arc(360:180:0.2) \midarrow -- ++(0,1) \midarrow;
     \draw (1,1.5) ++(0,-0.1) -- ++(-0.3,-0.1) \midarrow;
     \draw (1,0.5) ++(0,0.1) -- ++(-0.3,0.1) \midarrowrev;
     \draw[int] (0.7,1.3) -- ++ (0,-0.6);
    \end{tikzpicture} \; - [b-1] \;
    \begin{tikzpicture}[smallnodes]
   \draw[int] (0,0) node[below] {$\mathclap{\vphantom{b^*}a}$} -- ++(0,2) \midarrow;
   \draw[int] (1,0) node[below] {$\mathclap{b^*}$} -- ++(0,2) \midarrowrev ;
     \draw (0,1.5) ++(0,-0.1) ++(0.3,-0.1) ++(0,0.2) arc (180:0:0.2) \midarrow -- ++(0,-1) \midarrow arc(360:180:0.2) \midarrow -- ++(0,1) \midarrow;
     \draw (0,1.5) ++(0,-0.1) -- ++(0.3,-0.1) \midarrowrev;
     \draw (0,0.5) ++(0,0.1) -- ++(0.3,0.1) \midarrow;
     \draw[int] (0.3,1.3) -- ++ (0,-0.6);
    \end{tikzpicture} \; + [a-1][b-1]  \;
    \begin{tikzpicture}[smallnodes]
   \draw[int] (0,0) node[below] {$\mathclap{\vphantom{b^*}a}$} -- ++(0,2) \midarrow;
   \draw[int] (1,0) node[below] {$\mathclap{b^*}$} -- ++(0,2) \midarrowrev;
     \draw (0,1.5) ++(0,-0.1) ++(0.3,-0.1) ++(0,0.2) arc (180:0:0.2) \midarrow -- ++(0,-1) \midarrow arc(360:180:0.2) \midarrow -- ++(0,1) \midarrow;
    \end{tikzpicture} 
 \\
  &  =  \;
    \begin{tikzpicture}[smallnodes]
   \draw[int] (0,0) node[below] {$\mathclap{\vphantom{b^*}a}$} -- ++(0,2) \midarrow;
   \draw[int] (1,0) node[below] {$\mathclap{b^*}$} -- ++(0,2) \midarrowrev;
     \draw (0,1.5) ++(0,-0.1) ++(0.3,-0.1) .. controls ++(0.2,0.1) and ++(-0.2,0.1) .. ++(0.4,0) \midarrow -- ++(0,-0.6) \midarrow .. controls ++(-0.2,-0.1) and ++(+0.2,-0.1) .. ++(-0.4,0) \midarrow -- ++(0,0.6) \midarrow;
     \draw (0,1.5) ++(0,+0.1) -- ++(0.3,-0.1)  \midarrowrev -- ++(0,-1) -- ++(-0.3,-0.1) \midarrowrev;
     \draw (1,1.5) ++(0,+0.1) -- ++(-0.3,-0.1) \midarrow -- ++(0,-1) -- ++(0.3,-0.1) \midarrow;
     \draw[int] (0.3,1.3) -- ++ (0,-0.6);
     \draw[int] (0.7,1.3) -- ++ (0,-0.6);
    \end{tikzpicture} \; - [a-1]  \;
    \begin{tikzpicture}[smallnodes]
   \draw[int] (0,0) node[below] {$\mathclap{\vphantom{b^*}a}$} -- ++(0,2) \midarrow;
   \draw[int] (1,0) node[below] {$\mathclap{b^*}$} -- ++(0,2) \midarrowrev;
     \draw (0,1.5) ++(0,-0.1) ++(0.3,-0.1) .. controls ++(0.2,0.1) and ++(-0.2,0.1) .. ++(0.4,0) \midarrow -- ++(0,-0.6) \midarrow .. controls ++(-0.2,-0.1) and ++(+0.2,-0.1) .. ++(-0.4,0) \midarrow -- ++(0,0.6) \midarrow;
     \draw (1,1.5) ++(0,+0.1) -- ++(-0.3,-0.1) \midarrow -- ++(0,-1) -- ++(0.3,-0.1) \midarrow;
     \draw[int] (0.7,1.3) -- ++ (0,-0.6);
    \end{tikzpicture} \; - [b-1] \;
    \begin{tikzpicture}[smallnodes]
   \draw[int] (0,0) node[below] {$\mathclap{\vphantom{b^*}a}$} -- ++(0,2) \midarrow ;
   \draw[int] (1,0) node[below] {$\mathclap{b^*}$} -- ++(0,2) \midarrowrev;
     \draw (0,1.5) ++(0,-0.1) ++(0.3,-0.1) .. controls ++(0.2,0.1) and ++(-0.2,0.1) .. ++(0.4,0) \midarrow -- ++(0,-0.6) \midarrow .. controls ++(-0.2,-0.1) and ++(+0.2,-0.1) .. ++(-0.4,0) \midarrow -- ++(0,0.6) \midarrow;
     \draw (0,1.5) ++(0,+0.1) -- ++(0.3,-0.1)  \midarrowrev -- ++(0,-1) -- ++(-0.3,-0.1) \midarrowrev;
     \draw[int] (0.3,1.3) -- ++ (0,-0.6);
    \end{tikzpicture}  \; + [a-1][b-1][\beta]
    \begin{tikzpicture}[smallnodes]
   \draw[int] (0,0) node[below] {$\mathclap{\vphantom{b^*}a}$} -- ++(0,2) \midarrow;
   \draw[int] (1,0) node[below] {$\mathclap{b^*}$} -- ++(0,2) \midarrowrev;
   \end{tikzpicture} \\
   & \begin{multlined}  = \;
    \begin{tikzpicture}[smallnodes]
   \draw[int] (0,0) node[below] {$\mathclap{\vphantom{b^*}a}$} -- ++(0,2) \midarrow;
   \draw[int] (1,0) node[below] {$\mathclap{b^*}$} -- ++(0,2) \midarrowrev;
     \draw (0,1.5) ++(0,+0.1) -- ++(0.3,-0.1)  \midarrowrev -- ++(0,-0.2) .. controls ++(0.2,-0.1) and ++(-0.2,-0.1) .. ++(0.4,0) \midarrowrev -- ++(0,0.2) -- ++(0.3,+0.1) \midarrowrev; 
     \draw (0,0.5) ++(0,-0.1) -- ++(0.3,+0.1)  \midarrow -- ++(0,+0.2) .. controls ++(0.2,+0.1) and ++(-0.2,+0.1) .. ++(0.4,0) \midarrow -- ++(0,-0.2) -- ++(0.3,-0.1) \midarrow; 
     \draw[dashed] (0.3,1.3) -- ++ (0,-0.6);
     \draw[dashed] (0.7,1.3) -- ++ (0,-0.6);
    \end{tikzpicture} \; - [2-\beta] \;
    \begin{tikzpicture}[smallnodes]
   \draw[int] (0,0) node[below] {$\mathclap{\vphantom{b^*}a}$} -- ++(0,2) \midarrow ;
   \draw[int] (1,0) node[below] {$\mathclap{b^*}$} -- ++(0,2) \midarrowrev ;
     \draw (0,1.5) ++(0,+0.1) -- ++(0.3,-0.1)  \midarrowrev -- ++(0,-1) -- ++(-0.3,-0.1) \midarrowrev;
     \draw (1,1.5) ++(0,+0.1) -- ++(-0.3,-0.1) \midarrow -- ++(0,-1) -- ++(0.3,-0.1) \midarrow;
    \end{tikzpicture}\; + [a-1][1-\beta]  \;
    \begin{tikzpicture}[smallnodes]
   \draw[int] (0,0) node[below] {$\mathclap{\vphantom{b^*}a}$} -- ++(0,2) \midarrow ;
   \draw[int] (1,0) node[below] {$\mathclap{b^*}$} -- ++(0,2) \midarrowrev ;
     \draw (1,1.5) ++(0,+0.1) -- ++(-0.3,-0.1) \midarrow -- ++(0,-1) -- ++(0.3,-0.1) \midarrow;
    \end{tikzpicture} \\ + [b-1][1-\beta] \;
    \begin{tikzpicture}[smallnodes]
   \draw[int] (0,0) node[below] {$\mathclap{\vphantom{b^*}a}$} -- ++(0,2) \midarrow ;
   \draw[int] (1,0) node[below] {$\mathclap{b^*}$} -- ++(0,2) \midarrowrev ;
     \draw (0,1.5) ++(0,+0.1) -- ++(0.3,-0.1)  \midarrowrev -- ++(0,-1) -- ++(-0.3,-0.1) \midarrowrev;
    \end{tikzpicture}  \; + [a-1][b-1][\beta]
    \begin{tikzpicture}[smallnodes]
   \draw[int] (0,0) node[below] {$\mathclap{\vphantom{b^*}a}$} -- ++(0,2) \midarrow ;
   \draw[int] (1,0) node[below] {$\mathclap{b^*}$} -- ++(0,2) \midarrowrev ;
   \end{tikzpicture} \;,
   \end{multlined}
  \end{aligned} 
  \end{equation}
\endgroup
  where we used relations \eqref{eq:119}, \eqref{eq:121} and \eqref{eq:122}.
The claim now follows since
\begin{align}\label{eq:142}
& [2-\beta][a][b]-[a-1][1-\beta][b]-[b-1][1-\beta][a]-[a-1][b-1][\beta] \\
& \qquad = [b] \big( [2-\beta][a]- [1-\beta][a-1] \big)
     +[b-1] \big( [\beta-1][a]-[\beta][a-1] \big)\notag\\
& \qquad   = [b] [1-\beta+a] + [b-1][\beta-a] = [b][-\beta+a+1]-[b-1][-\beta+a] \notag\\
& \qquad       = [a+b-\beta]. \notag\qedhere
  \end{align}
\end{proof}

\section{Mixed skew Howe duality and the idempotented quantum group}
\label{sec:dbleSchur}

We introduce now quotients of the idempotented version of \(U_q(\gl_k)\), depending on the choice of a sequence of signs, which are the natural targets for \(\gl_{m|n}\) link invariants. We constructed a first example of this in \cite{QS}, here we generalize it further.

\subsection{Doubled Schur algebras}

Let us  fix a sequence \(\boldeta=(\eta_1,\dotsc,\eta_k) \in \{\pm 1\}^k \) of signs.
Let \(\lattice_\boldeta\) be the set of sequences \(\lambda=(\lambda_1,\dotsc,\lambda_{k+l})\) with \(\lambda_i \in \Z_{\geq 0}\) if \(\eta_i=1\) and \(\lambda_i \in \beta-\Z_{\geq 0}\) if \(\eta_i =-1\). We let also \(\upalpha_i = (0,\dotsc,0,1,-1,0,\dotsc,0)\), the entry \(1\) being at position \(i\).

\begin{definition} \label{defn_UzklB}
We define \(\U_q(\gl_{\boldeta})_\beta\) to be the additive \(\C(q,q^\beta)\)--linear category with:
\begin{itemize}
\item {\bf objects:} formal direct sums of \(\onell{\lambda}\) for \(\lambda \in \lattice_\boldeta\);
\item {\bf morphisms:} generated by identity endomorphisms \(\onel\) in \(\Hom(\onell{\lambda},\onell{\lambda})\), and morphisms \(E_i\onel=\onell{\lambda+\upalpha_i}E_i \in \Hom(\onel,\onell{\lambda+\upalpha_i})\), \(F_i\onel=\onell{\lambda-\upalpha_i}F_i  \in \Hom(\onel, \onell{\lambda-\upalpha_i})\).
We will often abbreviate by omitting some of the symbols \(\onel\). The morphisms are subject to the following relations:
\begin{enumerate}[label=(DS\arabic*),itemsep=5pt,topsep=5pt]
\item \label{item:4} \([E_i,F_j]\onel=\delta_{i,j}[\lambda_{i}-\lambda_{i+1}]\onel\),
\item \label{item:6} \(\begin{aligned}[t]
E_i^2E_j\onel-(q+q^{-1})E_iE_jE_i\onel+E_jE_i^2\onel& =0  &\text{if } j=i\pm 1,  \\
F_i^2F_j\onel-(q+q^{-1})F_iF_jF_i\onel+F_jF_i^2\onel&=0 &\text{if } j=i\pm 1,
\end{aligned}\)
\item \label{item:2} \(E_iE_j\onel=E_jE_i\onel \text{ and } F_iF_j\onel = F_jF_i\onel\) if \(\abs{i-j}>1\). 
\end{enumerate}
\end{itemize}
\end{definition}

We indifferently use: \(\onell{\lambda+\upalpha_i}E_i\onel=E_i\onel=\onell{\lambda+\upalpha_i}E_i\), since knowing the source or the target of the \(1\)-morphism is enough to determine the other one. We use the following convention: if the symbol \(\onell{\lambda}\) appears and \(\lambda \notin \lattice_\boldeta\), then \(\onell{\lambda}=0\).

For the sequence of signs \(\boldeta_{r,s}\) (see \eqref{eq:166}), the category \(\U_q(\gl_{r+s})_{\boldeta_{r,s}}\) was already introduced in \cite{QS}. In the special case \(s=0\) the category \(\U_q(\gl_r)_{\boldeta_{r,0}}\) was already defined in \cite{CKM}, and  denoted  by \(\U_q^{\geq 0}(\gl_r)\).

Let us define a functor \(\Phi_{\boldeta}\colon \U_q(\gl_{\boldeta})_\beta \mapto \Spider(\beta)\) by setting on the generators
\begin{align}
  \onell{\lambda}  & \longmapsto \bolda, \label{eq:184}\\
  E_i\onell{\lambda} & \longmapsto \big(\id^{\otimes(i-1)} \otimes \sfE \otimes \id^{\otimes(k-i-1)}\big)\idem_{\bolda},\label{eq:185}  \\
  F_i\onell{\lambda} & \longmapsto \big(\id^{\otimes(i-1)} \otimes \sfF \otimes \id^{\otimes(k-i-1)}\big)\idem_{\bolda},\label{eq:186}
\end{align}
where \(\bolda\) is determined by
\begin{equation}
  \label{eq:183}
 a_i =  \begin{cases}
 \lambda_i & \text{if } \eta_i = +1,\\
 (-\lambda_i+\beta)^* & \text{if } \eta_i =-1,
   \end{cases}
\end{equation}
Notice that depending on the index \(i\), the morphisms \(\sfE\) and \(\sfF\) in the equations above can be either \eqref{eq:71}, \eqref{eq:107}, \eqref{eq:116} or \eqref{eq:118}.

\begin{prop}
  \label{prop:14}
  The functor \(\Phi_\boldeta\) is  well-defined.
\end{prop}

\begin{proof}
  We shall check that the relations defining \(\U_q(\gl_{\boldeta})_\beta\) are satisfied by the images of the \(E_i\)'s and \(F_i\)'s.  The commuting relation~\ref{item:2} is straightforward. Also \ref{item:4} is straightforward in case \(\abs{i-j}>1\). Since the other relations are local, we only need to consider the case in which two neighbor indices \(i,i+1\) are involved such that \(\eta_i \neq \eta_{i+1}\) (otherwise the relations hold because they hold in \(\Spider^+\) by definition).

Relation \ref{item:4} for \(i=j \pm 1\) is very easy to check, and follows from relation~\eqref{eq:74}. In the case \(i=j\), the relation \ref{item:4} is proved
 in lemma~\ref{lem:13}. 
The Serre relations~\ref{item:6}
are proven in lemmas~\ref{lem:10}, \ref{lem:11}
and \ref{lem:12} (cf.\ also remark~\ref{rem:4}).
\end{proof}

We denote by \(\Spider(\beta)_\boldeta\) the additive full subcategory of \(\Spider(\beta)\) with objects \((a_1,\dotsc,a_k)\) such that \(a_i \in \Z_{\geq 0}\) if \(\eta_i=1\) and \(a_i=b_i^*\) for some \(b_i \in \Z_{\geq 0}\) if \(\eta_i=-1\).
Then \(\Phi_\boldeta\) has values in \(\Spider(\beta)_\boldeta\) and, indeed, we will prove that this is an equivalence of categories. But first, we want to investigate what happens when we swap two signs of the sequence \(\boldeta\).

 Let also \(1 \leq i \leq k\) be an index with \(\eta_i = 1\), \(\eta_{i+1}=-1\) and let \(\sfs\) be the simple transposition \((i,i+1)\). Denote by \(\boldeta' = \sfs \boldeta\) the swapped sequence. Using the braiding in \(\Spider(\beta)\) it is immediate to show that \(\Spider(\beta)_\boldeta\) and \(\Spider(\beta)_{\boldeta'}\) are equivalent. Pulling back this braiding we can construct an equivalence \(\U_q(\gl_{\boldeta})_\beta \mapto \U_q(\gl_{\boldeta'})_\beta\), which is explicitly given by some analogue of Lusztig's symmetries: 

\begin{lemma}\label{lem:21}
We have an equivalence of categories \(\sfT_i \colon \U_q(\gl_{\boldeta})_\beta \mapto \U_q(\gl_{\boldeta'})_\beta\) given by the map
\begin{equation}\label{eq:117}
\begin{aligned}
\onel & \mapsto \onell{\sfs \lambda}, & & \\
E_{i-1} \onel &\mapsto (q E_{i-1}E_i -E_iE_{i-1})\onell{\sfs \lambda}, & F_{i-1}\onel &\mapsto (q^{-1}F_iF_{i-1}-F_{i-1}F_i)\onell{\sfs \lambda}, \\
E_{i+1}\onel &\mapsto (q E_{i+1}E_i-E_iE_{i+1})\onell{\sfs \lambda}, & F_{i+1}\onel &\mapsto (q^{-1}F_iF_{i+1}-F_{i+1}F_i)\onell{\sfs \lambda}, \\
E_i\onel &\mapsto q^{\lambda_i-\lambda_{i+1}}F_i\onell{\sfs \lambda}, & F_i \onel &\mapsto q^{-\lambda_i+\lambda_{i+1}+2}E_i\onell{\sfs \lambda}.
\end{aligned}
\end{equation}
Moreover, we have a commutative diagram
\begin{equation}\label{eq:174}
  \begin{tikzpicture}[baseline=(current bounding box.center)]
  \matrix (m) [matrix of math nodes, row sep=3em, column
  sep=3em, text height=1.5ex, text depth=0.25ex] {
     \U_q(\gl_{\boldeta})_\beta  & \Spider(\beta)_\boldeta  \\
     \U_q(\gl_{\boldeta'})_\beta  & \Spider(\beta)_{\boldeta'}  \\};
  \path[->] (m-1-1) edge node[above] {$\Phi_\boldeta$} (m-1-2);
  \path[->] (m-2-1) edge node[above] {$\Phi_{\boldeta'}$} (m-2-2);
  \path[->] (m-1-1) edge node[left] {$\sfT_i$} (m-2-1);
  \path[->] (m-1-2) edge node[right] {$\sfT_i$} (m-2-2);
\end{tikzpicture}
\end{equation}
where the right vertical arrow is defined on objects by applying the transposition \(\sfs\) and on morphisms by \(\sfT_i(x) = T_i x T_{i}^{-1}\), where
\begin{equation}
T_i = \id_{(a_1,\dotsc,a_{i-1})} \otimes c_{a_i,a_{i+1}} \otimes \id_{(a_{i+2},\dotsc,a_{k})}.\label{eq:173}
\end{equation}
\end{lemma}

\begin{proof} 
The proof is a lengthy but straightforward calculation. One shows explicitly that \eqref{eq:117} give a well-defined map and that the diagram \eqref{eq:174} commutes. Formulas analogous to \eqref{eq:117} define explicitly the inverse of \(\sfT_i\). 
\end{proof}

As a consequence, we get:
\begin{corollary}\label{cor:3}
  If \(\boldeta,\boldeta'\) are two sequences of signs with the same number of \(+1\) and \(-1\) then \(\U_q(\gl_{\boldeta})_\beta\) and \(\U_q(\gl_{\boldeta'})_\beta\) are isomorphic.
\end{corollary}

We can now prove:

\begin{prop}\label{prop:13}
The functor \(\Phi_\boldeta\colon \U_q(\gl_{\boldeta})_\beta \mapto \Spider(\beta)\) is full.
\end{prop}

\begin{proof}
By lemma~\ref{lem:21}, and in particular by \eqref{eq:174}, we can restrict to the case \(k=r+s\) and \(\boldeta=\boldeta_{r,s}\).

  Let \(\lambda=(1,\dotsc,1,\beta-1,\dotsc,\beta-1)\) and \(\bolda=\Phi_\boldeta(\onell{\lambda}) = 1^{\otimes r} \otimes (1^*)^{\otimes s}\),
and let us prove that the image of \(\Phi_\boldeta\) contains \(\End_{\Spider(\beta)}(\bolda)\). Indeed, \(\End_{\Spider(\beta)}(\bolda)\) is just  the walled Brauer algebra \(\Br_{\boldeta}(\beta)\), which is generated by upward or downward pointing crossings and cup-caps (cf.\ for example \cite{MR2955190}). 
Hence \(\End_{\Spider(\beta)}(\bolda)\) is generated by 
\begin{align}
\id^{\otimes i} \otimes c_{1,1} \otimes \id^{\otimes (r+s-i-2)}  & \qquad \text{for } i=0,\dotsc,r-2,\\
\id^{\otimes i} \otimes c_{1^*,1^*} \otimes \id^{\otimes (r+s-i-2)} &\qquad \text{for } i=r+1,\dotsc,r+s-2,\\
\id^{\otimes (r-1)}\otimes (\sfE \sfF \idem_{1 \otimes 1^*})\otimes \id^{\otimes (s-1)}.
\end{align}
 These elements are clearly in the image of \(\U_q(\gl_{\boldeta_{r,s}})_\beta\). Similarly, we can show that \(\Phi_{\boldeta_{r,s}} \colon \Hom_{\U_q(\gl_{\boldeta_{r,s}})_\beta}(\onell{\lambda},\onell{\mu}) \mapto \Hom_{\Spider(\beta)}(\Phi_{\boldeta_{r,s}}(\onell{\lambda}),\Phi_{\boldeta_{r,s}}(\onell{\mu}))\) is surjective in the case \(\onell{\lambda} = \bolda' \otimes \bolda''\) and \(\onell{\mu} = \boldb' \otimes \boldb''\), where \(\bolda',\boldb'\) are tensor products of \(1\)'s and \(0\)'s and \(\bolda'',\boldb''\) are tensor products of \(1^*\)'s and \(0\)'s. 

  Let us now consider the  case of two general elements \(\onell{\lambda},\onell{\mu} \in \lattice_{\boldeta_{r,s}}\), so that  \(\Phi_{\boldeta_{r,s}}(\onell{\lambda})=\bolda=\bolda' \otimes \bolda''\) and \(\Phi_{\boldeta_{r,s}}(\onell{\mu})=\boldb=\boldb' \otimes \boldb''\) with \(\bolda' ,\boldb' \in \Spider^+\) and \(\bolda'',\boldb''\) duals of objects from \(\Spider^+\), and let us show that  \(\Hom_{\Spider(\beta)}(\bolda,\boldb)\) is in the image of \(\Phi_{\boldeta}\). The strategy of the proof is the following: first, we prove that this is true up to enlarging \(r\) and \(s\) (pictorially, we allow more horizontal space for more strands) and then we show that this enlarging was not necessary.

Let also \(r'\geq r, s' \geq s\) and let \(j\colon \U_q(\gl_{\boldeta_{r,s}})_\beta \into \U_q(\gl_{\boldeta_{r',s'}})_\beta\) be the inclusion which maps \(X_i \mapsto X_{i+r'-r}\) for \(X_i=E_i,F_i\). Notice that the following diagram commutes by construction
  \begin{equation}
    \label{eq:167}
      \begin{tikzpicture}[baseline=(current bounding box.center)]
  \matrix (m) [matrix of math nodes, row sep=3em, column
  sep=4.5em, text height=1.5ex, text depth=0.25ex] {
    \U_q(\gl_{\boldeta_{r+s}})_\beta &  \Spider(\beta)\\
    \U_q(\gl_{\boldeta_{r'+s'}})_\beta &  \Spider(\beta)\\};
  \path[right hook->] (m-1-1) edge node[left] {$j$} (m-2-1);
  \path[right hook->] (m-1-2) edge node[left] {} (m-2-2);
  \path[->] (m-1-1) edge node[above] {$\Phi_{\boldeta_{r,s}}$} (m-1-2);
  \path[->] (m-2-1) edge node[above] {$\Phi_{\boldeta_{r',s'}}$} (m-2-2);
    \end{tikzpicture}
  \end{equation}
where the vertical map on the right is given by tensoring with \(0^{\otimes (r'-r)}\) and \(0^{\otimes(s'-s)}\) on the left and on the right, respectively.
We claim that there are some \(r'\) and \(s'\) such that \(\Hom_{\Spider(d)}(\bolda,\boldb)\) is in the image of \(\Phi_{\eta_{r',s'}}\). 
Indeed, we can write any morphism \(\varphi \colon \bolda \mapto \boldb\) as \((\pi_{\boldb'} \otimes \pi_{\boldb''}) \circ \varphi' \circ (\iota_{\bolda'} \otimes \iota_{\bolda''})\) for \(\varphi' \in \Hom(1^{\otimes \abs{\bolda'}} \otimes 1^{* \otimes \abs{\bolda''}}, 1^{\otimes \abs{\boldb'}} \otimes 1^{* \otimes \abs{\boldb''}})\). Choose \(r' = \max\{\abs{\bolda'},\abs{\boldb'}\}\) and \(s' = \max\{\abs{\bolda''},\abs{\boldb''}\}\). Then, by the first paragraph of the proof, \(\varphi'\)  is in the image of \(\Phi_{\boldeta_{r',s'}}\). Since \(\pi_{\boldb'} \otimes \pi_{\boldb''}\) and \(\iota_{\bolda'} \otimes \iota_{\bolda''}\) are also in the image of \(\Phi_{\boldeta_{r',s'}}\), the morphism \(\varphi\) is, too.

Choose now \(r'\geq r\) and \(s'\geq s\) minimal such that \(\Hom_{\Spider(\beta)}(\bolda,\boldb)\) is in the image of \(\Phi_{\eta_{r',s'}}\). We want to show by contradiction that \(r'=r\) and \(s'=s\). Suppose, on the contrary, that \(r'>r\) (the case \(s'>s\) being analogous). Choose \(\varphi \in \Hom_{\Spider(\beta)}(\bolda,\boldb)\) and pick \(x \in \U_q(\gl_{\boldeta_{r',s'}})_\beta\) with \(\varphi=\Phi_{\boldeta_{r',s'}}(x)\). 
It follows from the PBW Theorem that we can write \(x\) as linear combination \(x = \sum_h \gamma_h x_h\) of monomials \(x_h\) in the generators \(E_i,F_i\) such that in each \(x_h\) the generators \(F_i\) appear to the right of the \(E_i\)'s. For each \(x_h\) we have now three possibilities:
\begin{enumerate}[(i)]
\item The generator \(F_1\) appears at least once in \(x_h\), and we can assume that there are no \(E_1\)'s on the right of this \(F_1\). Then \(\Phi_{\boldeta_{r',s'}}(x_h)=0\), since this \(F_1\) is sent to some \((\sfF \otimes \id^{r'+s'-2})\idem_{\boldc}\), where \(c_1=0\). Hence we can remove the term \(\gamma_h x_h\) from \(x\) without changing the value of \(\Phi_{\boldeta_{r',s'}}(x)\).
\item The generator \(E_1\) appears at least once in \(x_h\), and there are no \(F_1\)'s. Similarly as before, by looking at the first entry of the sequences (which are zero both at the bottom and at the top, since \(r'>r\)) we conclude that this monomial \(x_h\) acts by zero. Hence we can remove \(\gamma_hx_h\) from \(x\). 
\item The generators \(E_1\) and \(F_1\) do not appear in \(x_h\).
\end{enumerate}
So we proved that we do not need \(E_1\) and \(F_1\), so \(r'\) was not minimal, and this is a contradiction.
\end{proof}

\subsection{Mixed skew Howe duality}

Let us now set \(\beta=d=m-n\) and  fix  an integer \(N\). 
Let 
\begin{equation}
  \label{eq:161}
  \Xi_{\boldeta,N} = \{a_1^{\eta_1}\otimes \dotsb \otimes a_k^{\eta_k} \suchthat a_i \in \Z_{\geq 0},\, \eta_1 a_1+\dotsb+\eta_k a_k - d p_\boldeta  = N\}
\end{equation}
where \(p_\boldeta = \# \{i \suchthat \eta_i = -1\}\), and we use the convention \(a^{-1} = a^*\) for \(a \in \Z_{\geq 0}\). 

Consider the space
\begin{equation}
  \label{eq:165}
  W = \bigoplus_{\bolda \in \Xi_{\boldeta,N}} \bigwedgeq^{\bolda} \C_q^{m|n}.
\end{equation}
Composing the map \(\Phi_\boldeta\) and the functor \(\calG_{m|n}\) we have then:

\begin{prop}
\label{prop:15}
  The assignment   \( x \mapsto \calG_{m|n} \circ \Phi_\boldeta (x)\) defines a map 
  \begin{equation}
\Psi_{\boldeta,N} \colon \U_q(\gl_{\boldeta})_\beta \mapto \End_{U_g(\gl_{m|n})}(W).\label{eq:201}
\end{equation}
\end{prop}

As a direct consequence of proposition~\ref{prop:13} we obtain:

\begin{theorem}[Mixed skew Howe duality]
  \label{thm:7}
  For all \(m,n\geq 0\), \(N \in \Z\) and for all sequences of signs \(\boldeta\) the map \(\Psi_{\boldeta,N}\) is surjective. In particular, \(\U_q(\gl_\boldeta)_d\) acts \(U_q(\gl_{m|n})\)--equivariantly on the \(U_q(\gl_{m|n})\)-module \(W\)
   and generates the full centralizer.
\end{theorem}

We expect a double centralizing property to hold, i.e. we expect the action of \(U_q(\gl_{m|n})\) to generate the full centralizer of the \(\U_q(\gl_{\boldeta_{r,s}})_d\)--action. 

\begin{remark}\label{rem:5}
  By the theorem, the space \(W\)
  inherits an action of \(\U_q(\gl_{\boldeta})_\beta\).
  Since the latter is an idempotented version of
  \(U_q(\gl_k)\),
  this induces an action of \(U_q(\gl_k)\)
  on \(W\). In particular, we have commuting actions
  \begin{equation}
    \label{eq:187}
    U_q(\gl_k) \acts \bigoplus_{\bolda \in \Xi_{\boldeta,N}} \bigwedgeq^{\bolda} \C_q^{m|n} \actsb U_q(\gl_{m|n}).
  \end{equation}
  Anyway, since there are infinitely many \(\gl_k\)
  weights involved, the projections onto the
  \(\gl_k\)--weight
  spaces are not in the image of the map
  \(U_q(\gl_k) \mapto \End_{U_q(\gl_{m|n})}(W)\),
  which is therefore not surjective. But the image of \(U_q(\gl_k)\) together with the projections onto the weight spaces generate all intertwiners.
\end{remark}

\begin{remark}\label{rem:3}
  Consider the special case \(\boldeta=\boldeta_{r,s}\). Then we have commuting actions of \(U_q(\gl_{r+s})\) and \(U_q(\gl_{m|n})\) on
\begin{equation}
     \label{eq:168}
 \bigoplus_{\mathclap{\qquad\substack{a_1+\dotsb+a_r-a_{r+1}\\-\dotsb-a_{r+s}-(m-n)s = N}}} \qquad \bigwedgeq^{a_1} \C_q^{m|n} \otimes \dotsb \otimes \bigwedgeq^{a_r}\C_q^{m|n} \otimes \bigwedgeq^{a_{r+1}} \big(\C_q^{m|n}\big)^* \otimes \dotsb \otimes \bigwedgeq^{a_{r+s}} \big(\C_q^{m|n}\big)^*.
\end{equation}
Notice that for \(s=0\) this gives back the usual skew Howe duality.  As a \(U_q(\gl_{r+s})\)--module, \eqref{eq:168} is a weight module and is \(U_q(\gl_{r})\otimes U_q(\gl_s)\)--finite. Moreover, the action of the lower triangular part \(U_q(\mathfrak{n}^-)\) is locally finite. In general, however, it is not finitely generated. This is true in the case \(n=1\) (since the dimension of the exterior powers of \(\C^{m|1}\) is bounded) and then we can conclude that \eqref{eq:168} is in the parabolic category \(\catO^{\mathfrak{p}}(U_q(\gl_{r+s}))\) (however with respect to the negative Borel) corresponding to the parabolic subalgebra \(\mathfrak{p} \subseteq \gl_{r+s}\) with Levi subalgebra \(\gl_r \oplus \gl_s\).
\end{remark}

\begin{example}\label{ex:1}
  Let us illustrate the example \(m=n=1\), \(r=s=1\) and \(N=0\). Then we have two commuting actions
  \begin{equation}
    \label{eq:169}
  U_q(\gl_{1|1}) \acts  \bigoplus_{a=0}^{\infty} \bigwedgeq^a \C^{1|1}_q \otimes \bigwedgeq^a \big(\C^{1|1}_q\big)^* \actsb U_q(\gl_2).
  \end{equation}
  Notice that \(\bigwedgeq^0 \C^{1|1}_q \otimes \bigwedgeq^0 \big(\C^{1|1}_q\big)^* \cong \C_q\) (the trivial representation), while all other \(\bigwedgeq^a \C^{1|1}_q \otimes \bigwedgeq^a \big(\C^{1|1}_q\big)^*\) for \(a \geq 1\) are four-dimensional, indecomposable and isomorphic to each other (see \cite{SarAlexander}). The two commuting actions \eqref{eq:169} can be pictured as in figure~\ref{fig:exampleMixedHowe}. In the picture, each dot corresponds to a basis vector. The vertical arrows denote the action of the generators \(E\) and \(F\) of \(U_q(\gl_{1|1})\), while the horizontal double arrows denote the action of the generators \(E\) (solid) and \(F\) (dashed, although not the whole action is shown) of \(U_q(\gl_2)\). Of course, the actions have coefficients which are not indicated. From the picture, one sees that \eqref{eq:169}, as an \(U_q(\gl_2)\)--module, decomposes as \(M^-(1,-1) \oplus M^-(1,-1) \oplus P^-(1,-1)\), where \(M^-(1,-1)\)  denotes the (opposite) Verma module with lowest weight \((1,-1)\) and \(P^-(1,-1)\) denotes the (opposite) indecomposable projective module with head isomorphic to \(M^-(1,-1)\). In particular, one can see explicitly that the action of \(\U_q(\gl_2)\), together with the projections onto the weight spaces, gives the full centralizer of the \(U_q(\gl_{1|1})\)--action.
  \begin{figure}
    \centering
    \begin{tikzpicture}[xscale=1.8,every node/.style={font=\tiny}]
      \node[circle, fill=black,minimum size=5pt,inner sep=0,outer sep=2pt] (A52) at (5,2) {};
      \foreach \i in {1,...,4} {
        \foreach \j in {1,2,3} {
          \node[circle, fill=black,minimum size=5pt,inner sep=0,outer sep=2pt] (A\i\j) at (\i,\j) {};
        }
          \node[circle, fill=black,minimum size=5pt,inner sep=0,outer sep=2pt] (B\i) at ([xshift=0.3cm,yshift=0.3cm]\i,2) {};
          \draw[bend right=15,->] (A\i3) to node[right] {$E$} (A\i2);
          \draw[bend left=15,->] (A\i1) to node[right] {$F$} (A\i2);
          \draw[bend right=15,->] (B\i) to node[right] {$F$} (A\i3);
          \draw[bend left=15,->] (B\i) to node[right] {$E$} (A\i1);
      }
      \foreach \i in {1,...,3} {
        \pgfmathtruncatemacro{\iplusone}{\i + 1}
        \foreach \j in {1,2,3} {
          \draw[bend left=5,->,double,cross line] (A\iplusone\j) to (A\i\j);
        }
        \draw[bend left=5,->,double,cross line] (B\iplusone) to (B\i);
      }
      \draw[bend left=5,->,double,cross line] (A52) to (A42);
      \draw[bend left=5,->,double,cross line] (A12) to ++(-0.8,0);
      \draw[bend left=5,->,double,cross line] (A13) to ++(-0.8,0);
      \draw[bend left=5,->,double,cross line] (A11) to ++(-0.8,0);
      \draw[bend left=5,->,double,cross line] (B1) to ++(-0.8,0);
      \node at (0,1) {$\cdots$};
      \node at (0,2) {$\cdots$};
      \node at (0,3) {$\cdots$};
      \node at (0.3,2.3) {$\cdots$};
      \draw[bend left=5,->,double,dashed,cross line] (B4) to (A52);

      \begin{pgfonlayer}{background}
        \foreach \i in {1,...,4} {
          \draw[xshift=\i cm,gray] (-0.2,0.7) rectangle (0.5,3.3);
          \pgfmathtruncatemacro{\iminusone}{5 -\i}
          \node[font=\normalsize] at ([xshift=\i cm]0.15,0.3) {$\bigwedgeq^\iminusone  \otimes \bigwedgeq^{-\iminusone}$};
        }
          \draw[xshift=5 cm,gray] (-0.2,0.7) rectangle (0.2,3.3);
          \node[font=\normalsize] at (5,0.3) {$\C_q$};
        \draw[decorate,decoration={brace}] (5.3,3.3) -- node[right,font=\normalsize] {$M^-(1,-1)$} ++(0,-0.6);
        \node[font=\normalsize] at (5.6,2.5) {$\oplus$};
        \draw[decorate,decoration={brace}] (5.3,1.3) -- node[right,font=\normalsize] {$M^-(1,-1)$} ++(0,-0.6);
        \node[font=\normalsize] at (5.6,1.5) {$\oplus$};
        \draw[decorate,decoration={brace}] (5.3,2.3) -- node[right,font=\normalsize] {$P^-(1,-1)$} ++(0,-0.6);
      \end{pgfonlayer}
    \end{tikzpicture}
    \caption{The two commuting actions of \eqref{eq:169}.}
    \label{fig:exampleMixedHowe}
  \end{figure}
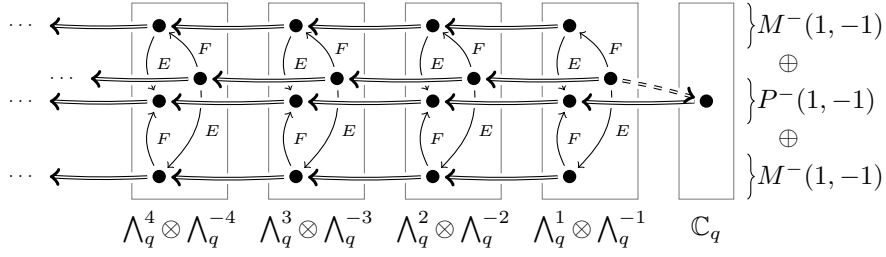
\end{example}

\subsection{Equivalence of categories}

We conclude this section  proving that the functors \(\Phi_{\boldeta_{r,s}}\) glue together giving an equivalence of categories between the direct limit of \(\U_q(\gl_{\boldeta_{r,s}})_\beta\) for \(r,s \rightarrow \infty\) and \(\Spider(\beta)\).

\begin{prop}\label{prop:17}
The functor \(\Phi_\boldeta\colon \U_q(\gl_{\boldeta})_\beta \mapto \Spider(\beta)\) is faithful.
\end{prop}

\begin{proof}
We will only sketch the idea of the proof, leaving the details, which are anyway straightforward, to the interested reader. As in the previous proof, by lemma~\ref{lem:21} we can restrict to the case 
\(\boldeta=\boldeta_{r,s}\).
 In the case \(s=0\) faithfulness can be proven as usual, by reducing to the Hecke algebra case, see \cite{CKM}. For \(s > 0\), if the target categories $U_q(\gl_{r+s})_{\boldeta_{r,s}}$ were ribbon, a standard argument would allow to reduce to the upwards case. Although this is not exactly the case, one can adapt the argument. 

Let $\phi\in \Hom_{U_q(\gl_{\boldeta_{r,s}})_\beta} (\onell{\lambda},\onell{\mu})$ and suppose \(\Phi_{\boldeta_{r,s}}(\phi)=0\).
One can define a  map \(\U_q(\gl_{r+s})_{\boldeta_{r,s}} \mapto \U_q(\gl_{r'+s})_{\boldeta_{r',s}}\) for \(r'\) big enough, given graphically by
\begin{equation}\label{eq:182}
\begin{tikzpicture}[scale=.75,smallnodes,anchorbase]
\draw (0,0) rectangle (3,2);
\node at (1.5,1) {$\phi$};
\draw[int,directed=.5] (.25,-1) -- (.25,0);
\path (.75,-1) -- node{$\cdots$} (.75,0);
\draw[int,directed=.5] (1.25,-1) -- (1.25,0);
\draw[int,<-] (1.75,-1) -- (1.75,0);
\path[int] (2.25,-1) -- node{$\cdots$} (2.25,0);
\draw[int,<-] (2.75,-1) -- (2.75,0);
\draw[int,->] (.25,2) -- (.25,3);
\path[int] (.75,2) -- node{$\cdots$} (.75,3);
\draw[int,->] (1.25,2) -- (1.25,3);
\draw[int,rdirected=.5] (1.75,2) -- (1.75,3);
\path (2.25,2) -- node {$\cdots$} (2.25,3);
\draw[int,rdirected=.5] (2.75,2) -- (2.75,3);
\end{tikzpicture}
\;\longmapsto{\phi'}=\;
\begin{tikzpicture}[xscale=.75,yscale=.5,smallnodes,anchorbase]
\draw [int,->] (-.25,0) .. controls (-.25,.5) and (-1.75,1.5) .. (-1.75,4) -- (-1.75,5);
\path [int,->] (-.75,0) node {$\cdots$} .. controls (-.75,.5) and (-2.25,1.5) .. (-2.25,4) node {$\cdots$} -- (-2.25,5) ;
\draw [int,->] (-1.25,0) .. controls (-1.25,.5) and (-2.75,1.5) .. (-2.75,4) -- (-2.75,5);
\draw [int, cross line] (-1.75,-3) -- (-1.75,-2) .. controls (-1.75,.5) and (-.25,1.5) .. (-.25,2);
\path [int]  (-2.25,-3)  --  (-2.25,-2) node {$\cdots$} .. controls (-2.25,.5) and (-.75,1.5) .. (-.75,2) node {$\cdots$};
\draw [int, cross line]  (-2.75,-3) -- (-2.75,-2) .. controls (-2.75,.5) and (-1.25,1.5) .. (-1.25,2);
\draw [int, directed=.9] (-.25,2) .. controls (-.25,2.5) and (1.25,2.5) .. (1.25,3) .. controls (1.25,3.5) and (1.75,3.5) .. (1.75,3) -- (1.75,2);
\path [int] (-.75,2) .. controls (-.75,2.75) and (1.25,2.75) .. (1.25,3.5) .. controls (1.25,4) and (2.25,4) .. (2.25,3.5) -- node{$\cdots$} (2.25,2);
\draw [int, directed=.9] (-1.25,2) .. controls (-1.25,3) and (1.25,3) .. (1.25,4) .. controls (1.25,4.5) and (2.75,4.5) .. (2.75,4) -- (2.75,2);
\draw [int] (-.25,0) .. controls (-.25,-.5) and (1.25,-.5) .. (1.25,-1) .. controls (1.25,-1.5) and (1.75,-1.5) .. (1.75,-1) -- (1.75,0);
\path [int] (-.75,0) .. controls (-.75,-.75) and (1.25,-.75) .. (1.25,-1.5) .. controls (1.25,-2) and (2.25,-2) .. (2.25,-1.5) -- node{$\cdots$} (2.25,0);
\draw [int] (-1.25,0) .. controls (-1.25,-1) and (1.25,-1) .. (1.25,-2) .. controls (1.25,-2.5) and (2.75,-2.5) .. (2.75,-2) -- (2.75,0);
\draw [int,cross line] (.25,2) -- (.25,4.75);
\draw[int,->] (.25,4.75) -- (.25,5);
\path[int,->] (.75,4) -- node {$\cdots$} (.75,5);
\draw [int, cross line] (1.25,2) -- (1.25,4.75);
\draw[int,->] (1.25,4.75) -- (1.25,5);
\draw [int, cross line] (.25,-3) -- (.25,0);
\path (.75,-3) -- node {$\cdots$} (.75,-2);
\draw [int, cross line] (1.25,-3) -- (1.25,0);
\draw (0,0) rectangle (3,2);
\node at (1.5,1) {$\phi$};
\end{tikzpicture}\;.
\end{equation}
Using an argument similar to the proof of proposition~\ref{prop:13} (cf.\ also
 \cite[Lemma 4.1 and Remark 4.2]{QS}), it follows that \({\phi'}\) is in the image of the obvious map \(\U_q(\gl_{\boldeta_{r',0}})_\beta \mapto \U_q(\gl_{\boldeta_{r',s}})_\beta\), hence it has a preimage \(\tilde{\phi}' \in \U_q(\gl_{\boldeta_{r',0}})_\beta = \U_q^{\geq 0}(\gl_{r'})\). Since its image in \(\Spider(\beta)\) is zero, we have \(\tilde\phi'=0\), hence \(\phi'=0\). Since the construction~\eqref{eq:182} can be inverted (by a similar picture) it follows that also \(\phi=0\).
\end{proof}

In particular, we can define an infinite version \(\U_q(\gl_{\boldeta_{\infty,\infty}})_\beta\) as the direct limit of \(\U_q(\gl_{\boldeta_{r,s}})_\beta\) for \(r,s \rightarrow \infty\). Propositions~\ref{prop:13} and \ref{prop:17} give then:
\begin{corollary}\label{cor:4}
  The functors \(\Phi_{\boldeta_{r,s}}\) induce  an equivalence of categories between \(\U_q(\gl_{\boldeta_{\infty,\infty}})_\beta\) and
  \(\Spider(\beta)\).
\end{corollary}

\section{Link invariants}
\label{sec:link-invariants}

In this final section, we explain how colored link invariants of type A can be interpreted inside the categories \(\Spider(\beta)\) or \(\U_g(\gl_{\boldeta_{\infty,\infty}})_\beta\).

\subsection{The $\gl_{m|n}$ link invariant and the HOMFLY-PT polynomial}
\label{sec:gl_mn-link-invariant}

By proposition~\ref{prop:9} we have an invariant of tangles \(\calQ_\beta \colon \catTangles_q^\lab \mapto \Spider(\beta) \) with values in \(\Spider(\beta)\). Moreover, for \(\beta=d=m-n\) the diagram
 \begin{equation}\label{eq:194}
   \begin{tikzpicture}[baseline=(current bounding box.center)]
     \matrix (m) [matrix of math nodes, row sep=2.5em, column
       sep=5em, text height=1.5ex, text depth=0.25ex] {
       & \Spider(d) \\
      \catTangles^\lab & \catRep_{m|n} \\};
    \path[->] (m-2-1.north) edge node[midway,above] {$\calQ_d$} (m-1-2);
    \path[->] (m-2-1) edge node[midway,below] {$\RT_{m|n}$} (m-2-2);
    \path[->] (m-1-2) edge node[midway,right] {$\calG_{m|n}$} (m-2-2);
   \end{tikzpicture}
 \end{equation}
commutes. Hence the invariant \(\calQ_d\) is universal among the \(\gl_{m|n}\)--invariants. 

Let \(T \in \Hom_{\catTangles^\lab}(\bolda,\boldb)\) be a labeled tangle. Then the lifted invariant \(\calQ_d\) assigns to it an element of \(\Hom_{\Spider(d)}(\bolda,\boldb)\). Since this space is in general bigger than \(\Hom_{\catRep_{m|n}}(\bigwedgeq^\bolda \C_q^{m|n},\bigwedgeq^\boldb \C_q^{m|n})\), the lifted invariant \(\calQ_d\) is finer than the Reshetikhin-Turaev invariant \(\RT_{m|n}\).

We are mostly interested in the special case of a link \(L\) (hence \(L \in \End_{\catTangles_q^\lab}(\varnothing)\), where \(\varnothing\) is the empty sequence). We will denote by \(P_\beta(L)\) the Reshetikhin-Turaev invariant of links \(\calQ_\beta(L)\). By corollary~\ref{cor:1}, the endomorphism space \(\End_{\Spider(\beta)}(\varnothing)\) is one-dimensional, and if \(\beta=d=m-n\) then this endomorphism space is canonically isomorphic  to \(\End_{\catRep_{m|n}}(\C_q)\). Hence in the case of links the lifted invariant \(\calQ_d\) gives the same link polynomial as the classical Reshetikhin-Turaev invariant. Since \(\calQ_d\) does not depend on \(m\) and \(n\), but only on their difference \(d=m-n\),
we deduce the following well-known observation:

\begin{prop}\label{thm:10}
  The Reshetikhin-Turaev invariant of links labeled by exterior powers of the vector representation of  \(\gl_{m|n}\) only depends on the difference \(m-n\).
\end{prop}

\begin{remark}
  This is more generally true for links labeled by partitions, see for example \cite{RemarksOnLinkInvariants}.
\end{remark}

Notice that also in the case \(\beta\) generic \(P_\beta(L)\) defines an invariant of oriented framed links, which is a polynomial in two variables \(q\) and \(q^\beta\). Then \(P_\beta(L)\) is the colored HOMFLY-PT polynomial of the link \(L\) (see \cite{QS,RemarksOnLinkInvariants}).

\subsection{Cutting strands and the Alexander polynomial}
\label{sec:cutt-strands-alex}

As usual (cf.\ for example \cite{SarAlexander}), \(P_0(L)\) is trivial and in order to obtain an interesting invariant for \(\beta=0\) one has to cut open one strand.
We hence define for \(a \in \Z_{\geq 0}\) the link invariant
\begin{equation}
  \label{eq:188}
  \tilde P_{\beta,a} \colon \Links^\lab \longrightarrow \C_q
\end{equation}
given by assigning to  \(L \in \End_{\catTangles^\lab}(\varnothing)\) the element \(\calQ_{\beta} (\tilde L) \in \End_{\Spider(\beta)}(a)\), where \(\tilde L\) is obtained by \(L\) by cutting one strand labeled by \(a\). By the following proposition~\ref{prop:18}, this does not depend on the strand cut and is therefore a well-defined link invariant.

In the case $\beta$ generic, one obtains the reduced HOMFLY-PT polynomial. The two specializations $\beta=0$ and $\beta=m$ yield respectively the Alexander and the reduced $\slm$ Reshetikhin-Turaev polynomial (see \cite{QS,RemarksOnLinkInvariants} for our normalizations of these invariants and for details).

\begin{prop}\label{prop:18}
Let \(a \in \Z_{\geq 0}\) and let \(\tau \in \End_{\Spider(\beta)}( a \otimes a)\). Then  
\begin{equation}
\begin{tikzpicture}[smallnodes,anchorbase,xscale=0.5]
\draw [int] (0,0) -- ++(0,2) node[below left,inner sep=1.5pt] {$a$} \nstartarrow \nendarrow;
\draw [int] (1,0.5) -- ++(0,1) node [above left,inner sep=1.5pt] {$a$} .. controls ++(0,0.5) and ++(0,0.5) ..   ++ (1,0) \midarrow  -- ++(0,-1) .. controls ++(0,-0.5) and ++(0,-0.5) ..   ++(-1,0) \midarrow;
\draw [fill=white]  (-.3,0.7)  rectangle (1.3,1.3);
\node at (.5,1) {\({\tau}\)};
\end{tikzpicture}\;
=\;
\begin{tikzpicture}[smallnodes,anchorbase,xscale=0.5]
\draw [int] (0,0) -- ++(0,0.1) .. controls ++(0,0.25) and ++(0,-0.25) .. ++(1,0.5) -- ++(0,0.8) node[yshift=0.1cm,right,inner sep=1.5pt] {$a$} .. controls ++(0,0.25) and ++(0,-0.25) .. ++(-1,0.5)  -- ++(0,0.1)  \nendarrow;
\draw [int, cross line] (1,0.1)         ++(-1,0.5) --
          ++(0,0.8) node [yshift=0.1cm,left,inner sep=1.5pt] {$a$} .. controls ++(0,0.5) and ++(0,0.4) ..  
          ++ (2,0.1)  \midarrow  -- 
          ++(0,-1) .. controls ++(0,-0.4) and ++(0,-0.5) ..   ++(-2,0.1) \midarrow;
\draw [fill=white]  (-.3,0.7)  rectangle (1.3,1.3);
\node at (.5,1) {\({\tau}\)};
\end{tikzpicture}\;.
\end{equation}
\end{prop}

\begin{proof}
 It is enough to check the statement for \(\tau\) a basis element. Since \(\End_{\Spider(\beta)}(a \otimes a) \cong \End_{\Spider^+}(a \otimes a)\),
it is easy to check that the elements \(\{E^{(k)}F^{(k)} \suchthat 0\leq k\leq a\}\) give a basis. We have
\begin{equation}
\begin{tikzpicture}[smallnodes,anchorbase,xscale=0.5]
\draw [int] (0,0) -- ++(0,0.1) .. controls ++(0,0.25) and ++(0,-0.25) .. ++(1,0.5) -- ++(0,0.8) node[yshift=0.1cm,right,inner sep=1.5pt] {$a$} .. controls ++(0,0.25) and ++(0,-0.25) .. ++(-1,0.5)  -- ++(0,0.1)  \nendarrow;
\draw [int, cross line] (1,0.1)         ++(-1,0.5) --
          ++(0,0.8) node [yshift=0.1cm,left,inner sep=1.5pt] {$a$} .. controls ++(0,0.5) and ++(0,0.4) ..  
          ++ (2,0.1)  \midarrow  -- 
          ++(0,-1) .. controls ++(0,-0.4) and ++(0,-0.5) ..   ++(-2,0.1) \midarrow;
\draw [int] (0,1.1) -- node[above] {$k$} ++(1,0.2) \midarrow;
\draw [int] (1,0.7) -- node[below] {$k$} ++(-1,0.2) \midarrow;
\end{tikzpicture}
\;=\;
\begin{tikzpicture}[smallnodes,anchorbase,xscale=0.5]
\draw [int] (0,0) -- ++(0,0.5) .. controls ++(0,0.25) and ++(0,-0.25) .. ++(1,0.5) -- ++(0,0)  .. controls ++(0,0.25) and ++(0,-0.25) .. ++(-1,0.5)  -- ++(0,0.4) -- ++(0,0.1)  \nendarrow;
\draw [int, cross line] (1,0.1)  ++(-1,0.5) ++(0,0.4) --   ++(0,0)  .. controls ++(0,0.25) and ++(0,-0.25) .. ++ (1,0.5) .. controls ++(0,0.5) and ++(0,0.5) .. ++(1,0) -- ++(0,-1) \midarrow .. controls ++(0,-0.5) and ++(0,-0.5) .. ++(-1,0)  .. controls ++(0,+0.25) and ++(0,-0.25) ..   ++(-1,0.5);
\draw [int] (0,0.4) -- node[below] {$k$} ++(1,0.2) \midarrow;
\draw [int] (1,1.6) -- node[above] {$k$} ++(-1,0.2) \midarrow;
\end{tikzpicture}
\;=\;
\begin{tikzpicture}[smallnodes,anchorbase,xscale=0.5]
\draw [int] (0,0) -- ++(0,0.5) -- ++(0,0.5) -- ++(0,0)  -- ++(0,0.5)  -- ++(0,0.4) -- ++(0,0.1)  \nendarrow;
\draw [int, cross line] (1,0.1)  ++(0,0.5) ++(0,0.4) --   ++(0,0)  --  ++ (0,0.5) .. controls ++(0,0.5) and ++(0,0.5) .. ++(1,0) -- ++(0,-1) \midarrow .. controls ++(0,-0.5) and ++(0,-0.5) .. ++(-1,0)  --   ++(0,0.5);
\draw [int] (0,0.4) -- node[below] {$k$} ++(1,0.2) \midarrow;
\draw [int] (1,1.6) -- node[above] {$k$} ++(-1,0.2) \midarrow;
\end{tikzpicture} 
\;=\;
\begin{tikzpicture}[smallnodes,anchorbase,xscale=0.5]
\draw [int] (0,0) -- ++(0,0.5) -- ++(0,0.5) -- ++(0,0)  -- ++(0,0.5)  -- ++(0,0.4) -- ++(0,0.1)  \nendarrow;
\draw [int, cross line] (1,0.1)  ++(0,0.5) ++(0,0.4) --   ++(0,0)  --  ++ (0,0.5) .. controls ++(0,0.5) and ++(0,0.5) .. ++(1,0) -- ++(0,-1) \midarrow .. controls ++(0,-0.5) and ++(0,-0.5) .. ++(-1,0)  --   ++(0,0.5);
\draw [int] (0,1.4) -- node[above] {$k$} ++(1,0.2) \midarrow;
\draw [int] (1,0.6) -- node[below] {$k$} ++(-1,0.2) \midarrow;
\end{tikzpicture}\;,
\end{equation}
\noindent where the first equality is a consequence of relations \eqref{eq:EqW1} and \eqref{eq:EqW2}
and the third one follows by an iterated use of \eqref{eq:77}.
\end{proof}

In light of the inclusion of categories \(\Phi_{\boldeta}\), it is possible to regard all these invariants of tangles also as having values in \(\U_q(\gl_{\boldeta})_\beta\) for some sequence of signs \(\boldeta\) long enough. For a more detailed description of link invariants with values in \(\U_q(\gl_{\boldeta})_\beta\) we refer to \cite{QS}. We point out that working in \(\U_q(\gl_{\boldeta})_\beta\)  is particularly handy for computation purposes, since \(\U_q(\gl_{\boldeta})_\beta\) has a more rigid structure than \(\Spider(\beta)\). Moreover, we believe this can give more insight for developing a categorification.

\begin{remark}\label{rem:6}
  In this paper, we studied \emph{skew} Howe duality, generalizing \cite{CKM}, hence we obtain the category \(\Spider(d)\) which describes intertwining operators between exterior powers of the natural representation of \(U_q(\gl_{m|n})\). In this setting, the easiest case where we can explicitly describe all relations is when \(n=0\) and we consider the category \(\catRep_{m|0}\) (see corollary~\ref{cor:2}).

However, one could also be interested in replacing \emph{skew} Howe duality by \emph{symmetric} Howe duality, and define analogous (and similar) categories \(\sSpider^+(\beta)\) of ``symmetric'' spiders, which would describe intertwining operators between symmetric powers of \(\C_q^{m|n}\) (see also \cite{RTub}).

It is interesting to notice that symmetric powers  actually appear already in our picture: indeed, \(\bigwedgeq^a(\C_q^{m|n})\) is isomorphic to \((\bigsymmq^a \C_q^{n|m})\langle a\rangle\), where \(\langle a\rangle\) denotes a shift by \(a\) in the \(\Z/2\Z\)--degree. 
This suggests that there is a strong similarity between the ``skew categories'' \(\Spider^+\), \(\Spider(\beta)\) and the ``symmetric categories'' \(\sSpider^+\), \(\sSpider(\beta)\). Indeed, one can fix the conventions so that \(\Spider^+\) and \(\sSpider^+\) are the same. However, the parity shift above introduces some sign differences  between \(\Spider(\beta)\) and \(\sSpider(\beta)\) which prevent an isomorphism on the nose. 
\end{remark}

\newcommand{\etalchar}[1]{$^{#1}$}
\providecommand{\bysame}{\leavevmode\hbox to3em{\hrulefill}\thinspace}
\providecommand{\MR}{\relax\ifhmode\unskip\space\fi MR }
\providecommand{\MRhref}[2]{%
  \href{http://www.ams.org/mathscinet-getitem?mr=#1}{#2}
}
\providecommand{\href}[2]{#2}

\end{document}